\documentclass[reqno,a4paper]{amsart}
\usepackage{amsmath}
\usepackage{amssymb}
\usepackage{mathrsfs}
\usepackage{bm}
\usepackage{amsfonts}
\usepackage{amsthm}
\usepackage[colorlinks,linkcolor=blue,anchorcolor=red,citecolor=blue]{hyperref}
\usepackage{geometry}
\usepackage{fullpage}
\newcommand{\qe}{}
\numberwithin{equation}{section}
\newtheorem{Thm}{Theorem}[section]{\bf}{\it}
{\bf}{\it}
\newtheorem{Lem}[Thm]{Lemma}{\bf}{\it}
{\bf}{\it}
{\bf}{\it}
{\bf}{\it}

\newcommand{\1}{\mathbf{1}}
\newcommand{\R}{\mathbb{R}}

\newcommand{\Z}{\mathbb{Z}}
\renewcommand{\P}{\mathbf{P}}
\renewcommand{\Re}{\text{Re}}

\newcommand{\F}{\mathscr{F}}

\newcommand{\C}{\mathbb{C}}
\newcommand{\E}{\mathcal{E}}
\newcommand{\D}{\mathcal{D}}
\newcommand{\ve}{\varepsilon}
\newcommand{\vt}{\vartheta}
\newcommand{\ga}{\gamma}

\newcommand{\pa}{\partial}
\newcommand{\na}{\nabla}

\newcommand{\wh}{\widehat}
\allowdisplaybreaks
\renewcommand{\Re}{\text{Re}}

\newcommand{\<}{\langle}
\renewcommand{\>}{\rangle}
\newcommand{\T}{\mathbb{T}}
\newcommand{\I}{\mathbf{I}}
\newcommand{\II}{\mathbf{I}_{\pm}}
\renewcommand{\P}{\mathbf{P}}
\newcommand{\PP}{\mathbf{P}_{\pm}}

\newcommand{\vertiii}[1]{{\left\vert\kern-0.25ex\left\vert\kern-0.25ex\left\vert #1 \right\vert\kern-0.25ex\right\vert\kern-0.25ex\right\vert}}

\usepackage{cite}

\begin{document}
	
	\title[Vlasov-Poisson-Landau/Boltzmann system on torus and finite channel]{Low Regularity Solutions for the Vlasov-Poisson-Landau/Boltzmann System}
	
	\author[D.-Q. Deng]{Dingqun Deng}
	\address[D.-Q. Deng]{Beijing Institute of Mathematical Sciences and Applications and Yau Mathematical Science Center, Tsinghua Univeristy, Beijing, People's Republic of China}
	\email{dingqun.deng@gmail.com}
	
	\author[R.-J. Duan]{Renjun Duan}
	\address[R.-J. Duan]{Department of Mathematics, The Chinese University of Hong Kong, Shatin, Hong Kong,
		People's Republic of China}
	\email{rjduan@math.cuhk.edu.hk}
	
	\begin{abstract}
		In the paper, we are concerned with the nonlinear Cauchy problem on the Vlasov-Poisson-Landau/Boltzmann system around global Maxwellians in torus or finite channel. The main goal is to establish the global existence and large time behavior of small amplitude solutions for a class of low regularity initial data. The molecular interaction type is restricted to the case of hard potentials for two classical collision operators because of the effect of the self-consistent forces. The result extends the one by Duan-Liu-Sakamoto-Strain [{\it Comm. Pure Appl. Math.} 74 (2021), no.~5, 932--1020] for the pure Landau/Boltzmann equation to the case of the VPL and VPB systems.  
	\end{abstract}

	\date{\today}
	\maketitle	
	\tableofcontents	
	
	\thispagestyle{empty}

	\section{Introduction}

	\subsection{Equations}
	
	We consider the Vlasov-Poisson-Boltzmann (VPB) and Vlasov-Poisson-Landau (VPL) systems describing the motion of plasma particles of two species (e.g. ions and electrons) in a domain $\Omega \subset \R^3$:
	\begin{equation}\label{F1}\left\{
		\begin{aligned}
			&\pa_tF_+ + v\cdot \na_xF_+ - \na_x\phi\cdot\na_vF_+ = Q(F_+,F_+)+Q(F_-,F_+),\\
			&\pa_tF_- + v\cdot \na_xF_- + \na_x\phi\cdot\na_vF_- = Q(F_+,F_-)+Q(F_-,F_-),\\
			&-\Delta_x\phi = \int_{\R^3}(F_+-F_-)\,dv,\\ 
			& F_\pm(0,x,v)=F_{0,\pm}(x,v). %,\quad E(0,x)=E_0(x).
		\end{aligned}\right. 
	\end{equation}
	Here, the unknowns $F_\pm(t,x,v)\ge 0$ stand for the velocity distribution functions for the particles of ions $(+)$ and electrons $(-)$, respectively, at position $x\in\Omega$ and velocity $v\in\R^3$ and time $t\ge 0$. For later use we introduce the self-consistent electrostatic field taking the form 
\begin{equation*}
%\label{ }
E(t,x)= -\na_x\phi(t,x),
\end{equation*}
where $\phi(t,x)$ is the potential function given %as $\phi=\frac{1}{4\pi |x|}\ast \rho$ with $\rho(t,x)=\int_{\R^3} (F_+-F_-)\,dv$ 
in terms of the Poisson equation.

	For the case of VPL system, the collision operator $Q$ is given by 
	\begin{align}\label{VPL}\notag
		Q(G,F)&=\nabla_v\cdot\int_{\R^3}\phi(v-v')\big[G(v')\nabla_vF(v)-F(v)\nabla_vG(v')\big]\,dv'\\
		&=\sum^3_{i,j=1}\partial_{v_i}\int_{\R^3}\phi^{ij}(v-v')\big[G(v')\partial_{v_j}F(v)-F(v)\partial_{v_j}G(v')\big]\,dv'.
	\end{align}
	The non-negative definite matrix-valued function $\phi=[\phi^{ij}(v)]_{1\leq i,j\leq 3}$ takes the form of 
	\begin{align*}
		\phi^{ij}(v) = \Big\{\delta_{ij}-\frac{v_iv_j}{|v|^2}\Big\}|v|^{\gamma+2},
	\end{align*}
	with $\gamma\ge -3$. It is convenient to call it {\em hard potential} when $\gamma\ge -2$ and {\em soft potential} when $-3\le\gamma<-2$. The case $\gamma=-3$ corresponds to the physically realistic Coulomb interactions; cf. \cite{Guo2002a}.

	For the case of VPB system, the collision operator $Q$ is defined by 
	\begin{align}
		\label{VPB}
		Q(G,F) = \int_{\R^3}\int_{\mathbb{S}^{2}} B(v-v_*,\sigma)\big[G(v'_*)F(v')-G(v_*)F(v)\big]\,d\sigma dv_*.
	\end{align}  
	Here, $v,v_*$ and $v',v'_*$ are velocity pairs given in terms of the $\sigma$-representation by
	\begin{align*}%\label{collisiongeo}
		v'=\frac{v+v_*}{2}+\frac{|v-v_*|}{2}\sigma,\quad v'_*=\frac{v+v_*}{2}-\frac{|v-v_*|}{2}\sigma,\quad \sigma\in\mathbb{S}^2,
	\end{align*}
	that satisfy 
	$v+v_*=v'+v'_*$ and 
	$|v|^2+|v_*|^2=|v'|^2+|v'_*|^2$.
	The Boltzmann collision kernel $B(v-v_*,\sigma)$ depends only on $|v-v_*|$ and the deviation angle $\theta$ through $\cos\theta=\frac{v-v_*}{|v-v_*|}\cdot\sigma$. Without loss of generality we can assume $B(v-v_*,\sigma)$ is supported on $0\le\theta\le\pi/2$, since one can reduce the situation with {\it symmetrization}: $\overline{B}(v-v_*,\sigma)={B}(v-v_*,\sigma)+{B}(v-v_*,-\sigma)$. Moreover, 
	$B(v-v_*,\sigma) = |v-v_*|^\gamma b(\cos\theta)$,
	and we assume that there exist $C_b>0$ and $0<s<1$ such that  
	\begin{align*}
		\frac{1}{C_b\theta^{1+2s}}\le \sin\theta b(\cos\theta)\le \frac{C_b}{\theta^{1+2s}}, \quad\forall\,\theta\in (0,\frac{\pi}{2}].
	\end{align*} 
	%We further assume $\gamma+2s\ge 1$ and $1/2\le s<1$. 
	It is convenient call it {\em hard potential} when $\gamma+2s\ge 0$ and {\em soft potential} when $-3<\gamma+2s<0$.

	%There exists $C_b>0$ such that 
	%\begin{equation}\label{bcos}
	%	\frac{1}{C_b}\theta^{1+2s}\le \sin\theta\,b(\cos\theta)\le \frac{C_b}{\theta^{1+2s}},\ s\in (0,1), \forall\,\theta\in\Big(0,\frac{\pi}{2}\Big].
	%\end{equation}
	
	\subsection{Reformulation}
We will reformulate the problem \eqref{F1} near a global Maxwellian. For this, consider the global Maxwellian equilibrium state:
	\begin{equation*}
		\mu = \mu(v) = (2\pi)^{-3/2}e^{-|v|^2/2}. 
	\end{equation*}
	We will look for a solution to \eqref{F1} of the form 
	\begin{equation*}
		F(t,x,v) = \mu + \mu^{1/2}f(t,x,v).
	\end{equation*}
	Then $f=f(t,x,v)$ satisfies 
	\begin{equation}\label{1}\left\{
		\begin{aligned}
			&\partial_tf_\pm + v\cdot\nabla_xf_\pm \pm \frac{1}{2}\nabla_x\phi\cdot vf_\pm  \mp\nabla_x\phi\cdot\nabla_vf_\pm \pm \nabla_x\phi\cdot v\mu^{1/2} - L_\pm f = \Gamma_{\pm}(f,f),\\
			&
			-\Delta_x\phi = \int_{\R^3}(f_+-f_-)\mu^{1/2}\,dv,\\
			&f(0,x,v) = f_0(x,v),%\quad E(0,x)=E_0(x), 
		\end{aligned}\right.
	\end{equation}
	where the linearized collision operator $L=[L_+,L_-]$ and nonlinear collision operator $\Gamma=[\Gamma_+,\Gamma_-]$ are given respectively by 
	\begin{equation*}
		L_\pm f = \mu^{-1/2}\Big\{2Q(\mu,\mu^{1/2}f_\pm)+Q(\mu^{1/2}(f_\pm+f_\mp),\mu)\Big\},
	\end{equation*}
	and 
	\begin{equation*}
		\Gamma_\pm(f,g) =  \mu^{-1/2}\Big\{Q(\mu^{1/2}f_\pm,\mu^{1/2}g_\pm)+Q(\mu^{1/2}f_\mp,\mu^{1/2}g_\pm)\Big\}.
	\end{equation*}
	The kernel of $L$ on $L^2_v\times L^2_v$ is the span of 
	$$
	\{[1,0]\mu^{1/2},[0,1]\mu^{1/2},[1,1]v\mu^{1/2},[1,1]|v|^2\mu^{1/2}\}
	$$ 
	and we define the projection of $L^2_v\times L^2_v$ onto $\ker L$ to be 
	\begin{equation*}
		\P f = \Big(a_+(t,x)[1,0]+a_-(t,x)[0,1]+v\cdot b(t,x)[1,1]+(|v|^2-3)c(t,x)[1,1]\Big)\mu^{1/2},
	\end{equation*}or equivalently by 
	\begin{equation*}
		\PP f = \Big(a_\pm(t,x)+v\cdot b(t,x)+(|v|^2-3)c(t,x)\Big)\mu^{1/2}.
	\end{equation*}
	Then for given $f$, one can decompose $f$ uniquely as 
	\begin{equation*}
		f = \P f+ (\I-\P)f. 
	\end{equation*}
	The function $a_\pm,b,c$ are given by 
	\begin{equation*}%\label{abc}
		\left\{\begin{aligned}
			a_\pm &= (\mu^{1/2},f_\pm)_{L^2_v},\\
			b_j&= \frac{1}{2}(v_j\mu^{1/2},f_++f_-)_{L^2_v},\\
			c&=\frac{1}{12}((|v|^2-3)\mu^{1/2},f_++f_-)_{L^2_v}. 
		\end{aligned}\right.
	\end{equation*}

	\subsection{Spatial domain and boundary condition}
	In this paper, we focus on two kinds of specific domains $\Omega\subset\R^3$, either torus or finite channel. In what follows we give their definitions and the corresponding conservation laws to equation \eqref{1}. 
	
	\subsubsection{Case I: Torus} In this case, we let 
	\begin{equation*}
		\Omega=\T^3 = [-\pi,\pi]^3.
	\end{equation*}
	Correspondingly, $F(t,x,v)$ is assumed to be spatially periodic in $x$. We also assume $\int_{\T^3}\phi(t,x)dx=0$ for any $t\geq 0$. 
	It's well known that if the following   
	%the initial data $(f(0),E(0))$ satisfies 
%	\begin{equation}\label{4a}\left\{\begin{aligned}
%			&\int_{\T^3\times\R^3}f_+(0)\mu^{1/2}\,dvdx = \int_{\T^3\times\R^3}f_-(0)\mu^{1/2}\,dvdx =0,\\
%			&\int_{\T^3\times\R^3}(f_+(0)+f_-(0))v\mu^{1/2}\,dvdx =0,\\
%			&\int_{\T^3\times\R^3}(f_+(0),f_-(0))|v|^2\mu^{1/2}\,dvdx +\int_{\T^3}|E(0)|^2\,dx=0,
%		\end{aligned}\right.
%	\end{equation}
	\begin{equation}\label{4}\left\{\begin{aligned}
			&\int_{\T^3}\int_{\R^3}f_+(t)\mu^{1/2}\,dvdx = \int_{\T^3}\int_{\R^3}f_-(t)\mu^{1/2}\,dvdx =0,\\
			&\int_{\T^3}\int_{\R^3}(f_+(t)+f_-(t))v\mu^{1/2}\,dvdx =0,\\
			&\int_{\T^3}\int_{\R^3}(f_+(t)+f_-(t))|v|^2\mu^{1/2}\,dvdx +\int_{\T^3}|E(t)|^2\,dx=0,
		\end{aligned}\right.
	\end{equation}
holds initially at $t=0$, then we have conservation laws \eqref{4} for the solution $f(t,x,v)$ for any $t\ge 0$.

	\subsubsection{Case II: Finite channel} In this case, we set 
	\begin{equation*}
		\Omega=[-1,1]\times\T^2=\{x=(x_1,\bar{x}):\ x_1\in[-1,1],\ \bar{x}:=(x_2,x_3)\in \T^2=[-\pi,\pi]^2\}.
	\end{equation*}
	Correspondingly, $f(t,x,v)$ is assumed to be spatially periodic for $\bar{x}$ and satisfies the following {\it specular reflection} boundary condition at $x_1=\pm 1$:
	\begin{equation}\label{specular_finite}\begin{aligned}
			{f}(t,-1,\bar{x},v_1,\bar{v})|_{v_1>0} &= {f}(t,-1,\bar{x},-v_1,\bar{v}),\\
			{f}(t,1,\bar{x},v_1,\bar{v})|_{v_1<0} &= {f}(t,1,\bar{x},-v_1,\bar{v}),
		\end{aligned}
	\end{equation}
	where $v=(v_1,\bar{v})\in\R^3$. For the case of finite channel, we further assume 
	\begin{equation}
		\label{Neumann_finite}
		\partial_{x_1}\phi = 0,\ \text{ on } x_1=\pm 1,
	\end{equation}
	for Poisson equation to $\phi$. 
	Due to conservation laws similar to the torus case, we assume that $f(t,x,v)$  satisfies 
	\begin{equation}\label{conservation_finite}\left\{\begin{aligned}
&\int_{[-1,1]\times\T^2}\int_{\R^3}f_+(t)\mu^{1/2}\,dvdx = \int_{[-1,1]\times\T^2}\int_{\R^3}f_-(t)\mu^{1/2}\,dvdx =0,\\			&\int_{[-1,1]\times\T^2}\int_{\R^3}(f_+(t)+f_-(t))v_i\mu^{1/2}\,dvdx \equiv0,\quad i=2,3,\\
		&\int_{[-1,1]\times\T^2}\int_{\R^3}(f_+(t)+f_-(t))|v|^2\mu^{1/2}\,dvdx +\int_{[-1,1]\times\T^2}|E(t)|^2\,dx=0,
		\end{aligned}\right.
	\end{equation}
with all $t\ge 0$.
	%Note that we couldn't have conservation laws for momentum when $i=1$. 
	
	\medskip
	
	This work concerns the low regularity global solution and large time asymptotic behavior for the VPB and VPL systems in torus or finite channel as above. 
	The high regularity solution for Boltzmann and Landau equations near a global Maxwellian has been well-studied in the last few decades. For the cutoff Boltzmann case, Caflisch \cite{Caflisch1980,Caflisch1980a} and Ukai-Asano \cite{Ukai1982} gave the global solution to Cauchy problem in torus. For the non-cutoff Boltzmann case, we refer to the work AMUXY \cite{Alexandre2012} and Gressman-Strain \cite{Gressman2011}. For Landau case, we refer to Guo \cite{Guo2002a}. If the Boltzmann and Landau equations are combined with the self-consistent electrical or electromagnetic fields, that is, the Poisson and Maxwell equations, we refer to \cite{Mischler2000,Guo2002,Guo2012,Strain2013,Guo2003a,Duan2011,Duan2013,Wang2015,Strain2006}. 
	For the low regularity solution, one may go back to the classic work by DiPerna-Lions \cite{Diperna1989}, which constructed the renormalized solution by using weak compactness method in $L^1$ framework. For the framework near a global Maxwellian, Duan-Huang-Wang-Yang \cite{Duan2017a} gave the global well-posedness of cutoff Boltzmann equation with a class of large amplitude data and Duan-Wang \cite{Duan2019b} generalized it to the case of general bounded domains with diffusive reflection boundaries. Recently, for the non-cutoff Boltzmann and Landau equations, Duan-Liu-Sakamoto-Strain \cite{Duan2020} studied the global mild solutions for small-amplitude initial data in space $L^1_kL^2_v$ with very low regularity. In this work, we expect to generalize this result to the case of the VPB or VPL system. Also, we consider the VPB and VPL system in finite channel with the physically important specular-reflection boundary condition. 
	
	For the boundary value theory of collisional kinetic problems such as on Landau and Boltzmann equations, we refer to \cite{Cercignani1992,Hamdache1992, Mischler2000,Yang2005, Liu2006,Guo2009, Esposito2013, Guo2016,Kim2017,Cao2019, Guo2020,Dong2020}. 
	Since the fundamental work by Guo \cite{Guo2009} using an $L^2-L^\infty$ method, many results have been developed for Boltzmann equation and Landau equation. For instance, Guo-Kim-Tonon-Trescases \cite{Guo2016} gave regularity of cutoff Boltzmann equation with several physical boundary conditions. Esposito-Guo-Kim-Marra \cite{Esposito2013} constructed a non-equilibrium stationary solution and studied the exponential asymptotic stability. Kim-Lee \cite{Kim2017} studied cutoff Boltzmann equation with specular boundary condition with external potential. Liu-Yang \cite{Liu2016} extended the result in \cite{Guo2009} to cutoff soft potential case. Cao-Kim-Lee \cite{Cao2019} proved the global existence for Vlasov-Poisson-Boltzmann with diffuse boundary condition. Guo-Hwang-Jang-Ouyang \cite{Guo2020} gave the global stability of Landau equation with specular reflection boundary. Duan-Liu-Sakamoto-Strain \cite{Duan2020} proved the global existence for Landau and non-cutoff Boltzmann equation in finite channel. Dong-Guo-Ouyang \cite{Dong2020} established the global existence for VPL system in general bounded domain with specular boundary condition. 
	
	In this work, we would give the global well-posedness of VPB and VPL systems in torus and finite channel in a function space with very low regularity, related to the Weiner space $A(\T^3)=L^1_k$  given in \eqref{L1k}. We shall establish the global existence and exponential time decay for such low regularity initial data in the case of hard potentials.

	\subsection{Notations}
	We first give some notations throughout the paper. 
	Let  $I$ be the identity mapping. Set $\<v\>=\sqrt{1+|v|^2}$. $\1_{S}$ is the indicator function on a set $S$. $(\cdot|\cdot)$ denotes the inner product in $\C$. Let
	$\partial^\alpha_\beta = \partial^{\alpha_1}_{x_1}\partial^{\alpha_2}_{x_2}\partial^{\alpha_3}_{x_3}\partial^{\beta_1}_{v_1}\partial^{\beta_2}_{v_2}\partial^{\beta_3}_{v_3}$,
	where $\alpha=(\alpha_1,\alpha_2,\alpha_3)$ and $\beta=(\beta_1,\beta_2,\beta_3)$ are multi-indices. If each component of $\beta'$ is not greater than that of $\beta$'s, we denote by $\beta'\le\beta$.
	The notation $a\approx b$ (resp. $a\gtrsim b$, $a\lesssim b$) for positive real function $a$ and $b$ means that there exists $C>0$ not depending on possible free parameters such that $C^{-1}a\le b\le Ca$ (resp. $a\ge C^{-1}b$, $a\le Cb$) on their domain.
	We will write $C>0$ (large) to be a generic constant, which may change from line to line. Denote the $L^2_v$ and $L^2_{x,v}$, respectively, as 
	\begin{align*}
		|f|^2_{L^2_v} = \int_{\R^3}|f|^2\,dv,\quad \|f\|_{L^2_{x,v}}^2 = \int_{\Omega}|f|^2_{L^2_v}\,dx.
	\end{align*}
	For any $m\ge 0$, we denote spaces $L^1_k$, $L^1_{\bar{k}}$ and their weighted forms $L^1_{k,m}$, $L^1_{\bar{k},m}$ respectively by 
	\begin{align}\label{L1k}
		\|\wh{f}\|^2_{L^1_{k}} = \int_{\Z^3}|\widehat{f}(k)|\,dk,& \qquad
		\|\widehat{f}\|_{L^1_{\bar{k}}}=\int_{\Z^2}|\widehat{f}(k)|\,d\bar{k},\\
		\|\wh{f}\|^2_{L^1_{k,m}} = \int_{\Z^3}\<k\>^m|\widehat{f}(k)|\,dk,& \qquad
		\|\widehat{f}\|_{L^1_{\bar{k},m}}=\int_{\Z^2}\<\bar{k}\>^m|\widehat{f}(k)|\,d\bar{k}. \notag
	\end{align}

	To capture the dissipation rate, for Landau case, we denote
	\begin{align*}
		\sigma^{ij}(v) &= \phi^{ij}*\mu = \int_{\R^3}\phi^{ij}(v-v')\mu(v')\,dv',\\
		\sigma^i(v)&=\sum_{j=1}^3\sigma^{ij}\frac{v_j}{2}=\sum_{j=1}^3\phi^{ij}*\big[\frac{v_j}{2}\mu\big].
	\end{align*}
%Here and after repeated indices are implicitly summed over. 
	Define
	\begin{align*}
		|f|^2_{L^2_{D,w}}=\sum^3_{i,j=1}\int_{\R^3}w^2\big(\sigma^{ij}\partial_{v_i}f\partial_{v_j}f+\sigma^{ij}\frac{v_i}{2}\frac{v_j}{2}|f|^2\big)dv,\quad \|f\|^2_{L^2_xL^2_{D,w}}=\int_\Omega|f|_{L^2_{D,w}}^2\,dx,
	\end{align*}
	and $|\widehat{f}(k)|^2_{L^2_{D}}=|\widehat{f}(k)|^2_{L^2_{D,1}}$.
	Then from \cite[Corollary 1, p.399]{Guo2002a}, we have 
	\begin{align*}%\label{dissiaption}
	%\notag
		%	|\widehat{f}(k)|^2_{L^2_{D,w}} &= \sum^3_{i=1}|w\<v\>^{\frac{\gamma}{2}}P_v(\partial_{v_j}f)^\wedge(k)|_{L^2_v}^2 + \sum^3_{i=1}|w\<v\>^{\frac{\gamma+2}{2}}\{I-P_v\}(\partial_{v_j}f)^\wedge(k)|_{L^2_v}^2 \\&\qquad\qquad+ |w\<v\>^{\frac{\gamma+2}{2}}\widehat{f}(k)|_{L^2_v}^2,
		|{f}|^2_{L^2_{D,w}} &= |w\<v\>^{\frac{\gamma}{2}}P_v\nabla_v f|_{L^2_v}^2 + |w\<v\>^{\frac{\gamma+2}{2}}\{I-P_v\}\nabla_vf|_{L^2_v}^2 + |w\<v\>^{\frac{\gamma+2}{2}}{f}|_{L^2_v}^2,
	\end{align*}
where $P_v$ is the projection along the direction of $v$.

	For Boltzmann case, as in \cite{Gressman2011}, we denote
	\begin{equation*}
		|f|^2_{L^2_D}:=|\<v\>^{\frac{\gamma+2s}{2}}f|^2_{L^2_v}+ \int_{\R^3}dv\,\<v\>^{\gamma+2s+1}\int_{\R^3}dv'\,\frac{(f'-f)^2}{d(v,v')^{3+2s}}\1_{d(v,v')\le 1},
	\end{equation*}
	and 
	\begin{equation*}
		|f|^2_{L^2_{D,w}}=|wf|^2_{L^2_D},\quad \|f\|^2_{L^2_xL^2_{D,w}} = \int_{\Omega}|f|_{L^2_{D,w}}^2\,dx.
	\end{equation*}
	The fractional differentiation effect is measured using the anisotropic metric on the {\it lifted} paraboloid
	$d(v,v'):=\{|v-v'|^2+\frac{1}{4}(|v|^2-|v'|^2)^2\}^{1/2}$.
	%
	%We will consider norms from \cite{Alexandre2012} and \cite{Gressman2011} as 
	%\begin{align*}
	%	|f|_{L^2_D}^2:&=\int B(v-v_*,\sigma)\Big(\mu_*(f'-f)^2+f^2_*((\mu')^{1/2}-\mu^{1/2})^2\Big)\,d\sigma dv_*dv,\\
	%	|f|^2_{N^{s,\gamma}}:&=\|\<v\>^{\gamma/2+s}f\|^2_{L^2}+\int(\<v\>\<v'\>)^{\frac{\gamma+2s+1}{2}}\frac{(f'-f)^2}{d(v,v')^{d+2s}}\1_{d(v,v')\le 1},
	%\end{align*}with $d(v,v'):=\sqrt{|v-v'|^2+\frac{1}{4}(|v|^2-|v'|^2)^2}$.
	Then by \cite[Proposition 2.2]{Alexandre2012}, we have 
	\begin{align*}
		|\<v\>^{\frac{\gamma}{2}}\<D_v\>^sf|^2_{L^2_v}+|\<v\>^{\frac{\gamma+2s}{2}}f|^2_{L^2_v}\lesssim |f|^2_{L^2_D}\lesssim |\<v\>^{\frac{\gamma+2s}{2}}\<D_v\>^sf|^2_{L^2_v}.
	\end{align*}
	%Then by (2.13)(2.15) in \cite{Gressman2011} and Proposition 2.1 in \cite{Alexandre2012} for $f\in\S$, $l\in\R$, we have the equivalence of norms:
	%\begin{align}\label{equivalent_norm}
	%	|f|_{L^2_D}\approx|f|^2_{N^{s,\gamma}}\approx (-Lf,f)_{L^2_v}+\|\<v\>^lf\|_{L^2_v},
	%\end{align}where the constants depend on $l$. So, we denote 
	%\begin{align*}
	%	|f|_{L^2_D} := \vertiii{f}^2\approx|f|^2_{N^{s,\gamma}}
	%\end{align*}
	
	\subsection{Function space} 
	
	To study the global well-posedness of problem \eqref{1} in different domains, we will consider the following function spaces and energy functionals. We will consider 
	\begin{equation*}\begin{aligned}
			&\gamma\ge -2, \qquad\qquad\qquad\quad\ \text{ for VPL case},\\&\gamma+2s\ge 1,\ \frac{1}{2}\le s<1, \quad\text{ for VPB case}.
		\end{aligned}
	\end{equation*} We will make use of weight function:
	\begin{align}
		w &=w(t,v):= \exp\Big({\frac{q\<v\>^\vt}{(1+t)^N}}\Big),\label{w2}
	\end{align}
	where $q\ge 0$ and we restrict $0<q<1/8$ when $\vt=2$. We let $\vt = -\gamma$ for VPL case when $-2\le \gamma<-1$ and $q = 0$ when $\gamma\ge -1$ for VPL case. For the VPB case, we let $q=0$ when $\gamma+2s\ge 1$ and $\frac{1}{2}\le s<1$. Therefore, such weight $w$ in \eqref{w2} will be necessarily included only for the VPL case with $-2\leq \ga<-1$ giving that $\vt\in[1,2]$.  
	
%	\medskip\noindent{\bf Case I: Torus.}
	\subsubsection{Case I: Torus}
	For the torus case, to derive the low regularity solution, as in \cite{Duan2020}, we will consider the Weiner space $L^1_k=A(\T^3)$, where $k$ is the Fourier variable with respect to $x$. In order to close the energy estimate with time integral, we consider   
	\begin{equation*}
		L^1_kL^\infty_TL^2_v,
	\end{equation*}
	with norm 
	\begin{equation*}
		\|\wh{f}\|_{L^1_kL^\infty_TL^2_v}:=\int_{\Z^3}\sup_{0\le t\le T}|\wh{f}(t,k,\cdot)|_{L^2_v}\,d\Sigma(k)<\infty,
	\end{equation*}
	where the Fourier transform of $f(t,x,v)$ with respect to $x\in\T^3$ is denoted by 
	\begin{equation*}
		\wh{f}(t,k,v) = \int_{\T^3}e^{-ik\cdot x}f(t,x,v)\,dx,\quad k\in\Z^3.
	\end{equation*}
	Here we denote $d\Sigma(k)$ to be the discrete measure on $\Z^3$, namely, $\int_{\Z^3}g(k)\,d\Sigma(k)=\sum_{k\in\Z^3}g(k)$. 
	To obtain the global existence of \eqref{1}, we will denote the ``total energy functional" $\E_T$ and ``dissipation rate functional" $\D_T$ respectively by 
	\begin{equation}\label{defe}
		\E_T = \|e^{\delta t}\widehat{f}\|_{L^1_kL^\infty_TL^2_v}+\|e^{\delta t}\widehat{E}\|_{L^1_kL^\infty_T},
	\end{equation}
	and
	\begin{equation}\label{defd}
		\D_T = \|e^{\delta t}\widehat{f}\|_{L^1_kL^2_TL^2_{D}}+\|e^{\delta t}\widehat{E}\|_{L^1_kL^2_T},
	\end{equation}
for any $T>0$ and some small constant $\delta>0$ to be chosen in Theorem \ref{Main}. 
	Their weighted form $\E_{T,w}$ and $\D_{T,w}$ are given by 
	\begin{equation}\label{ETw}
		\E_{T,w} := \|e^{\delta t}w\widehat{f}\|_{L^1_kL^\infty_TL^2_v},
	\end{equation}
	and
	\begin{equation}\label{DTw}
		\D_{T,w} := \|e^{\delta t}\widehat{f}\|_{L^1_kL^2_TL^2_{D,w}}+\sqrt{qN}\|{\<v\>^{\frac{\vt}{2}}}{(1+t)^{-\frac{N+1}{2}}}e^{\delta t}w\widehat{f}\|_{L^1_kL^2_TL^2_v},
	\end{equation}
where $w$ is given in \eqref{w2}.

%	\medskip\noindent{\bf Case II: Finite channel.} 
	\subsubsection{ Case II: Finite channel} 
	For finite channel case, we will also use weight function \eqref{w2} and define function space 
	\begin{equation*}
		L^1_{\bar{k}}L^\infty_TL^2_{x_1,v},
	\end{equation*}
	with norm 
	\begin{equation*}
		\|\wh{f}\|_{L^1_{\bar{k}}L^\infty_TL^2_{x_1,v}}:=\int_{\Z^2}\sup_{0\le t\le T}|\wh{f}(t,x_1,\bar{k},v)|_{L^2_{x_1,v}}\,d\Sigma(\bar{k})<\infty.
	\end{equation*}
	where the Fourier transform of $f(t,x_1,\bar{x},v)$ with respect to $\bar{x}\in\T^2$ is denoted by 
	\begin{equation*}
		\wh{f}(t,x_1,\bar{k},v) = \int_{\T^2}e^{-i\bar{k}\cdot \bar{x}}f(t,x_1,\bar{x},v)\,d\bar{x},\quad \bar{k}\in\Z^2.
	\end{equation*}
	In the case of finite channel, we need to include an extra first-order derivative in $x$ and thus we define the ``total energy functional" $\E_T$ and ``dissipation rate functional" $\D_T$ respectively by 
	\begin{equation}\label{141}
		\E_T = \sum_{|\alpha|\le 1}\Big(\|e^{\delta t}\widehat{\partial^\alpha f}\|_{L^1_{\bar{k}}L^\infty_TL^2_{x_1,v}} +\|e^{\delta t}\wh{\partial^\alpha E}\|_{L^1_{\bar{k}}L^\infty_TL^2_{x_1}}\Big),
	\end{equation}
	and
	\begin{equation}\label{142}
		\D_T = \sum_{|\alpha|\le 1}\Big(\|e^{\delta t}\widehat{\partial^\alpha f}\|_{L^1_{\bar{k}}L^2_TL^2_{x_1}L^2_{D}}+\|e^{\delta t}\widehat{\partial^\alpha E}\|_{L^1_{\bar{k}}L^2_TL^2_{x_1}}\Big).
	\end{equation}
Here, in the $x_1$ direction, we use $L^2_{x_1}$ space instead of $L^1_{k_1}$. 
	Moreover, their weighted forms are denoted respectively by 
	\begin{equation}
		\label{143}\E_{T,w} = \sum_{|\alpha|\le 1}\|e^{\delta t}w\widehat{\partial^\alpha f}\|_{L^1_{\bar{k}}L^\infty_TL^2_{x_1,v}},
	\end{equation}
	and
	\begin{equation}
		\label{144}\D_{T,w} = \sum_{|\alpha|\le 1}\Big(\sqrt{qN}\|{\<v\>^{\frac{\vt}{2}}}{(1+t)^{-\frac{N+1}{2}}}e^{\delta t}w\widehat{\partial^\alpha f}\|_{L^1_kL^2_TL^2_v} +\|e^{\delta t}\widehat{\partial^\alpha f}\|_{L^1_{\bar{k}}L^2_TL^2_{x_1}L^2_{D,w}}\Big).
	\end{equation}

	\subsection{Main results}
	
	In this section, we state our main results on global well-posedness of VPL and VPB systems. 
	We will consider $\gamma\ge -2$ for VPL system and $\gamma+2s\ge 1$, $1/2\le s<1$ for VPB system for the case of either torus or finite channel,
	%For the case of unions of cubes, we consider $\gamma\ge -3$ for VPL case and $\{-\frac{3}{2}<\gamma+2s<0 ,\ \frac{1}{2}\le s<1\}\cup \{\gamma+2s\ge 0,\ 0<s<1\}$ for Vlasov-Poisson-Boltzmann case. 

	\subsubsection{Case of torus.} Here we state our main result on the VPL and VPB systems in torus. 
	In this subsection, we denote $\widehat{\cdot}$ to be the Fourier transform on $x\in\T^3$. 
	
	\begin{Thm}[Existence and large-time behavior]
		\label{Main}Let $\Omega=\T^3$ and $w$ be chosen by \eqref{w2}. Assume that $f_0(x,v)$ satisfies \eqref{4} with $t=0$. Then the followings hold. 
		
		\medskip
		\noindent (1) For the VPL system with $-2\le\gamma< -1$, there exist $\varepsilon_0,\delta>0$ such that if $F_0(x,v)=\mu+\mu^{1/2}f_0(x,v)\ge 0$ and 
		\begin{align*}
			\|\widehat{wf_0}\|_{L^1_kL^2_v}+\|\widehat{E_0}\|_{L^1_k}\le \varepsilon_0,
		\end{align*}
		then there exists a unique global mild solution $f=f(t,x,v)$ to the problem \eqref{1} and \eqref{VPL} satisfying that $F(t,x,v)=\mu+\mu^{1/2}f(t,x,v)\ge 0$ and for any $T>0$, 
		\begin{equation}\label{energytorus}
			\E_T+\D_T+\E_{T,w}+\D_{T,w}\le \|\widehat{wf_0}\|_{L^1_kL^2_v}+\|\widehat{E_0}\|_{L^1_k},
			%\|e^{\delta t}\widehat{f}\|_{L^1_kL^\infty_TL^2_v}+\|e^{\delta t}\widehat{E}\|_{L^1_kL^\infty_T}+\|e^{\delta t}w\widehat{f}\|_{L^1_kL^\infty_TL^2_v}
			%+\|e^{\delta t}f\|_{L^1_kL^2_TL^2_{D}}+\|e^{\delta t}\widehat{E}\|_{L^1_kL^2_T}\\+\sqrt{qN}\|{\<v\>^{\frac{\vt}{2}}}{(1+t)^{-\frac{N+1}{2}}}w\widehat{h}\|_{L^1_kL^2_TL^2_v}+\|\widehat{h}\|_{L^1_kL^2_TL^2_{D,w}}\\
			%	\lesssim \|\widehat{wf_0}\|_{L^1_kL^2_v}+\|\widehat{E_0}\|_{L^1_k}.
		\end{equation}
		where $\E_T$, $\D_T$, $\E_{T,w}$, $\D_{T,w}$ are defined in \eqref{defe}, \eqref{defd}, \eqref{ETw} and \eqref{DTw} respectively.
		In particular, one has the rate of convergence: 
		\begin{align*}
			\|\widehat{wf}(t)\|_{L^1_kL^2_v} + \|\widehat{E}(t)\|_{L^1_k}\lesssim e^{-\delta t}\big(\|\widehat{wf_0}\|_{L^1_kL^2_v}+\|\widehat{E_0}\|_{L^1_k}\big).
		\end{align*}
%	for some generic constant $\delta>0$. 

		\medskip
		\noindent (2)
		For the VPL system with $\gamma\ge -1$ and VPB system with $\gamma+2s\ge 1$ and $\frac{1}{2}\le s<1$,  there exist $\varepsilon_0,\delta>0$ such that if $F_0(x,v)=\mu+\mu^{1/2}f_0(x,v)\ge 0$ and 
		\begin{align*}
			\|\widehat{f_0}\|_{L^1_kL^2_v}+\|\widehat{E_0}\|_{L^1_k}\le \varepsilon_0,
		\end{align*}
		then there exists a unique global mild solution $f(t,x,v)$ to the problem \eqref{1} satisfying \eqref{VPL} or \eqref{VPB} with $F(t,x,v)=\mu+\mu^{1/2}f(t,x,v)\ge 0$ and for any $T>0$, 
		\begin{align*}
			%	\|e^{\delta t}\widehat{f}\|_{L^1_kL^\infty_TL^2_v}+\|e^{\delta t}\widehat{E}\|_{L^1_kL^\infty_T}+\|e^{\delta t}\widehat{f}\|_{L^1_kL^2_TL^2_{D}}+\|e^{\delta t}\widehat{E}\|_{L^1_kL^2_T}
			\E_T+\D_T\le \|\widehat{f_0}\|_{L^1_kL^2_v}+\|\widehat{E_0}\|_{L^1_k}.
		\end{align*}
		Moreover, one also has the rate of convergence: 
		\begin{align*}
			\|\widehat{f}(t)\|_{L^1_kL^2_v} + \|\widehat{E}(t)\|_{L^1_k}\lesssim e^{-\delta t}\big(\|\widehat{f_0}\|_{L^1_kL^2_v}+\|\widehat{E_0}\|_{L^1_k}\big).
		\end{align*}
	\end{Thm}
	Notice that the constants in the above estimates are independent of $T$. 
	For higher spatial regularity, we prove that the initial regularity preserves over time. 
	\begin{Thm}[Propagation of spatial regularity]\label{spatial_regurlarity}
		Let all conditions in Theorem \ref{Main} be satisfied and $T,m\ge 0$. 
		Then the followings hold. 
		
		\medskip\noindent (1)
		For the VPL systems with $-2\le\gamma< -1$, there exist $\varepsilon_0,\delta>0$ such that if
		\begin{align*}
			\|w\widehat{f_0}\|_{L^1_{k,m}L^2_v}+\|\widehat{E_0}\|_{L^1_{k,m}}\le \varepsilon_0,
		\end{align*}
		then the solution $f$ to \eqref{1} and \eqref{VPL} established in Theorem \ref{Main} satisfies 
		\begin{align}\label{regularity}
			\|e^{\delta t}w\widehat{f}\|_{L^1_{k,m}L^\infty_TL^2_v}+\|e^{\delta t}\widehat{E}\|_{L^1_{k,m}L^\infty_T}
%			+\|e^{\delta t}w\widehat{f}\|_{L^1_{k,m}L^\infty_TL^2_v}
%			+\|e^{\delta t}f\|_{L^1_{k,m}L^2_TL^2_{D}}+\|e^{\delta t}\widehat{E}\|_{L^1_{k,m}L^2_T}\\
%			+\sqrt{qN}\|{\<v\>^{\frac{\vt}{2}}}{(1+t)^{-\frac{N+1}{2}}}e^{\delta t}w\widehat{f}\|_{L^1_{k,m}L^2_TL^2_v}
%			+\|\widehat{h}\|_{L^1_{k,m}L^2_TL^2_{D,w}}
			\lesssim \|\widehat{wf_0}\|_{L^1_{k,m}L^2_v}+\|\widehat{E_0}\|_{L^1_{k,m}}.
		\end{align}
		
		\medskip\noindent (2)
		For the VPL system with $\gamma\ge -1$ and VPB system with $\gamma+2s\ge 1$  and $\frac{1}{2}\le s<1$, there exist $\varepsilon_0,\delta>0$ such that if 
		\begin{align*}
			\|\widehat{f_0}\|_{L^1_{k,m}L^2_v}+\|\widehat{E_0}\|_{L^1_{k,m}}\le \varepsilon_0,
		\end{align*}
		then the solution $f$ to the problem \eqref{1} satisfying \eqref{VPL} or \eqref{VPB} established in Theorem \ref{Main} satisfies  
		\begin{align*}
			\|e^{\delta t}\widehat{f}\|_{L^1_{k,m}L^\infty_TL^2_v}+\|e^{\delta t}\widehat{E}\|_{L^1_{k,m}L^\infty_T} 
%			+\|e^{\delta t}\widehat{f}\|^2_{L^1_{k,m}L^2_TL^2_{D}} 
%			+\|e^{\delta t}\widehat{E}\|_{L^1_{k,m}L^2_T}
			\lesssim \|\widehat{f_0}\|_{L^1_{k,m}L^2_v}+\|\widehat{E_0}\|_{L^1_{k,m}}.
		\end{align*}
		
	\end{Thm}
	
	\subsubsection{Case of finite channel.}
	In this subsection, we give the main result on the finite channel. 
	In this case, we let $\widehat{\cdot}$ to be the Fourier transform on $\bar{x}\in\T^2$. 
	\begin{Thm}\label{Channel1}
		Let $\Omega = [-1,1]\times\T^2$ and $w$ be chosen by \eqref{w2}. Assume that $f_0(x,v)$ satisfies \eqref{conservation_finite} with $t=0$. Then the followings hold.

		\medskip\noindent (1)
		For the VPL system with $-2\le\gamma<-1$, there exist $\varepsilon_0,\delta>0$ such that if $F_0(x_1,\bar{x},v)=\mu+\mu^{1/2}f_0(x_1,\bar{x},v)\ge 0$ and 
		\begin{align}\label{1.23}
			\sum_{|\alpha|\le 1}\big(\|\widehat{w\partial^\alpha f_0}\|^2_{L^1_{\bar{k}}L^2_{x_1,v}}+\|\widehat{\partial^\alpha E_0}\|^2_{L^1_{\bar{k}}L^2_{x_1}}\big)\le \varepsilon_0,
		\end{align}
		then there exists a unique mild solution $f(t,x_1,\bar{x},v)$ to the initial boundary value problem with specular reflection boundary \eqref{1}, \eqref{specular_finite}, \eqref{Neumann_finite} and \eqref{VPL}, satisfying that $F(t,x_1,\bar{x},v)=\mu+\mu^{1/2}f(t,x_1,\bar{x},v)\ge 0$ and for any $T>0$, 
		\begin{align}\label{13}
			\E_{T,w}+\D_{T,w}+\E_T + \D_T &\lesssim \sum_{|\alpha|\le 1}\big(\|\widehat{w\partial^\alpha f_0}\|^2_{L^1_{\bar{k}}L^2_{x_1,v}}+\|\widehat{\partial^\alpha E_0}\|^2_{L^1_{\bar{k}}L^2_{x_1}}\big),
		\end{align}
		%for any $T>0$, 
		where $\E_T$, $\D_T$, $\E_{T,w}$, $\D_{T,w}$ are defined in \eqref{141}, \eqref{142}, \eqref{143} and \eqref{144}, respectively. 
		In particular, one has the rate of convergence:
		\begin{align*}
			\quad\,
			\sum_{|\alpha|\le 1}
			\Big(
			\|\widehat{w\partial^\alpha f}\|^2_{L^1_{\bar{k}}L^2_{x_1,v}}
%			+\|e^{\delta t}\widehat{\partial^\alpha f}\|^2_{L^1_{\bar{k}}L^2_TL^2_{x_1}L^2_{D,w}}
%			+\frac{\sqrt{q}}{N}\|e^{\delta t}\<v\>^{1/2}|v|\widehat{w\partial^\alpha f}\|_{L^1_{\bar{k}}L^2_TL^2_{x_1,v}}
%			+\|e^{\delta t}\widehat{\partial^\alpha f}\|_{L^1_{\bar{k}}L^\infty_TL^2_{x_1,v}} 
			+\|\wh{\partial^\alpha E}\|_{L^1_{\bar{k}}L^2_{x_1}}
%			+\|e^{\delta t}\widehat{\partial^\alpha f}\|_{L^1_{\bar{k}}L^2_TL^2_{x_1}L^2_{D}}
%			+\|e^{\delta t}\widehat{\partial^\alpha E}\|_{L^1_{\bar{k}}L^2_TL^2_{x_1}}
			\Big)
			\lesssim e^{-\delta t}\sum_{|\alpha|\le 1}\big(\|\widehat{w\partial^\alpha f_0}\|^2_{L^1_{\bar{k}}L^2_{x_1,v}}+\|\wh{\partial^\alpha E_0}\|^2_{L^1_{\bar{k}}L^2_{x_1}}\big).
		\end{align*}

		\medskip\noindent (2)
		For the VPL system with $\gamma\ge -1$ and VPB system with $\gamma+2s\ge 1$ and $\frac{1}{2}\le s<1$, there exist $\varepsilon_0,\delta>0$ such that if $F_0(x_1,\bar{x},v)=\mu+\mu^{1/2}f_0(x_1,\bar{x},v)\ge 0$ and 
		\begin{align}\label{124}
			\sum_{|\alpha|\le 1}\big(\|\widehat{\partial^\alpha f_0}\|^2_{L^1_{\bar{k}}L^2_{x_1,v}}+\|\widehat{\partial^\alpha E_0}\|^2_{L^1_{\bar{k}}L^2_{x_1}}\big)\le \varepsilon_0,
		\end{align}
		then there exists a unique mild solution $f(t,x_1,\bar{x},v)$ to the initial boundary value problem with specular reflection boundary  \eqref{1}, \eqref{specular_finite} and \eqref{Neumann_finite} satisfying \eqref{VPL} or \eqref{VPB} such that $F(t,x_1,\bar{x},v)=\mu+\mu^{1/2}f(t,x_1,\bar{x},v)\ge 0$ and  for any $T>0$, 
		\begin{align*}\notag
			\E_T + \D_T &\lesssim \sum_{|\alpha|\le 1}\big(\|\widehat{\partial^\alpha f_0}\|^2_{L^1_{\bar{k}}L^2_{x_1,v}}+\|\widehat{\partial^\alpha E_0}\|^2_{L^1_{\bar{k}}L^2_{x_1}}\big).
		\end{align*}%for any $T>0$. 
		In particular, one has the large time behavior: 
		\begin{align*}\notag
		\quad\,\|\widehat{\partial^\alpha f}\|_{L^1_{\bar{k}}L^2_{x_1,v}} 
		+\|\wh{\partial^\alpha E}\|_{L^1_{\bar{k}}L^2_{x_1}}
%		+\|e^{\delta t}\widehat{\partial^\alpha f}\|_{L^1_{\bar{k}}L^2_TL^2_{x_1}L^2_{D}}
%		+\|e^{\delta t}\widehat{\partial^\alpha E}\|_{L^1_{\bar{k}}L^2_TL^2_{x_1}}\big)\\
		\lesssim e^{-\delta t}\sum_{|\alpha|\le 1}\big(\|\widehat{w\partial^\alpha f_0}\|^2_{L^1_{\bar{k}}L^2_{x_1,v}}+\|\partial^\alpha\widehat{E_0}\|^2_{L^1_{\bar{k}}L^2_{x_1}}\big).
		\end{align*}
	\end{Thm}
	
	Similar to Theorem \ref{spatial_regurlarity}, we have the propagation of spatial regularity in variable $\bar{x}$. 
	\begin{Thm}[Propagation of spatial regularity in $\bar{x}$.] \label{Channel2}
		Let all the assumptions in Theorem \ref{Channel1} be satisfied and $T,m\ge 0$. 
		
		\medskip\noindent (1)
		For the VPL system with $-2\le\gamma< -1$, there exist $\varepsilon_0,\delta>0$ such that if
		\begin{align*}
			\|w\widehat{f_0}\|_{L^1_{\bar{k},m}L^2_{x_1}L^2_v}+\|\widehat{E_0}\|_{L^1_{\bar{k},m}L^2_{x_1}}\le \varepsilon_0,
		\end{align*}
		then the solution $f$ to \eqref{1} and \eqref{VPL} established in Theorem \ref{Main} satisfies 
		\begin{align*}
%			\|e^{\delta t}\widehat{f}\|_{L^1_{\bar{k},m}L^\infty_TL^2_{x_1}L^2_v}
\|e^{\delta t}w\widehat{f}\|_{L^1_{\bar{k},m}L^\infty_TL^2_{x_1}L^2_v}+
			\|e^{\delta t}\widehat{E}\|_{L^1_{\bar{k},m}L^\infty_TL^2_{x_1}}
%			+\|e^{\delta t}f\|_{L^1_{\bar{k},m}L^2_TL^2_{x_1}L^2_{D}}\\+\|e^{\delta t}\widehat{E}\|_{L^1_{\bar{k},m}L^2_TL^2_{x_1}}
%+\sqrt{qN}\|{\<v\>^{\frac{\vt}{2}}}{(1+t)^{-\frac{N+1}{2}}}e^{\delta t}w\widehat{f}\|_{L^1_{\bar{k},m}L^2_TL^2_{x_1}L^2_v}
%+\|\widehat{h}\|_{L^1_{\bar{k},m}L^2_TL^2_{x_1}L^2_{D,w}}
			\lesssim \|\widehat{wf_0}\|_{L^1_{\bar{k},m}L^2_{x_1}L^2_v}+\|\widehat{E_0}\|_{L^1_{\bar{k},m}L^2_{x_1}}.
		\end{align*}

		\medskip\noindent (2)
		For the VPL system with $\gamma\ge -1$ and VPB system with $\gamma+2s\ge 1$ and $\frac{1}{2}\le s<1$, there exist $\varepsilon_0,\delta>0$ such that if 
		\begin{align*}
			\|\widehat{f_0}\|_{L^1_{\bar{k},m}L^2_{x_1}L^2_v}+\|\widehat{E_0}\|_{L^1_{\bar{k},m}L^2_{x_1}}\le \varepsilon_0,
		\end{align*}
		then the solution $f$ to the problem \eqref{1} satisfying \eqref{VPL} or \eqref{VPB} established in Theorem \ref{Main} satisfies  
		\begin{align*}
			\|e^{\delta t}\widehat{f}\|_{L^1_{\bar{k},m}L^\infty_TL^2_{x_1}L^2_v}+\|e^{\delta t}\widehat{E}\|_{L^1_{\bar{k},m}L^\infty_TL^2_{x_1}} 
%			+\|e^{\delta t}\widehat{f}\|^2_{L^1_{\bar{k},m}L^2_TL^2_{x_1}L^2_{D}} 
%			+\|e^{\delta t}\widehat{E}\|_{L^1_{\bar{k},m}L^2_{x_1}L^2_T}\\
			\lesssim \|\widehat{f_0}\|_{L^1_{\bar{k},m}L^2_{x_1}L^2_v}+\|\widehat{E_0}\|_{L^1_{\bar{k},m}L^2_{x_1}}.
		\end{align*}
		% then for any $m\ge 0$, there exists $\varepsilon_0>0$ such that if 
		%	\begin{align}\label{14a}
		%		\sum_{|\alpha|\le 1}\big(\|\widehat{w\partial^\alpha f_0}\|^2_{L^1_{\bar{k},m}L^\infty_TL^2_{x_1,v}}+\|\widehat{\partial^\alpha E_0}\|^2_{L^1_{\bar{k},m}L^2_{x_1}}\big)\le \varepsilon_0,
		%	\end{align}then we have 
		%\begin{align*}
		%	&\quad\,\sum_{|\alpha|\le 1}\big(\|\widehat{\partial^\alpha f}\|_{L^1_{\bar{k},m}L^\infty_TL^2_{x_1,v}} 
		%	+\|\wh{\partial^\alpha E}\|_{L^1_{\bar{k},m}L^\infty_TL^2_{x_1}}+ \|\widehat{\partial^\alpha f}\|_{L^1_{\bar{k},m}L^2_TL^2_{x_1}L^2_{D}}
		%	+\|\widehat{\partial^\alpha E}\|_{L^1_{\bar{k},m}L^2_TL^2_{x_1}}\big)\\ &+\sum_{|\alpha|\le 1}\big(\|\widehat{w\partial^\alpha f}\|^2_{L^1_{\bar{k},m}L^\infty_TL^2_{x_1,v}}+\frac{\sqrt{q}}{N}\sum_{|\alpha|\le 1}\|\<v\>^{1/2}|v|\widehat{w\partial^\alpha f}\|_{L^1_{\bar{k},m}L^2_TL^2_{x_1,v}}+ \|\widehat{\partial^\alpha f}\|^2_{L^1_{\bar{k},m}L^2_TL^2_{x_1}L^2_{D,w}}\big)\\ &\lesssim \sum_{|\alpha|\le 1}\big(\|\widehat{w\partial^\alpha f_0}\|^2_{L^1_{\bar{k},m}L^\infty_TL^2_{x_1,v}}+\|\widehat{\partial^\alpha E_0}\|^2_{L^1_{\bar{k},m}L^2_{x_1}}\big).
		%\end{align*}
	\end{Thm}
	
	Noticing that the $L^1_k$ space contains no derivative, we obtain the solutions to \eqref{1} in Theorems \ref{Main} and \ref{Channel1} with low regularity on the torus direction. This generalizes the results in \cite{Duan2020} in the pure Boltzmann and Landau case to those in the VPB and VPL case. In what follows, we give the main idea of the proof in brief. Firstly, we analyze the macroscopic estimates on the term $\P f$. As in \cite{Duan2020}, we take the velocity moments for equation \eqref{1} and take corresponding Fourier transform. 
	
	For the case of torus, we apply the conservation laws \eqref{4} to obtain that 
	\begin{align*}
		\Big(\int^T_0|\widehat{c}|_{k=0}|^2\,dt\Big)^{1/2}
%		\lesssim \Big(\int^T_0\|E\|^2_{L^\infty_x}\,dt\Big)^{1/2}
		\lesssim \Big(\int^T_0\|\widehat{E}\|^4_{L^1_k}\,dt\Big)^{1/2}
%		\lesssim \sup_{0\le t\le T}\|\widehat{E}\|_{L^1_k}\Big(\int^T_0\|\widehat{E}\|_{L^1_k}\,dt\Big)^{1/2}
		\lesssim \|\widehat{E}\|_{L^1_kL^\infty_T}\|\widehat{E}\|_{L^1_kL^2_T}.	
	\end{align*} 
	The arising new difficulty comes from the estimate on $\wh E$ and $\widehat{a_+}-\widehat{a_-}$. We take the inner product of \eqref{1} with 
%	\begin{align*}
		$\widehat{\Phi_a} = (|v|^2-10)iv\cdot k\widehat{\phi_a}\mu^{1/2}$,
%	\end{align*}
	where $\widehat{\phi_a}(t,k)=(\widehat{\phi_a}_+(t,k),\widehat{\phi_a}_-(t,k))$ is a solution to 
	\begin{align*}
		|k|^2\widehat{\phi_a}_+(t,k) = \widehat{a_+}-\widehat{a_-},\quad |k|^2\widehat{\phi_a}_-(t,k) = \widehat{a_-}-\widehat{a_+}.
	\end{align*}
Then one can derive the macroscopic terms $\int^T_0|\widehat{a_+}-\widehat{a_-}|^2dt$ and $\int^T_0|\widehat{\nabla_x\phi}|^2dt$ from $\sum_\pm\int^T_0(\PP f,iv\cdot k\widehat{\Phi_{a\pm}})_{L^2_v}dt$ and $\sum_\pm\mp\int^T_0(\widehat{\nabla_x\phi}\cdot v\mu^{1/2},\widehat{\Phi_{a\pm}})_{L^2_v}dt$ respectively. 

\smallskip
For the case of finite channel, we need to deal with the boundary term arising from $v_1\pa_{x_1}\wh{f_\pm}$. In this case, we need to choose the test function carefully. For instance, for $\wh{\pa c}$, we choose $\widehat{\Phi_c} = (|v|^2-5)\big(v\cdot\widehat{\nabla_{x}\phi_c}(t,x_1,\bar{k})\big)\mu^{1/2},$
where $\phi_c$ solves 
\begin{equation*}\left\{
\begin{aligned}
	&-\partial^2_{x_1}\widehat{\phi_c}+|\bar{k}|^2\widehat{\phi_c} = \widehat{\partial c},\\
	&\widehat{\phi_c}(\pm 1,\bar{k})=0, \qquad \text{ if }\partial=\partial_{x_1},\\
	&\partial_{x_1}\widehat{\phi_c}(\pm 1,\bar{k})=0, \quad\text{ if }\partial=I,\partial_{x_2},\partial_{x_3}.
\end{aligned}\right.
\end{equation*}
With these nice chosen test functions, one can check that the boundary term vanishes by applying change of variable $v_1\mapsto -v_1$ and specular-reflection boundary condition:
\begin{align*}
	-\int^T_0(v_1\widehat{\partial f_\pm}(1),\widehat{\Phi}(1))_{L^2_v}\,dt+ \int^T_0(v_1\widehat{\partial f_\pm}(-1),\widehat{\Phi}(-1))_{L^2_v}\,dt = 0. 
\end{align*}

Notice that we can obtain the ``full" dissipation terms as indicated in \eqref{defd}, \eqref{142} and hence, the exponential time decay on $(f,E)$ can be derived. By applying the weight $\<k\>^m$ after the Fourier transform $\wh{\cdot}$, one can obtain the propagation of initial regularity over time. 

	The paper is organized as follows. In Section \ref{Sec2}, we give some basic estimates on collision operators. In Section \ref{section2}, we give the macroscopic dissipation estimates for VPL and VPB systems in torus. In Section \ref{MacroFinite}, we derive the macroscopic dissipation  estimates in finite channel. In Sections \ref{SecTorus} and \ref{SecFinite}, we prove the global existence and large-time behavior with the help of macroscopic estimates for torus and finite channel, respectively. To complete the arguments, we give  in Section \ref{Sec_loc} the proof for the local-in-time existence of the Vlasov-Poisson-Landau equation with the specular reflection boundary condition in the finite channel.

	%
	%\subsubsection{Case of unions of cubes}
	%
	%In this section, we present the main result on unions of cubes. 
	%
	%\begin{Thm}\label{Theorem_bounded}
	%	Let $\Omega$ be defined by \eqref{Omega} and $w_{l,\nu}(\alpha,\beta)$ given by \eqref{w22}.
	%	Let $\gamma\ge -3$ for Landau case and $(\gamma,s)\in\{-\frac{3}{2}<\gamma+2s<0 ,\ \frac{1}{2}\le s<1\}\cup \{\gamma+2s\ge 0,\ 0<s<1\}$ for Boltzmann case. 
	%%	In particular, we let $\frac{1}{2}\le s<1$ for soft potential in Boltzmann case.
	%	Let $l\ge 3$.  
	%	There exists $\varepsilon_0,\nu>0$ such that if $F_0(x,v)=\mu+\mu^{1/2}f_0(x,v)\ge 0$ satisfying \eqref{conservation_bounded} and 
	%	\begin{align}\label{small}
	%		\sum_{|\alpha|+|\beta|\le 3}\big(\|{w_{l,\nu}(\alpha,\beta)\partial^\alpha_\beta f_0}\|^2_{L^2_{x,v}}+\|{\partial^\alpha E_0}\|^2_{L^2_{x}}\big)\le \varepsilon_0,
	%	\end{align}
	%	then there exists a unique mild solution $f(t,x,v)$ to the specular reflection boundary problem \eqref{1}, \eqref{Neumann} and \eqref{specular}, satisfying that $F(t,x,v)=\mu+\mu^{1/2}f(t,x,v)\ge 0$ and for any $T>0$,
	%	\begin{align}
	%		\sup_{0\le t\le T}e^{\delta t^p}\E(t) + \sup_{0\le t\le T}\E_\nu(t)\lesssim \ve_0,
	%		%	\sup_{0\le t\le T}\E(t)+\int^T_0\D(t)\,dt &\lesssim \sum_{|\alpha|\le 3}\big(\|{w\partial^\alpha f_0}\|^2_{L^2_{x,v}}+\|{\partial^\alpha E_0}\|^2_{L^2_{x}}\big),
	%	\end{align}
	%	where $\E_\nu(t)$ and $\E(t)=\E_0(t)$ are defined by \eqref{E1}.
	%	%for some $\lambda>0$, where $\E(t)=\E_0(t)$ and $\D(t)=\D_0(t)$ are defined by \eqref{E1} and \eqref{D} respectively. 
	%\end{Thm}

	\section{Preliminary}\label{Sec2}
	In this section, we give some basic estimate on collision operator $L$ and $\Gamma(\cdot,\cdot)$. 
	%
	%For the Landau case, by \cite[Lemma 6]{Strain2007} and \cite[Lemma 1]{Guo2002a}, we have 
	%\begin{equation*}
	%	2\mu^{-1/2}Q(\mu,\mu^{1/2}f_\pm) = 2\partial_{v_i}(\sigma^{ij}\partial_{v_j}f_\pm)-2\sigma^{ij}\frac{v_i}{2}\frac{v_j}{2}f_\pm+2\partial_{v_i}\sigma^if_\pm,
	%\end{equation*}
	%\begin{equation*}
	%	\mu^{-1/2}Q(\mu^{1/2}(f_++f_-),\mu) = -\sum_\pm\mu^{-1/2}\partial_{v_i}\Big[\mu\Big(\phi^{ij}*\Big[\mu^{1/2}(\partial_{v_j}f_\pm+\frac{v_j}{2}f_\pm)\Big]\Big)\Big],
	%\end{equation*}
	%\begin{align*}
	%	\mu^{-1/2}Q(\mu^{1/2}f,\mu^{1/2}g) =\notag \partial_{v_i}\Big[\phi^{ij}*\Big(\mu^{1/2}f\Big)\partial_{v_j}g\Big] - \phi^{ij}*\Big(v_i\mu^{1/2}f\Big)\partial_{v_j}g 
	%	\\- \partial_{v_i}\Big[\phi^{ij}*\Big(\mu^{1/2}\partial_{v_j}f\Big)g\Big] + \phi^{ij}*\Big(v_i\mu^{1/2}\partial_{v_j}f\Big)g.
	%\end{align*}
	We begin with splitting $L_\pm$. For the Landau case, let $\varepsilon>0$ small and choose a smooth cutoff function $\chi(|v|)\in[0,1]$ such that 
	$
	\chi(|v|)=1\text{ if } |v|<\varepsilon;\  \chi(|v|)=0 \text{ if } |v|>2\varepsilon.
	$ Then we split $L_\pm f = -A_\pm f + K_\pm f$ as in \cite[Section 4.2]{Yang2016}, where  
	\begin{align*}
		-A_\pm f &= 2\partial_{v_i}(\sigma^{ij}\partial_{v_j}f_\pm)-2\sigma^{ij}\frac{v_i}{2}\frac{v_j}{2}f_\pm+2\partial_{v_i}\sigma^i\1_{|v|> R}f_\pm+A_1f\notag\\
		&\qquad+(K_1-\1_{|v|\le R}K_1\1_{|v|\le R})f,\\
		K_\pm f &= 2\partial_{v_i}\sigma^i\1_{|v|\le R}f_\pm + \1_{|v|\le R}K_1\1_{|v|\le R}f,
	\end{align*}
	where $R>0$ is to be chosen large, $\varepsilon>0$ is to be chosen small, and $A_1$ and $K_1$ are given respectively by 
	\begin{align*}
		A_1f &= -\sum_\pm\mu^{-1/2}\partial_{v_i}\Big\{\mu\Big[\Big(\phi^{ij}\chi\Big)*\Big(\mu\partial_{v_j}\big[\mu^{-1/2}f_\pm\big]\Big)\Big]\Big\},\\
		K_1f &=  -\sum_\pm\mu^{-1/2}\partial_{v_i}\Big\{\mu\Big[\Big(\phi^{ij}\big(1-\chi\big)\Big)*\Big(\mu\partial_{v_j}\big[\mu^{-1/2}f_\pm\big]\Big)\Big]\Big\},
	\end{align*}
	with the convolution taken with respect to the velocity variable $v$. 
	Then \cite[eq. (4.33) and eq. (4.32)]{Yang2016} shows that 
	\begin{equation*}%\label{20a}
		\sum_\pm(A_\pm f,f_\pm)_{L^2_{v}}
		\ge c_0|f|^2_{L^2_{D}},
	\end{equation*}
	and 
	\begin{equation}\label{20aa}
		|(K_1g,h)_{L^2_v}|\lesssim |\mu^{1/10}g|_{L^2_v}|\mu^{1/10}h|_{L^2_v}. 
	\end{equation}
	From \cite[Lemma 3]{Guo2002a}, we know   
	\begin{align}\label{65aa}
		|\partial_\beta\sigma^{ij}(v)|+|\partial_\beta\sigma^i(v)|\le C_\beta(1+|v|)^{\gamma+2-|\beta|}.
	\end{align}
	Thus, \eqref{20aa} and \eqref{65aa} imply that $K$ is a bounded operator on $L^2_v$ with estimate 
	\begin{align*}%\label{K1}
		|Kf|_{L^2_v}\lesssim |\mu^{1/10}f|_{L^2_v}. 
	\end{align*}
	
	For Boltzmann case, we split $L_\pm f = -A_\pm f + K f$ with 
	\begin{align*}
		-A_\pm f &= 2\mu^{-1/2}Q(\mu,\mu^{1/2}f_\pm),\\
		K f &= \mu^{-1/2}Q(\mu^{1/2}(f_++f_-),\mu).
	\end{align*}
	Then by \cite[Lemma 2.15]{Alexandre2012}, we have 
	\begin{align*}%\label{K2}
		|Kf|_{L^2_v}\lesssim |\mu^{1/10^3}f|_{L^2_v}. 
	\end{align*}
Moreover, we have the following Lemma on the estimates of $L_\pm$ and $\Gamma_\pm$.

	\begin{Lem}%\label{Lem21}
	Let $w$ be given by \eqref{w2}. Assume $\gamma>\max\{-3,-2s-\frac{3}{2}\}$ for Boltzmann case and $\gamma\ge-3$ for Landau case. 
		Then 
		\begin{align}\label{36b}
			\sum_{\pm}(L_\pm f,f_\pm)_{L^2_v} \gtrsim |\{\I-\P\}f|_{L^2_D}^2, 
		\end{align}
		and
		\begin{equation}\label{36c}\begin{aligned}
				\sum_\pm(w^2L_\pm g,g_\pm)_{L^2_v}&\ge c_0|g|_{L^2_{D,w}}^2-C|g|_{L^2(B_C)}^2,
			\end{aligned}
		\end{equation}
		%\begin{equation}\label{36cw}\begin{aligned}
		%	\sum_\pm(w^2_{l,\nu}(0,\beta)\pa_{\beta}L_\pm g,g_\pm)_{L^2_v}&\ge c_0|\pa_\beta g|_{L^2_{D,w_{l,\nu}(0,\beta)}}^2
		%%	-\eta\sum_{|\beta_1|=|\beta|}|\pa_{\beta_1} g|_{L^2_{D,w_{l,\nu}(0,\beta_1)}}^2
		%	-C\sum_{|\beta_1|<|\beta|}|\pa_{\beta_1} g|_{L^2_{D,w_{l,\nu}(0,\beta_1)}}^2
		%	\\&\qquad\qquad\qquad\qquad\qquad\qquad\qquad -C|g|_{L^2(B_C)}^2.
		%\end{aligned}
		%\end{equation}
		%for some generic constant $c_0,C>0$. 
		There exists decomposition $L_\pm = -A_\pm + K_\pm$ such that $K_\pm$ is a bounded linear operator on $L^2_v$ and $A_\pm$ satisfies
		%	\begin{equation*}
		%		|g|^2_{L^2_{D,w}}\approx |w_2\<v\>^{\gamma/2}P_v\partial_{v_i}g|_{L^2_v}^2+|w_2\<v\>^{(\gamma+2)/2}\{I-P_v\}\partial_{v_i}g|^2_{L^2_v}+|w_2\<v\>^{(\gamma+2)/2}g|^2_{L^2_v},
		%	\end{equation*} 
		\begin{equation}
			\label{36a}
			\sum_\pm(w^2A_\pm g,g_\pm)\ge c_0|g|_{L^2_{D,w}}^2-C|g|_{L^2(B_C)}^2,
		\end{equation}
		%and
		%\begin{equation}\label{36aw}\begin{aligned}
		%		\sum_\pm(w^2_{l,\nu}(0,\beta)\pa_{\beta}A_\pm g,g_\pm)&\ge c_0|\pa_\beta g|_{L^2_{D,w_{l,\nu}(0,\beta)}}^2
		%		%	-\eta\sum_{|\beta_1|=|\beta|}|\pa_{\beta_1} g|_{L^2_{D,w_{l,\nu}(0,\beta_1)}}^2
		%		-C\sum_{|\beta_1|<|\beta|}|\pa_{\beta_1} g|_{L^2_{D,w_{l,\nu}(0,\beta_1)}}^2
		%		\\&\qquad\qquad\qquad\qquad\qquad\qquad\qquad -C|g|_{L^2(B_C)}^2.
		%	\end{aligned}
		%\end{equation}
		Moreover, for any $|\alpha|+|\beta|\le 3$, we have 
		%for estimate on Landau operator, we have
		\begin{multline}\label{35a}
			\big(w^2\pa^\alpha_\beta\Gamma_\pm(g_1,g_2),\pa^\alpha_\beta g_3\big)_{L^2_v}\\\lesssim \sum\big(|w\pa^{\alpha_1}_{\beta_1}g_1|_{L^2_v}|\pa^{\alpha_2}_{\beta_2}g_2|_{L^2_{D,w}}+|\pa^{\alpha_1}_{\beta_1}g_1|_{L^2_{D,w}}|w\pa^{\alpha_2}_{\beta_2}g_2|_{L^2_{v}}\big)|\pa^\alpha_\beta g_3|_{L^2_{D,w}},
		\end{multline}
		where the summation is taken over $\alpha_1+\alpha_2=\alpha$, $\beta_1+\beta_2\le \beta$. 
		%Similarly, for non-cutoff Boltzmann operator, it holds that 
		%\begin{align*}
		%	content...
		%\end{align*}
		%Consequently, 
		%\begin{multline}\label{37a}
		%	\big(w^2\pa^\alpha_\beta\Gamma_\pm(g_1,g_2),\pa^\alpha_\beta g_3\big)_{L^2_{x,v}}\lesssim \Big(\sum_{|\alpha_1|+|\beta_1|\le 3} \|w\pa^{\alpha_1}_{\beta_1}g_1\|_{L^2_xL^2_v}\sum_{|\alpha_1|+|\beta_1|\le 3}\|\pa^{\alpha_1}_{\beta_1}g_2\|_{L^2_xL^2_{D,w}}\\+\sum_{|\alpha_1|+|\beta_1|\le 3}\|\pa^{\alpha_1}_{\beta_1}g_1\|_{L^2_xL^2_{D,w}}\sum_{|\alpha_1|+|\beta_1|\le 3}\|w\pa^{\alpha_1}_{\beta_1}g_2\|_{L^2_xL^2_v}\Big)\|\pa^\alpha_\beta g_3\|_{L^2_xL^2_{D,w}}.
		%\end{multline}
		%and 
		%\begin{equation}
		%	(w^2\pa^\alpha_\beta\Gamma_\pm(g_1,g_2),\pa^\alpha_\beta g_3)_{L^2_{x,v}}\lesssim \sqrt{\E(t)}\D_\nu(t).
		%\end{equation}
		%%and 
		%\begin{equation*}
		%	(w^2_{l-|\beta|}\pa^\alpha_\beta\Gamma_\pm(g_1,g_2),\pa^\alpha_\beta g_3)_{L^2_{x,v}}\lesssim \sqrt{\E(t)}\D_\nu(t).
		%\end{equation*}
	\end{Lem}
	\begin{proof}
		The proof of \eqref{36b} and \eqref{36c} can be found in \cite[Lemma 5, Lemma 9]{Strain2007} for Landau case and
		\cite[Lemma 2.7]{Duan2013a} for the Boltzmann case.   Note that the proof in \cite{Duan2013a} is also valid for hard potential. The proof of \eqref{36a} can be found in \cite[Lemma 7 and Lemma 8]{Strain2007} for Landau case and \cite[Lemma 2.6]{Duan2013a} for Boltzmann case.
		%	 Using \eqref{36c} and boundedness of $K$ from \eqref{K1}, \eqref{K2}, we can obtain \eqref{36aw}. 
		The proof of \eqref{35a} is given by \cite[Lemma 10]{Strain2007} for Landau case and \cite[Lemma 2.4 and Lemma 2.4]{Duan2013a, Fan2017}. 
\end{proof}

	\section{Macroscopic estimates on torus}\label{section2}
	In this section we will derive the {\em a priori} estimates for the macroscopic part of the solution to equation:
	\begin{align}\label{40}
		\partial_t{f_\pm}+v\cdot \nabla_x{f_\pm} \pm {\nabla_x\phi}\cdot v\mu^{1/2} - L_\pm {f} = {g_\pm},
	\end{align}
	where 
	\begin{align}\label{40a}
		g_\pm = \pm&\nabla_x\phi\cdot\nabla_vf_\pm\mp\frac{1}{2}\nabla_x\phi\cdot vf_\pm+\Gamma_\pm(f,f),
	\end{align}
	and $\phi$ is given by 
	\begin{equation}
		\label{2}
		-\Delta_x\phi = a_+-a_-. 
	\end{equation}
	
	To capture the macroscopic dissipation, we take the following velocity moments
	\begin{equation*}
		\mu^{\frac{1}{2}}, v_j\mu^{\frac{1}{2}}, \frac{1}{6}(|v|^2-3)\mu^{\frac{1}{2}},
		(v_j{v_m}-1)\mu^{\frac{1}{2}}, \frac{1}{10}(|v|^2-5)v_j \mu^{\frac{1}{2}}
	\end{equation*}
	with {$1\leq j,m\leq 3$} for the equation \eqref{40}. By taking the average and difference on $\pm$ of the resultant equations, one sees that  
	the coefficient functions $[a_\pm,b,c]=[a_\pm,b,c](t,x)$ satisfy
	the fluid-type system 
	\begin{equation}\label{19}\left\{
		\begin{aligned}
			&\partial_t\Big(\frac{a_++a_-}{2}\Big)+\nabla_x\cdot b = 0,\\
			&\partial_tb_j+\partial_{x_j}\Big(\frac{a_++a_-}{2}+2c\Big)+\frac{1}{2}\sum_{m=1}^3\partial_{x_m}\Theta_{jm}(\{\I-\P\}f\cdot[1,1])
			= \frac{1}{2}\sum_\pm(g_\pm,v_j\mu^{1/2})_{L^2_v},\\
			&\partial_tc+\frac{1}{3}\nabla_x\cdot b + \frac{5}{6}\sum^3_{j=1}\partial_{x_j}\Lambda_j(\{\I-\P\}f\cdot[1,1]) = \frac{1}{12}\sum_\pm(g_\pm,(|v|^2-3)\mu^{1/2})_{L^2_v},\\
			&\partial_t\Big(\frac{1}{2}\Theta_{jm}((\I-\P)f\cdot[1,1])+2c\delta_{jm}\Big) + \partial_{x_j}b_m+\partial_{x_m}b_j = \frac{1}{2}\sum_\pm\Theta_{jm}(g_\pm+h_\pm),\\
			&\frac{1}{2}\partial_t\Lambda_j((\I-\P)f\cdot[1,1])+\partial_{x_j}c = \frac{1}{2}\Lambda_j(g_++g_-+h_++h_-),
		\end{aligned}\right.
	\end{equation}for $1\le j,m\le3$, where
	\begin{align*}
		h_\pm &= -v\cdot\nabla_x(\II-\PP)f+L_\pm f,\\
		\Theta_{jm}(f_\pm) = (f_\pm,(v_jv_m&-1)\mu^{1/2})_{L^2_v},\ \ \Lambda_j(f_\pm) =\frac{1}{10}(f_\pm,(|v|^2-5)v_j\mu^{1/2})_{L^2_v},
	\end{align*}
	and 
	\begin{equation}\label{98a}\left\{
		\begin{aligned}
			&\partial_t({a_+}-{a_-})+\nabla_{x}\cdot {G}=0,\\
			&\partial_t{G}+\nabla_x({a_+}-{a_-})-2{E}+\nabla_x\cdot\Theta(\{\I-\P\}{f}\cdot[1,-1])=(({g}+L{f})\cdot[1,-1],v\mu^{1/2})_{L^2_v},
		\end{aligned}\right.
	\end{equation}
	where
	\begin{align}\label{93a}
		G = (\{\I-\P\}f\cdot[1,-1],v\mu^{1/2})_{L^2_v}.
	\end{align}
	For further application in the case of torus, we take Fourier transform on $x\in\T^3$ of \eqref{19} \eqref{98a} and \eqref{2} to obtain 
	\begin{equation}\label{28}
		\left\{\begin{aligned}
			&\partial_t\Big(\frac{\widehat{a_+}+\widehat{a_-}}{2}\Big)+ik\cdot \widehat{b} = 0,\\
			&\partial_t\widehat{b_j}+ik_j\Big(\frac{\widehat{a_+}+\widehat{a_-}}{2}+2\widehat{c}\Big)+\frac{1}{2}\sum_{m=1}^3ik_m\Theta_{jm}(\{\I-\P\}\widehat{f}\cdot[1,1])
			%		\\&\qquad\qquad\qquad\qquad\qquad\qquad\qquad\qquad\qquad\qquad\qquad
			= \frac{1}{2}\sum_\pm(\widehat{g_\pm},v_j\mu^{1/2})_{L^2_v},\\
			&\partial_t\widehat{c}+\frac{1}{3}ik\cdot \widehat{b} + \frac{5}{6}\sum^3_{j=1}ik_j\Lambda_j(\{\I-\P\}\widehat{f}\cdot[1,1]) = \frac{1}{12}\sum_\pm(\widehat{g_\pm},(|v|^2-3)\mu^{1/2})_{L^2_v},\\
			&\partial_t\Big(\frac{1}{2}\Theta_{jm}(\{\I-\P\}\widehat{f}\cdot[1,1])+2\widehat{c}\delta_{jm}\Big) + ik_j\widehat{b_m}+ik_m\widehat{b_j} = \frac{1}{2}\sum_\pm\Theta_{jm}(\widehat{g_\pm}+\widehat{h_\pm}),\\
			&\frac{1}{2}\partial_t\Lambda_j(\{\I-\P\}\widehat{f}\cdot[1,1])+ik_j\widehat{c} = \frac{1}{2}\sum_\pm\Lambda_j(\widehat{g_\pm}+\widehat{h_\pm}),
		\end{aligned}\right.
	\end{equation}
	and 
	\begin{equation}\label{27}\left\{
		\begin{aligned}
			&\partial_t(\widehat{a_+}-\widehat{a_-})+ik\cdot \widehat{G}=0,\\
			&\partial_t\widehat{G}+ik(\widehat{a_+}-\widehat{a_-})-2\widehat{E}+ik\cdot\Theta(\{\I-\P\}\widehat{f}\cdot[1,-1])=((\widehat{g}+L\widehat{f})\cdot[1,-1],v\mu^{1/2})_{L^2_v},\\
			&ik\cdot \widehat{E}=\widehat{a_+}-\widehat{a_-}.
		\end{aligned}\right.
	\end{equation}
	%where
	%\begin{align*}
	%	h_\pm &= -v\cdot\nabla_x(\II-\PP)f+L_\pm f,\\
	%	\Theta_{jm}(f_\pm) = (f_\pm,(v_jv_m&-1)\mu^{1/2})_{L^2_v},\ \ \Lambda_j(f_\pm) =\frac{1}{10}(f_\pm,(|v|^2-5)v_j\mu^{1/2})_{L^2_v},
	%\end{align*}and 
	%\begin{align}\label{29}
	%	G = (\{\I-\P\}f\cdot[1,-1],v\mu^{1/2})_{L^2_v}.
	%\end{align}
	With the above preparation, we are ready to obtain the following macroscopic estimate in isotropic case. 
	\begin{Thm}\label{Thm21}
	Let $\gamma\ge -3$ for Landau case, $\gamma>\max\{-3,-2s-3/2\}$ for Boltzmann case and $T>0$. Let $f$ be the solution of \eqref{1} in torus with initial data satisfying \eqref{4}. Then 
		\begin{multline}\label{53}
			\|(\widehat{a_+},\widehat{a_-},\widehat{b},\widehat{c})\|_{L^1_kL^2_T}+\|\widehat{E}\|_{L^1_kL^2_T}\lesssim \|\widehat{f}\|_{L^1_kL^\infty_TL^2_v}+\|\widehat{f_0}\|_{L^1_kL^2_v}+\|\{\I-\P\}\widehat{f}\|_{L^1_kL^2_TL^2_{D}}\\\quad +\|\widehat{E}\|_{L^1_kL^\infty_T}\|\widehat{E}\|_{L^1_kL^2_T}	+\big(\|\widehat{E}\|_{L^1_kL^\infty_T}+\|\widehat{f}\|_{L^1_kL^\infty_TL^2_v}\big)\|\widehat{f}\|_{L^1_kL^2_TL^2_{D}}.	
		\end{multline}
	\end{Thm}
	\begin{proof}
	We will prove the desired estimate by dividing norm $L^1_k$ into cases $k=0$ and $k\neq 0$. 
		
		\medskip \noindent{\bf Case $k=0$:}	
		Using \eqref{28} with $k=0$, we have 
		\begin{align*}
			&\partial_t\Big(\frac{\widehat{a_+}+\widehat{a_-}}{2}\Big)|_{k=0} = 0,\\
			&\partial_t\widehat{b_j}|_{k=0}	= \frac{1}{2}\sum_\pm(\pm\nabla_x\phi\cdot\nabla_vf_\pm\mp\frac{1}{2}\nabla_x\phi\cdot vf_\pm,v_j\mu^{1/2})_{L^2_{x,v}},\\
			&\partial_t\widehat{c}|_{k=0} = \frac{1}{12}\sum_\pm(\pm\nabla_x\phi\cdot\nabla_vf_\pm\mp\frac{1}{2}\nabla_x\phi\cdot vf_\pm,(|v|^2-3)\mu^{1/2})_{L^2_{x,v}}.
		\end{align*}
		Taking integration by parts on $\nabla_v$ and using \eqref{2}, we have 
		\begin{align*}
			\partial_t\widehat{b_j}|_{k=0}	&= \frac{1}{2}\sum_\pm\Big((\mp\nabla_x\phi f_\pm,e_j\mu^{1/2}-\frac{1}{2}v_jv\mu^{1/2})_{L^2_{x,v}}\mp\frac{1}{2}(\nabla_x\phi\cdot vf_\pm,v_j\mu^{1/2})_{L^2_{x,v}}\Big)\\
			&= \frac{1}{2}\int_{\T^3}\partial^{e_j}\phi(a_--a_+)\,dx
			= \frac{1}{2}\int_{\T^3}\partial^{e_j}\phi\Delta_x\phi\,dx
			=0,
		\end{align*}
		where $e_j$ is unit vector in $\R^3$ with the $j$'s component being $1$. By using \eqref{27} and integration by parts on $\nabla_v$, 
		\begin{align*}
			\partial_t\widehat{c}|_{k=0} = -\frac{1}{12}\sum_\pm(\pm\nabla_x\phi f_\pm,2v\mu^{1/2})_{L^2_{x,v}} = -\frac{1}{6}\int_{\Z^3}{\nabla_x\phi}\cdot\overline{{G}}\,d\Sigma(k)
			=-\frac{1}{12}\partial_t\int_{\Z^3}|{E}|^2\,d\Sigma(k). 
		\end{align*}
		These equations give the conservation laws \eqref{4}. 
		Since the initial data $f_0$ satisfies \eqref{4}, 
		%	With conservation laws \eqref{4}, 
		using Plancherel's Theorem, we have 
		\begin{align*}
			&\widehat{a_+}|_{k=0}\equiv\widehat{a_-}|_{k=0}\equiv\widehat{b_j}|_{k=0}\equiv0,\\
			&\widehat{c}|_{k=0}\equiv-\frac{1}{12}\int_{\T^3}|{E}|^2\,dx.
		\end{align*}
		Taking $L^2_T$ norms on $\widehat{c}|_{k=0}$, we have 
		\begin{multline}\label{48}
			\Big(\int^T_0|\widehat{c}|_{k=0}|^2\,dt\Big)^{1/2}\lesssim \Big(\int^T_0\|E\|^4_{L^\infty_x}\,dt\Big)^{1/2}
			\lesssim \Big(\int^T_0\|\widehat{E}\|^4_{L^1_k}\,dt\Big)^{1/2}\\
			\lesssim \sup_{0\le t\le T}\|\widehat{E}\|_{L^1_k}\Big(\int^T_0\|\widehat{E}\|^2_{L^1_k}\,dt\Big)^{1/2}
			\lesssim \|\widehat{E}\|_{L^1_kL^\infty_T}\|\widehat{E}\|_{L^1_kL^2_T}.	
		\end{multline} 
		
		\medskip \noindent{\bf Case $|k|\ge 1$:}
		To obtain the estimate for $|k|\ge 1$, we use the trick in \cite[Theorem 5.1]{Duan2020}. In order to carry out the estimate in a unified way, we take a general test function as 
		\begin{align*}
			\widehat{\Phi}(t,k,v) \in C^1((0,\infty)\times\Z^3\times\R^3).
		\end{align*}
		Also, we denote $\zeta(v)$ to be a smooth function satisfying 
		\begin{align*}
			\zeta(v)\lesssim e^{-\lambda|v|^2},
		\end{align*}for some $\lambda >0$. The function $\zeta(v)$ may change from line to line.
		The following integration will be used frequently: if $p>-1$ is an even number, then 
		\begin{align}\label{25}
			\int_{\R}z^pe^{-\frac{|z|^2}{2}}dz = (p-1)!!\sqrt{2\pi}.
		\end{align}
		Taking the inner product of \eqref{40} with $\Phi$ in $L^2_v$ and integrating over $[0,T]$, we have 
		\begin{multline*}
			(\widehat{f_\pm}(T),\widehat{\Phi}(T))_{L^2_v} +(\widehat{f_\pm}(0),\widehat{\Phi}(0))_{L^2_v}-\int^T_0(\widehat{f_\pm},\partial_t\widehat{\Phi})_{L^2_v}dt-\int^T_0(\widehat{f_\pm},iv\cdot k\widehat{\Phi})_{L^2_v}dt\\\pm\int^T_0(\widehat{\nabla_x\phi}\cdot v\mu^{1/2},\widehat{\Phi})_{L^2_v}dt+\int^T_0(-L_\pm\widehat{f},\widehat{\Phi})_{L^2_v}dt = \int^T_0(\widehat{g_\pm},\widehat{\Phi})_{L^2_v}dt.
		\end{multline*}
		Then we decompose $f=\PP f+(\II-\PP)f$ to obtain 
		\begin{align*}
			-\int^T_0(\PP f,iv\cdot k\widehat{\Phi})_{L^2_v}\,dt =\sum_{j=1}^5S_j,
		\end{align*}
		where 
		\begin{align*}
			S_1 &= (\widehat{f_\pm},\widehat{\Phi})_{L^2_v}|_{t=0} - (\widehat{f_\pm},\widehat{\Phi})_{L^2_v}|_{t=T},\\
			S_2 &= \int^T_0(\widehat{f_\pm},\partial_t\widehat{\Phi})_{L^2_v}dt,\\
			S_3 &= \int^T_0((\II-\PP)f,iv\cdot k\widehat{\Phi})_{L^2_v}dt+\int^T_0(L_\pm\widehat{f},\widehat{\Phi})_{L^2_v}dt,\\
			S_4 &= \mp\int^T_0(\widehat{\nabla_x\phi}\cdot v\mu^{1/2},\widehat{\Phi})_{L^2_v}dt,\\
			S_5 &= \int^T_0(\widehat{g_\pm},\widehat{\Phi})_{L^2_v}dt.
		\end{align*}

		\medskip \noindent{\bf Step 1. Estimate of $\widehat{c}(t,k,v)$.}
		Firstly, we consider the estimate on $\widehat{c}$ for $|k|\ge 1$. We choose 
		\begin{align*}
			\widehat{\Phi}= \widehat{\Phi_c} = (|v|^2-5)iv\cdot k\widehat{\phi_c}(t,k)\mu^{1/2},
		\end{align*}
		where $\phi_c$ is defined by 
		\begin{align}\label{58}
			|k|^2\widehat{\phi_c}(t,k) = \widehat{c}(t,k). 
		\end{align}
		Note that here we assume $|k|\ge 1$ and write $\widehat{\phi_c}(t,k)=\widehat{c}(t,k)/|k|^2$. By this choice, we can deduce that 
		\begin{align*}
			&\quad\,-\int^T_0(\PP f,iv\cdot k\widehat{\Phi})_{L^2_v}\,dt\\
			&=\sum_{j,n=1}^3\int^T_0\big((\widehat{a_\pm}+\widehat{b}\cdot v+(|v|^2-3)\widehat{c})\mu^{1/2},(|v|^2-5)v_jv_n(-k_jk_n)\mu^{1/2}\widehat{\phi_c}\big)dt\\
			&=10\int^T_0(\widehat{c},|k|^2\widehat{\phi_c})dt = 
			10\int^T_0|\widehat{c}(t,k)|^2dt,
		\end{align*}
		where we used the orthogonality of the different integrands for the second equality. 
		Now we estimate the $S_j$'s. Since $|k|\ge 1$, it holds that 
		\begin{align}\label{59}
			|\widehat{\Phi_c}(t,k)|\lesssim \mu^{1/4}|k||\widehat{\phi_c}(t,k)|\lesssim\mu^{1/4}|k|^2|\widehat{\phi_c}(t,k)|=\mu^{1/4}|\widehat{c}(t,k)|.
		\end{align}
		%	and then we have $|\widehat{\Phi_c}(t,k)|_{L^2_v}\lesssim |\widehat{c}(t,k)|$. 
		Thus,
		\begin{align*}
			|S_1|&\lesssim |\widehat{f}(k,T)|^2_{L^2_v}+|\widehat{f_0}(k)|_{L^2_v}^2.
		\end{align*}
		For $S_2$, we first notice that from the third equation of \eqref{28}, 
		\begin{align*}
			|\partial_t\widehat{c}|\lesssim |k|\big(|\widehat{b}(t,k)|+|\{\I-\P\}\widehat{f}(t,k)|\big)+|(\widehat{g_+}+\widehat{g_-},\zeta)_{L^2_v}|.
		\end{align*}
		Thus, noticing $|k|\ge 1$, 
		\begin{align*}
			|S_2|&= \big|\int^T_0(\widehat{f_\pm},\partial_t\widehat{\Phi_c})_{L^2_v}dt\big| = \big|\int^T_0((\II-\PP)\widehat{f},\partial_t\widehat{\Phi_c})_{L^2_v}dt\big|\\
			&\lesssim \eta\int^T_0\frac{|\partial_t\widehat{c}|^2}{|k|^2}dt+C_\eta\int^T_0|\{\I-\P\}\widehat{f}|^2_Ddt\\
			&\lesssim \eta\int^T_0|\widehat{b}(t,k)|^2dt+\eta\int^T_0|(\widehat{g_+}+\widehat{g_-},\zeta)_{L^2_v}|^2dt+C_\eta\int^T_0|\{\I-\P\}\widehat{f}|^2_Ddt, 
		\end{align*}
		where the second equality follows from orthogonality. 
		For $S_3$, by H\"{o}lder's inequality, \eqref{58} and \eqref{59}, we have 
		\begin{align*}
			|S_3|&\lesssim \eta\int^T_0|\widehat{c}(t,k)|^2dt+C_\eta\int^T_0|\{\I-\P\}\widehat{f}(t,k)|^2_Ddt.
		\end{align*}
		For the term $S_4$, noticing $|k|\ge 1$ and \eqref{58}, 
		\begin{align*}
			|S_4|&\lesssim \int^T_0|\widehat{\nabla_x\phi}||k|^2|\widehat{\phi_c}|dt
			\lesssim C_\eta\int^T_0|\widehat{\nabla_x\phi}|^2\,dt+\eta\int^T_0|\widehat{c}(t,k)|^2dt.
		\end{align*}
		For the term $S_5$, noticing \eqref{58}, we have 
		\begin{align*}
			|\sum_\pm S_5|&\lesssim \eta\int^T_0|\widehat{c}|^2dt+C_\eta\int^T_0|(\widehat{g_+}+\widehat{g_-},\zeta)_{L^2_v}|^2dt.
		\end{align*}
		Combining the above estimates and choosing $\eta>0$ sufficiently small, we have 
		\begin{align}\notag\label{41}
			\int^T_0|\widehat{c}(t,k)|^2dt&\lesssim |\widehat{f}(k,T)|^2_{L^2_v}+|\widehat{f_0}(k)|_{L^2_v}^2+\eta\int^T_0|\widehat{b}(t,k)|^2dt+C_\eta\int^T_0|\widehat{\nabla_x\phi}|^2\,dt\\&\qquad+C_\eta\int^T_0|\{\I-\P\}\widehat{f}(t,k)|_D^2dt+C_\eta\int^T_0|(\widehat{g_+}+\widehat{g_-},\zeta)_{L^2_v}|^2dt.
		\end{align}

		\medskip \noindent{\bf Step 2. Estimate of $\widehat{b}(t,k,v)$.}
		Now we consider the estimate of $\widehat{b}$. For this purpose we choose 
		\begin{align*}
			\widehat{\Phi}=\widehat{\Phi_b}=\sum^3_{m=1}\widehat{\Phi^{j,m}_b},\ j=1,2,3,
		\end{align*}
		where 
		\begin{equation*}
			\widehat{\Phi^{j,m}_b}=\left\{\begin{aligned}
				\big(|v|^2v_mv_jik_m\widehat{\phi_j}-\frac{7}{2}(v_m^2-1)ik_j\widehat{\phi_j}\big)\mu^{1/2},\ j\neq m,\\
				\frac{7}{2}(v_j^2-1)ik_j\widehat{\phi_j}\mu^{1/2},\qquad\qquad\qquad j=m,
			\end{aligned}\right.
		\end{equation*}
		and 
		\begin{align}\label{60}
			|k|^2\widehat{\phi_j}(t,k)=\widehat{b_j}(t,k).
		\end{align}
		Under this choice we have 
		\begin{align*}
			&\quad\,-\sum^3_{m=1}\int^T_0(\PP \widehat{f},iv\cdot k\widehat{\Phi^{j,m}_b})_{L^2_v}dt\\
			&=-\sum^3_{m=1}\int^T_0\big((\widehat{a_\pm}+\widehat{b}\cdot v+(|v|^2-3)\widehat{c})\mu^{1/2},iv\cdot k\widehat{\Phi^{j,m}_b}\big)_{L^2_v}dt\\
			&=-\sum^{3}_{m=1,m\neq j}\int^T_0(v_mv_j\mu^{1/2}\widehat{b_j},|v|^2v_mv_j\mu^{1/2}(-k_m^2)\widehat{\phi_j})_{L^2_v}dt\\
			&\qquad-\sum^{3}_{m=1,m\neq j}\int^T_0(v_mv_j\mu^{1/2}\widehat{b_m},|v|^2v_mv_j\mu^{1/2}(-k_mk_j)\widehat{\phi_j})_{L^2_v}dt\\
			&\qquad+\frac{7}{2}\sum^{3}_{m=1,m\neq j}\int^T_0(v^2_m\mu^{1/2}\widehat{b_m},(v_m^2-1)\mu^{1/2}(-k_mk_j)\widehat{\phi_j})_{L^2_v}dt\\
			&\qquad-\frac{7}{2}\int^T_0(v_m^2\mu^{1/2}\widehat{b_m},(v_m^2-1)\mu^{1/2}(-k^2_j)\widehat{\phi_j})_{L^2_v}dt\\
			&= -7 \sum^3_{m=1}\int^T_0(\widehat{b_j},(-k_m^2)\widehat{\phi_j})dt=7\int^T_0|\widehat{b_j}(t,k)|^2dt.
		\end{align*}Note that $\int_{\R^3}v_m^2(v_m^2-1)\mu\,dv=2$, $\int_{\R^3}v_m^2(v_j^2-1)\mu\,dv=0$, $\int_{\R^3}v_m^2v_j^2|v|^2\mu\,dv=7$ when $m\neq j$. 
		Since $|k|\ge 1$, it holds that 
		\begin{align}\label{61}
			|\widehat{\Phi^{j,m}_b}(t,k)|\lesssim \mu^{1/4}|k||\widehat{\phi_j}(t,k)|\lesssim\mu^{1/4}|k|^2|\widehat{\phi_j}(t,k)|=\mu^{1/4}|\widehat{b}(t,k)|,
		\end{align}
		and then we have $|\widehat{\Phi^{j,m}_b}(t,k)|_{L^2_v}\lesssim |\widehat{b}(t,k)|$. Thus,
		\begin{align*}
			|S_1|&\lesssim |\widehat{f}(k,T)|^2_{L^2_v}+|\widehat{f_0}(k)|_{L^2_v}^2.
		\end{align*}
		For $S_2$, we first notice that from the second equation of \eqref{28}, 
		\begin{align*}
			|\partial_t\widehat{b_j}|\lesssim |k|\big(|\widehat{a_+}+\widehat{a_-}|+|\widehat{c}|+|\{\I-\P\}\widehat{f}(t,k)\big)+|(\widehat{g_+}+\widehat{g_-},\zeta)_{L^2_v}|.
		\end{align*}
		Thus, noticing $|k|\ge 1$, 
		\begin{align*}
			|S_2|&\le \int^T_0|(\widehat{(\II-\PP)f},\partial_t\widehat{\Phi^{j,m}_b})_{L^2_v}|dt + \int^T_0|(\PP\widehat{f},\partial_t\widehat{\Phi^{j,m}_b})_{L^2_v}|dt\\
			&\lesssim \eta\int^T_0\frac{|\partial_t\widehat{b_j}|^2}{|k|^2}dt+C_\eta\int^T_0|\{\I-\P\}\widehat{f}|^2_Ddt + C_\eta\int^T_0|\widehat{c}(t,k)|^2dt\\
			&\lesssim \eta\int^T_0|\widehat{a_+}+\widehat{a_-}|^2dt+\eta\int^T_0|(\widehat{g_+}+\widehat{g_-},\zeta)_{L^2_v}|^2dt\\&\qquad+C_\eta\int^T_0|\{\I-\P\}\widehat{f}|^2_Ddt+C_\eta\int^T_0|\widehat{c}(t,k)|^2dt.
		\end{align*}
		Here, in the second inequality, we use the orthogonality for the term $\int^T_0|(\PP\widehat{f},\partial_t\widehat{\Phi^{j,m}_b})_{L^2_v}|dt$ in the $v$-integration. 
		For $S_3$, by H\"{o}lder's inequality, \eqref{60} and \eqref{61}, we have 
		\begin{align*}
			|S_3|&\lesssim \eta\int^T_0|\widehat{b}(t,k)|^2dt+C_\eta\int^T_0|\{\I-\P\}\widehat{f}(t,k)|^2_Ddt.
		\end{align*}
		For the term $S_4$, noticing $|k|\ge 1$ and \eqref{60}, 
		\begin{align*}
			|S_4|&\lesssim \int^T_0|\widehat{\nabla_x\phi}||k|^2|\widehat{\phi_j}|dt
			\lesssim C_\eta\int^T_0|\widehat{\nabla_x\phi}|^2\,dt+\eta\int^T_0|\widehat{b}(t,k)|^2dt.
		\end{align*}
		For the term $S_5$, noticing \eqref{60}, we have 
		\begin{align*}
			|\sum_\pm S_5|&\lesssim \eta\int^T_0|\widehat{b}|^2dt+C_\eta\int^T_0|(\widehat{g_+}+\widehat{g_-},\zeta)_{L^2_v}|^2dt.
		\end{align*}
		Combining the above estimates and choosing $\eta>0$ sufficiently small, we have 
		\begin{multline}\label{42}
			\int^T_0|\widehat{b}(t,k)|^2dt\lesssim |\widehat{f}(k,T)|^2_{L^2_v}+|\widehat{f_0}(k)|_{L^2_v}^2+\eta\int^T_0|\widehat{a_+}+\widehat{a_-}|^2dt+C_\eta\int^T_0|\widehat{c}(t,k)|^2dt\\+C_\eta\int^T_0|\{\I-\P\}\widehat{f}(t,k)|_D^2dt+C_\eta\int^T_0|\widehat{\nabla_x\phi}|^2\,dt+C_\eta\int^T_0|(\widehat{g_+}+\widehat{g_-},\zeta)_{L^2_v}|^2dt.
		\end{multline}

		\medskip \noindent{\bf Step 3. Estimate of $\widehat{a_+}+\widehat{a_-}$.}
		Now we consider $\widehat{a_+}+\widehat{a_-}$. We set 
		\begin{align*}
			\widehat{\Phi} = \widehat{\Phi_a} = (|v|^2-10)iv\cdot k\widehat{\phi_a}\mu^{1/2},
		\end{align*}
		where $\widehat{\phi_a}(t,k)=(\widehat{\phi_a}_+(t,k),\widehat{\phi_a}_-(t,k))$ is a solution to 
		\begin{align*}
			|k|^2\widehat{\phi_a}_+(t,k)= |k|^2\widehat{\phi_a}_-(t,k)  = \widehat{a_+}+\widehat{a_-}.
		\end{align*}
		For this choice we have 
		\begin{align*}
			&\quad\,-\sum_\pm\int^T_0(\PP f,iv\cdot k\widehat{\Phi_{a\pm}})_{L^2_v}dt\\
			&=-\sum_\pm\sum_{j,n}\int^T_0\big((\widehat{a_\pm}+\widehat{b}\cdot v+(|v|^2-3)\widehat{c})\mu^{1/2},v_jv_n(|v|^2-10)\mu^{1/2}(-k_jk_n)\widehat{\phi_{a\pm}}\big)_{L^2_v}dt\\
			&= 5\sum_\pm\int^T_0\widehat{a_\pm}|k|^2\overline{\widehat{\phi_a}_\pm}dt
			= 5\int^T_0|\widehat{a_+}+\widehat{a_-}|^2dt.
		\end{align*}
		Note that $((|v|^2-3)\mu^{1/2},v_jv_n(|v|^2-10)\mu^{1/2})_{L^2_v}=0$. The estimate of $S_1$ is similar to the previous case. 
		For $S_2$, we first notice that from the second equation of \eqref{28}, 
		\begin{align*}
			|\partial_t(\widehat{a_+}+\widehat{a_-})|\lesssim |k||\widehat{b}|.
		\end{align*}
		Thus, noticing $|k|\ge 1$, 
		\begin{align*}
			|S_2|&\le \int^T_0|(\widehat{(\II-\PP)f},\partial_t\widehat{\Phi_a})_{L^2_v}|dt + \int^T_0|(\PP\widehat{f},\partial_t\widehat{\Phi_a})_{L^2_v}|dt\\
			&\lesssim \int^T_0\frac{|\partial_t(\widehat{a_+}+\widehat{a_-})|^2}{|k|^2}dt+\int^T_0|\{\I-\P\}\widehat{f}|^2_Ddt + \int^T_0|\widehat{b}(t,k)|^2dt\\
			&\lesssim \int^T_0|\widehat{b}(t,k)|^2dt+\int^T_0|\{\I-\P\}\widehat{f}|^2_Ddt.
		\end{align*}
		Here in the second line we used the orthogonality in the $v$-integration for the term  $\int^T_0|(\PP\widehat{f},\partial_t\widehat{\Phi_a})_{L^2_v}|dt$. 
		Similar to the previous case, we have 
		\begin{align*}
			|S_3|&\lesssim \eta\int^T_0|\widehat{a_+}+\widehat{a_-}|^2dt+C_\eta\int^T_0|\{\I-\P\}\widehat{f}(t,k)|^2_Ddt.
		\end{align*}
		For the term $S_4$, noticing the sign $\pm$, we have $\sum_\pm S_4=0$. 
		For $S_5$, we have 
		\begin{align*}
			|\sum_\pm S_5|&\lesssim \eta\int^T_0|\widehat{a_+}+\widehat{a_-}|^2dt+C_\eta\int^T_0|(\widehat{g_+}+\widehat{g_-},\zeta)_{L^2_v}|^2dt.
		\end{align*}
		Combining the above estimates and choosing $\eta>0$ sufficiently small, we have 
		\begin{align}\notag\label{43}
			\int^T_0\sum_\pm|\widehat{a_+}+\widehat{a_-}|^2dt&\lesssim |\widehat{f}(k,T)|^2_{L^2_v}+|\widehat{f_0}(k)|_{L^2_v}^2+\int^T_0|\widehat{b}(t,k)|^2dt\\&\qquad+\int^T_0|\{\I-\P\}\widehat{f}(t,k)|_D^2dt+\int^T_0|(\widehat{g_+}+\widehat{g_-},\zeta)_{L^2_v}|^2dt.
		\end{align}

		\medskip \noindent{\bf Step 4. Estimate of $\widehat{a_+}-\widehat{a_-}$ and $\widehat{E}$.}
		Set 
		\begin{align*}
			\widehat{\Phi} = \widehat{\Phi_a} = (|v|^2-10)iv\cdot k\widehat{\phi_a}\mu^{1/2},
		\end{align*}
		where $\widehat{\phi_a}(t,k)=(\widehat{\phi_a}_+(t,k),\widehat{\phi_a}_-(t,k))$ is a solution to 
		\begin{align*}%\label{50}
			|k|^2\widehat{\phi_a}_+(t,k) = \widehat{a_+}-\widehat{a_-},\quad |k|^2\widehat{\phi_a}_-(t,k) = \widehat{a_-}-\widehat{a_+}.
		\end{align*}
		Then the macroscopic part reduces to 
		\begin{align*}
			&\quad\,-\sum_\pm\int^T_0(\PP f,iv\cdot k\widehat{\Phi_{a\pm}})_{L^2_v}dt\\
			&=-\sum_\pm\sum_{j,n}\int^T_0\big((\widehat{a_\pm}+\widehat{b}\cdot v+(|v|^2-3)\widehat{c})\mu^{1/2},v_jv_n(|v|^2-10)\mu^{1/2}(-k_jk_n)\widehat{\phi_a}_\pm\big)_{L^2_v}dt\\
			&= 5\sum_\pm\int^T_0\widehat{a_\pm}|k|^2\overline{\widehat{\phi_a}_\pm}dt
			= 5\int^T_0|\widehat{a_+}-\widehat{a_-}|^2dt.
		\end{align*}
		For the part $S_2$, by \eqref{27}, we have
		\begin{align*}
			|\partial_t\widehat{\phi_{a\pm}}|\le |k|^{-2}|\widehat{\nabla_x\cdot G}|\le |k|^{-1}|\widehat{G}|,
		\end{align*} and hence, 
		\begin{align*}
			\sum_\pm |S_2| &\le \sum_\pm\int^T_0|(\widehat{\PP f},(|v|^2-10)iv\cdot k\partial_t\widehat{\phi_a}\mu^{1/2})_{L^2_v}|dt\\
			&\qquad+\sum_\pm\int^T_0|({(\II-\PP)\widehat{f}},(|v|^2-10)iv\cdot k\partial_t\widehat{\phi_a}\mu^{1/2})_{L^2_v}|dt\\
			&\lesssim \eta^2\int^T_0|\widehat{b}|^2\,dt + C_\eta\int^T_0|\{\I-\P\}\widehat{f}|^2_D\,dt.
		\end{align*}
		For the part $S_4$, we observe from \eqref{2} that  
		\begin{align*}
			\sum_\pm S_4 &= \sum_\pm\sum_{j,n}\mp\int^T_0(\widehat{\partial^{e_j}\phi} \mu^{1/2},(|v|^2-10)iv_jv_n k_n\widehat{\phi_{a\pm}}\mu^{1/2})_{L^2_v}dt\\
			&=5\sum_\pm\sum_{j}\mp\int^T_0(\widehat{\partial^{e_j}\phi}\,|\,ik_j{\widehat{\phi_{a\pm}}})dt\\
			&= 10\sum_j\frac{ik_j}{|k|^2}\int^T_0(\widehat{\partial^{e_j}\phi}\,|\,{\widehat{a_+}-\widehat{a_-}})dt\\
			&= -10 \int^T_0(\widehat{\phi}\,|\,\widehat{-\Delta_x\phi})dt= -10\int^T_0|\widehat{\nabla_x\phi}|^2dt.
		\end{align*}
		This gives the dissipation rate of $\widehat{\nabla_x\phi}(t,k)$ when $|k|\ge 1$. Notice that when $k=0$, $\widehat{\nabla_x\phi}=ik\widehat{\phi}=0$. 
		The estimates on terms $S_1,S_3,S_5$ are similar to Step 3 and we omit it for brevity. 
		Therefore, combining the above estimates we have
		\begin{multline}\label{41a}
			\int^T_0|\widehat{a_+}-\widehat{a_-}|^2dt + \int^T_0|\widehat{\nabla_x\phi}|^2dt
			\lesssim|\widehat{f}(T,k)|^2_{L^2_v}+|\widehat{f_0}(k)|_{L^2_v}^2+\eta^2\int^T_0|\widehat{b}|^2\,dt\\
			+C_\eta\int^T_0|\{\I-\P\}\widehat{f}(t,k)|_D^2dt+\int^T_0|(\widehat{g_+}+\widehat{g_-},\zeta)_{L^2_v}|^2dt.
		\end{multline}

		\medskip \noindent{\bf Step 5. Energy estimate.}
		Now taking combination $\eqref{41a}+\kappa\times\eqref{41}+\kappa^2\times\eqref{42}+\kappa^3\times\eqref{43}$ and choosing $\eta<<\kappa<<1$, we have for $|k|\ge 1$ that 
		\begin{multline*}
			\int^T_0\big|[\wh{a_+},\wh{a_-},\wh{b},\wh{c}]\big|^2dt+\int^T_0|\widehat{\nabla_x\phi}|^2\,dt\lesssim|\widehat{f}(k,T)|^2_{L^2_v}+|\widehat{f_0}(k)|_{L^2_v}^2\\+ \int^T_0|\{\I-\P\}\widehat{f}(t,k)|_D^2dt+\int^T_0|(\widehat{g_+}+\widehat{g_-},\zeta)_{L^2_v}|^2dt.
		\end{multline*}
		Note that $2|\widehat{a_+}|^2+2|\widehat{a_-}|^2=|\widehat{a_+}+\widehat{a_-}|^2+|\widehat{a_+}-\widehat{a_-}|^2$. Together with \eqref{48}, taking the square root and summation over $k\in\Z^3$, we have 
		%\begin{multline}\label{51}
		%	\|[\wh{a_+},\wh{a_-},\wh{b},\wh{c}]\|_{L^1_kL^2_T}+\|\widehat{E}\|_{L^1_kL^2_T}\lesssim \|\widehat{f}\|_{L^1_kL^\infty_TL^2_v}+\|\widehat{f_0}\|_{L^1_kL^2_v} +C_\eta\|\{\I-\P\}\widehat{f}\|_{L^1_kL^2_TL^2_{D}} \\+\int_{\Z^3_k}\Big(\int^T_0|(\widehat{g_+}+\widehat{g_-},\zeta)_{L^2_v}|^2dt\Big)^{1/2}d\Sigma(k).
		%\end{multline}
		%Therefore, together with \eqref{48}, we obtain 
		%\begin{align}\label{47}\notag
		%	\|(\widehat{a_+}+\widehat{a_-},\widehat{b},\widehat{c})\|_{L^1_kL^2_T}&\lesssim \|\widehat{f}\|_{L^1_kL^\infty_TL^2_v}+\|\widehat{f_0}\|_{L^1_kL^2_v}+\|\{\I-\P\}\widehat{f}\|_{L^1_kL^2_TL^2_{D}}+\|\widehat{E}\|_{L^1_kL^2_T}\\&\qquad +\|\widehat{E}\|_{L^1_kL^\infty_T}\|\widehat{E}\|_{L^1_kL^2_T}+\int_{\Z^3_k}\Big(\int^T_0|(\widehat{g_\pm},\zeta)_{L^2_v}|^2dt\Big)^{1/2}d\Sigma(k).
		%\end{align}
		%In order to eliminate the term $\|\widehat{E}\|_{L^1_kL^2_T}$, we need to find the dissipation rate of $\widehat{E}$. 
		%	Now taking the combination $\eqref{51}+\kappa\times\eqref{47}$ with sufficiently small $\kappa>0$ and then letting $\eta$ suitably small, noticing that $2|\widehat{a_+}|^2+2|\widehat{a_-}|^2=|\widehat{a_+}+\widehat{a_-}|^2+|\widehat{a_+}-\widehat{a_-}|^2$, we have 
		\begin{multline}\label{52}
			\|(\widehat{a_+},\widehat{a_-},\widehat{b},\widehat{c})\|_{L^1_kL^2_T}+\|\widehat{E}\|_{L^1_kL^2_T}\lesssim \|\widehat{f}\|_{L^1_kL^\infty_TL^2_v}+\|\widehat{f_0}\|_{L^1_kL^2_v}+\|\{\I-\P\}\widehat{f}\|_{L^1_kL^2_TL^2_{D}}\\ +\|\widehat{E}\|_{L^1_kL^\infty_T}\|\widehat{E}\|_{L^1_kL^2_T}+\int_{\Z^3_k}\Big(\int^T_0|(\widehat{g_+}+\widehat{g_-},\zeta)_{L^2_v}|^2dt\Big)^{1/2}d\Sigma(k).
		\end{multline}

		For the last term, we will use the trick as \eqref{48}. Noticing \eqref{40a}, we estimate the term in $g_\pm$ one by one. For the term $(\nabla_x\phi\cdot \nabla_vf_\pm)^\wedge$, noticing $\zeta(v)$ is smooth and has exponential decay, we take integration by parts with respect to $v$ to obtain  
		\begin{align}\label{84b}\notag
			&\quad\,\int_{\Z^3_k}\Big(\int^T_0\big|\big((\nabla_x\phi\cdot \nabla_vf_\pm)^\wedge,\zeta(v)\big)_{L^2_v}\big|^2 dt\Big)^{1/2}d\Sigma(k)\\
			&\lesssim\notag \int_{\Z^3_k}\Big(\int^T_0\big(\int_{\Z^3}|\widehat{E}(k-l)||\widehat{f}(l)|_{L^2_{D}}\,d\Sigma(l)\big)^2dt\Big)^{1/2}d\Sigma(k)\\
			&\lesssim\notag \int_{\Z^3_k}\int_{\Z^3}\Big(\int^T_0|\widehat{E}(k-l)|^2|\widehat{f}(l)|_{L^2_{D}}^2dt\Big)^{1/2}d\Sigma(l)d\Sigma(k)\\
			&\lesssim\notag \int_{\Z^3_k}\int_{\Z^3}\sup_{0\le t\le T}|\widehat{E}(k-l)|\Big(\int^T_0|\widehat{f}(l)|_{L^2_{D}}^2dt\Big)^{1/2}d\Sigma(l)d\Sigma(k)\\
			&\lesssim \|\widehat{E}\|_{L^1_kL^\infty_T}\|\widehat{f}\|_{L^1_kL^2_TL^2_{D}}.
		\end{align}
		Similarly, for the term $\frac{1}{2}((\nabla_x\phi\cdot vf_\pm)^\wedge$, we have 
		\begin{align}\label{2.23}
			\quad\,\int_{\Z^3_k}\Big(\int^T_0\big|\big(\frac{1}{2}((\nabla_x\phi\cdot vf_\pm)^\wedge,\zeta(v)\big)_{L^2_v}\big|^2 dt\Big)^{1/2}d\Sigma(k)
			&\lesssim \|\widehat{E}\|_{L^1_kL^\infty_T}\|\widehat{f}\|_{L^1_kL^2_TL^2_{D}}.
		\end{align}
		For the term $\widehat{\Gamma_\pm(f,f)}$, we have from \eqref{35a} that 
		\begin{align}\label{2.24}\notag
			&\quad\,\int_{\Z^3_k}\Big(\int^T_0\big|\big(\widehat{\Gamma_\pm(f,f)},\zeta(v)\big)_{L^2_v}\big|^2dt\Big)^{1/2}d\Sigma(k)\\
			&=\notag \int_{\Z^3_k}\Big(\int^T_0\big|\int_{\Z^3}\big(\Gamma_\pm(\widehat{f}(k-l),\widehat{f}(l)),\zeta(v)\big)_{L^2_v}d\Sigma(l)\big|^2dt\Big)^{1/2}d\Sigma(k)\\
			&\notag\lesssim \int_{\Z^3_k}\Big(\int^T_0\big(\int_{\Z^3}|\widehat{f}(k-l)|_{L^2_v}|\widehat{f}(l)|_{L^2_{D}}d\Sigma(l)\big)^{2}dt\Big)^{1/2}d\Sigma(k)\\
			&\notag\lesssim \int_{\Z^3_k}\int_{\Z^3}\sup_{0\le t\le T}|\widehat{f}(k-l)|_{L^2_v}\Big(\int^T_0|\widehat{f}(l)|_{L^2_{D}}^2dt\Big)^{1/2}d\Sigma(l)d\Sigma(k)\\
			&\lesssim \|\widehat{f}\|_{L^1_kL^\infty_TL^2_v}\|\widehat{f}\|_{L^1_kL^2_TL^2_{D}}.
		\end{align}
		Plugging the above estimate into \eqref{52}, we obtain \eqref{53} and complete the proof of Theorem \ref{Thm21}. 
		%	\begin{multline*}
		%		\quad\,\|(\widehat{a_+},\widehat{a_-},\widehat{b},\widehat{c})\|_{L^1_kL^2_T}+\|\widehat{E}\|_{L^1_kL^2_T}\lesssim \|\widehat{f}\|_{L^1_kL^\infty_TL^2_v}+\|\widehat{f_0}\|_{L^1_kL^2_v}+\|\{\I-\P\}\widehat{f}\|_{L^1_kL^2_TL^2_{D}}\\ +\|\widehat{E}\|_{L^1_kL^\infty_T}\|\widehat{E}\|_{L^1_kL^2_T}	+\big(\|\widehat{E}\|_{L^1_kL^\infty_T}+\|\widehat{f}\|_{L^1_kL^\infty_TL^2_v}\big)\|\widehat{f}\|_{L^1_kL^2_TL^2_{D}}.
		%	\end{multline*}
		\qe\end{proof}

	\section{Macroscopic estimates for finite channel}\label{MacroFinite}
	In this section, we write $\widehat{\cdot}$ to denote the Fourier transform with respect to $\bar{x}\in \T^2$. We will derive the macroscopic estimates in the case of finite channel. 
	Consider the following problem  
	\begin{align}\label{40abc}
		\partial_t{f_\pm}+v_1\partial_{x_1}{f_\pm} &+\bar{v}\cdot\nabla_{\bar{x}}f_\pm\pm {\nabla_x\phi}\cdot v\mu^{1/2} - L_\pm {f} = {g_\pm},
	\end{align}
	with initial data $(f_0,E_0)$ and boundary condition 
	\begin{equation}\label{specular_finite2}\begin{aligned}
			\widehat{f}(t,-1,\bar{k},v_1,\bar{v})|_{v_1>0} &= \widehat{f}(t,-1,\bar{k},-v_1,\bar{v}),\\
			\widehat{f}(t,1,\bar{k},v_1,\bar{v})|_{v_1<0} &= \widehat{f}(t,1,\bar{k},-v_1,\bar{v}),
		\end{aligned}
	\end{equation} where 
	\begin{align*}
		g_\pm = \pm&\nabla_x\phi\cdot\nabla_vf_\pm\mp\frac{1}{2}\nabla_x\phi\cdot vf_\pm+\Gamma_\pm(f,f),
	\end{align*}
	and the potential is determined by Poisson equation:
	\begin{align}
		-\Delta_x\phi &= a_+-a_-,\label{40abcd}
	\end{align}
	with the $zero$ Neumann boundary condition 
	\begin{align}\label{97a}
		\partial_{x_1}\phi = 0,\ \text{ on } x_1=\pm 1. 
	\end{align}
	As in the torus case, we denote $\zeta(v)$ to be a smooth function satisfying $
	\zeta(v)\lesssim e^{-\lambda|v|^2},$ for some $\lambda >0$. The function $\zeta(v)$ may change from line to line. Then we have the following macroscopic estimate in anisotropic case.

	%Denote 
	%\begin{align*}
	%E_{\bar{k}}(\widehat{u_{\pm}}) = \sum_{\pm}\int^T_0\int_{v_1<0}|v_1|^{-1}\Big(|\partial_t\widehat{u_{\pm}}|^2+|\bar{k}\cdot\bar{v}|^2|\widehat{u_{\pm}}|^2+|\widehat{\nabla_x\phi}(1)v\mu^{1/2}|^2+|L\widehat{u_{\pm}}|^2+|\widehat{g}(1)|^2\Big)\,dvdt
	%\end{align*}

	%We define the boundary integral functionals:
	%\begin{align*}
	%	|G^+_w(\widehat{f})|^2 = \int^T_0\int_{v_1>0}|v_1||\widehat{wf}(t,1,v)|^2\,dvdt+\int^T_0\int_{v_1<0}|v_1||\widehat{wf}(t,-1,v)|^2\,dvdt,
	%\end{align*}
	%for particles with outgoing velocities on $x_1=\pm 1$, and 
	%\begin{align*}
	%	|G^-_w(\widehat{f})|^2 = \int^T_0\int_{v_1<0}|v_1||\widehat{wf}(t,1,v)|^2\,dvdt+\int^T_0\int_{v_1>0}|v_1||\widehat{wf}(t,-1,v)|^2\,dvdt,
	%\end{align*}
	%for particles with incoming velocities on $x_1=\pm 1$.

	\begin{Thm}\label{Thm41f}
		Let $\gamma\ge -3$ for Landau case, $\gamma>\max\{-3,-2s-3/2\}$ for Boltzmann case and $T>0$. Let $f$ be the solution of \eqref{40abc}, \eqref{specular_finite2}, \eqref{40abcd} and \eqref{97a} in finite channel with initial data satisfying \eqref{conservation_finite}. Then 
		\begin{align*}
			&\quad\,\sum_{|\alpha|\le 1} \|\partial^\alpha(\widehat{a_+},\widehat{a_-},\widehat{b},\widehat{c})\|_{L^1_{\bar{k}}L^2_TL^2_{x_1}}+\sum_{|\alpha|\le 1}\|\widehat{\partial^\alpha E}\|_{L^1_{\bar{k}}L^2_TL^2_{x_1}}\\
			\notag&\lesssim \sum_{|\alpha|\le 1}\big(\|\widehat{\partial^\alpha f}(T)\|_{L^1_{\bar{k}}L^2_{x_1,v}}+\|\widehat{\partial^\alpha f_0}\|_{L^1_{\bar{k}}L^2_{x_1,v}}\big)+\sum_{|\alpha|\le1}\|\{\I-\P\}\widehat{\partial^\alpha f}\|_{L^1_{\bar{k}}L^2_TL^2_{x_1}L^2_D}\\
			&\qquad+\sum_{|\alpha|\le 1}\big(\|\widehat{\partial^\alpha  E}\|_{L^1_{\bar{k}}L^\infty_TL^2_{x_1}}+\|\widehat{\partial^\alpha f}\|_{L^1_{\bar{k}}L^\infty_TL^2_{x_1}L^2_v}\big)\sum_{|\alpha|\le 1}\|\widehat{\partial^\alpha f}\|_{L^1_{\bar{k}}L^2_TL^2_{x_1}L^2_{D}}\\
			&\qquad+\|\wh{E}\|_{L^1_{\bar{k}}L^\infty_TL^2_{x_1}}\|\wh{E}\|_{L^1_{\bar{k}}L^2_TL^2_{x_1}}.
		\end{align*}
	\end{Thm}
	\begin{proof}
		
		Let $|\alpha|\le 1$. Acting $\partial :=\partial^\alpha$ to \eqref{40abc} and taking the Fourier transform with respect to $\bar{x}$, we have 
		\begin{align}\label{40dh}
			\partial_t\widehat{\partial f_\pm}+v_1\partial_{x_1}\widehat{\partial f_\pm} &+i\bar{v}\cdot{\bar{k}}\widehat{\partial f_\pm}\pm (\partial{\nabla_x\phi})^\wedge\cdot v\mu^{1/2} - L_\pm \widehat{\partial f} = \widehat{\partial g_\pm},
			%		\\
			%		-\Delta_x\partial\phi &= \int_{\T^3}(\partial f_+-\partial f_-)\mu^{1/2}\,dx,
		\end{align}
		%with the inflow boundary condition
		%\begin{align}\label{inflow}
		%\widehat{f}(t,-1,\bar{k},v)|_{v_1>0} = \widehat{u_+}(t,\bar{k},v),\quad\widehat{f}(t,1,\bar{k},v)|_{v_1<0} = \widehat{u_{-}}(t,\bar{k},v),
		%\end{align}
		with the specular reflection condition 
		\begin{equation}\label{reflection}\begin{aligned}
				\widehat{f}(t,-1,\bar{k},v_1,\bar{v})|_{v_1>0} &= \widehat{f}(t,-1,\bar{k},-v_1,\bar{v}),\\
				\widehat{f}(t,1,\bar{k},v_1,\bar{v})|_{v_1<0} &= \widehat{f}(t,1,\bar{k},-v_1,\bar{v}).
			\end{aligned}
		\end{equation}
		%with symmetric condition:
		%\begin{align}\label{symmetric}
		%{f}(x_1,v_1)={f}(-x_1,-v_1),\text{ for } x_1\in[-1,1].
		%\end{align}

		Let $\widehat{\Phi}(t,x_1,\bar{k},v)\in C^1((0,+\infty)\times(-1,1)\times\R^3)$ with $\bar{k}=(k_2,k_3)\in\Z^2$ be a test function. Taking the inner product of $\widehat{\Phi}(t,x_1,\bar{k},v)$ and \eqref{40dh} with respect to $(x_1,v)$ and integrating the resulting identity with respect to $t$ over $[0,T]$ for any $T>0$, we obtain 
		\begin{multline*}
			\quad\,(\widehat{\partial f_\pm},\widehat{\Phi})_{L^2_{x_1,v}}(T)-(\widehat{\partial f_\pm},\widehat{\Phi})_{L^2_{x_1,v}}(0) - \int^T_0(\widehat{\partial f_\pm},\partial_t\widehat{\Phi})_{L^2_{x_1,v}}\,dt\\
			-\int^T_0(\widehat{\partial f_\pm},v\cdot\widehat{\nabla_{x_1,\bar{x}}\Phi})_{L^2_{x_1,v}} \,dt
			+\int^T_0(v_1\widehat{\partial f_\pm}(1),\widehat{\Phi}(1))_{L^2_v}\,dt- \int^T_0(v_1\widehat{\partial f_\pm}(-1),\widehat{\Phi}(-1))_{L^2_v}\,dt\\
			\pm \int^T_0((\partial{\nabla_x\phi})^\wedge\cdot v\mu^{1/2},\widehat{\Phi})_{L^2_{x_1,v}}\,dt - \int^T_0(L_\pm \widehat{\partial f},\widehat{\Phi})_{L^2_{x_1,v}}\,dt = \int^T_0(\widehat{\partial g_\pm},\widehat{\Phi})_{L^2_{x_1,v}}\,dt.
		\end{multline*}
		Using the decomposition $\widehat{f_\pm}=\P\widehat{f_\pm}+\{\I-\P\}\widehat{f_\pm}$, we have 
		\begin{align}\label{100a}
			-\int^T_0(\widehat{\partial \P_\pm f},v\cdot\widehat{\nabla_{x_1,\bar{x}}\Phi})_{L^2_{x_1,v}} \,dt = \sum_{j=1}^6S_j,
		\end{align}
		where $S_j$ are defined by 
		\begin{align*}
			S_1 &= -(\widehat{\partial f_\pm},\widehat{\Phi})_{L^2_{x_1,v}}(T)+(\widehat{\partial f_\pm},\widehat{\Phi})_{L^2_{x_1,v}}(0),\quad S_2 = \int^T_0(\widehat{\partial f_\pm},\partial_t\widehat{\Phi})_{L^2_{x_1,v}}\,dt,\\
			S_3 &= \int^T_0(({\partial (\II-\PP)f})^\wedge,v\cdot\widehat{\nabla_{x_1,\bar{x}}\Phi})_{L^2_{x_1,v}} \,dt,\\
			S_4&= \int^T_0(L_\pm \widehat{\partial f},\widehat{\Phi})_{L^2_{x_1,v}}\,dt+\int^T_0(\widehat{\partial g_\pm},\widehat{\Phi})_{L^2_{x_1,v}},\\
			S_5 &= \mp\int^T_0((\partial{\nabla_x\phi})^\wedge\cdot v\mu^{1/2},\widehat{\Phi})_{L^2_{x_1,v}}\,dt,\\
			S_6 &= -\int^T_0(v_1\widehat{\partial f_\pm}(1),\widehat{\Phi}(1))_{L^2_v}\,dt+ \int^T_0(v_1\widehat{\partial f_\pm}(-1),\widehat{\Phi}(-1))_{L^2_v}\,dt.
		\end{align*}

		\medskip \noindent{\bf Step 1. Estimate on $\widehat{c}(t,x_1,\bar{k})$.} We choose the following test function 
		\begin{align*}
			\widehat{\Phi} = \widehat{\Phi_c} = (|v|^2-5)\big(v\cdot\widehat{\nabla_{x_1,\bar{x}}\phi_c}(t,x_1,\bar{k})\big)\mu^{1/2},
		\end{align*}
		where $\phi_c$ solves 
		\begin{equation}\label{4.120a}\left\{
			\begin{aligned}
				&-\partial^2_{x_1}\widehat{\phi_c}+|\bar{k}|^2\widehat{\phi_c} = \widehat{\partial c},\\
				&\widehat{\phi_c}(\pm 1,\bar{k})=0, \qquad \text{ if }\partial=\partial_{x_1},\\
				&\partial_{x_1}\widehat{\phi_c}(\pm 1,\bar{k})=0, \quad\text{ if }\partial=I,\partial_{x_2},\partial_{x_3}.
			\end{aligned}\right.
		\end{equation}
		Then by standard elliptic estimate, we have 
		\begin{align}\label{99cc}
			\|\pa^2_{x_1}\widehat{\phi_c}\|_{L^2_{x_1}}+|\bar{k}|\|\pa_{x_1}\widehat{ \phi_c}\|_{L^2_{x_1}}+|\bar{k}|^2\|\widehat{\phi_c}\|_{L^2_{x_1}}\lesssim\|\widehat{\partial c}\|_{L^2_{x_1}}.
		\end{align}
		When $\pa=I$ and $\bar{k}=0$, \eqref{4.120a} is a pure Neumann boundary problem. However, the mean of $\wh{c}|_{\bar{k}=0}$ doesn't vanish: $\int_{-1}^1\wh{c}(x_1,0)\,dx_1\neq 0$ and, we can't find a solution. 
		In order to derive the dissipation estimates for $\wh{c}|_{\bar{k}=0}$, we apply Poincar\'{e}'s inequality and \eqref{conservation_finite} to obtain that 
		\begin{align*}
			\|\wh{c}|_{\bar{k}=0}\|_{L^2_{x_1}} \lesssim \|\pa_{x_1}\wh{c}|_{\bar{k}=0}\|_{L^2_{x_1}}+\Big|\int_{-1}^1\wh{c}|_{\bar{k}=0}dx_1\Big| \lesssim \|\pa_{x_1}\wh{c}|_{\bar{k}=0}\|_{L^2_{x_1}} + \|E\|_{L^2_x}^2. 
		\end{align*}
	Similar to the case of torus, we can estimate the term $\|E\|_{L^2_x}^2$ by 
	\begin{align*}
		\int^T_0\|E\|_{L^2_x}^4\,dt\le \int^T_0\|{E}\|_{L^2_{x_1}L^\infty_{\bar{x}}}^4\,dt
		\le \int^T_0\|\wh{E}\|_{L^2_{x_1}L^1_{\bar{k}}}^4\,dt \le \sup_{0\le t\le T}\|\wh{E}\|_{L^2_{x_1}L^1_{\bar{k}}}^2\int^T_0\|\wh{E}\|^2_{L^2_{x_1}L^1_{\bar{k}}}\,dt.
	\end{align*}
Thus, by Young's inequality $\|\cdot\|_{L^pL^q}\lesssim \|\cdot\|_{L^qL^p}$ $(p\ge q\ge 1)$, we have 
	\begin{align}\label{c0}
		\int^T_0\|\wh{c}|_{\bar{k}=0}\|^2_{L^2_{x_1}}\,dt \lesssim
		\int^T_0\|\pa_{x_1}\wh{c}|_{\bar{k}=0}\|^2_{L^2_{x_1}}\,dt + \|\wh{E}\|_{L^1_{\bar{k}}L^\infty_TL^2_{x_1}}^2\|\wh{E}\|^2_{L^1_{\bar{k}}L^2_TL^2_{x_1}}. 
	\end{align}

%		 ensure the existence, which follows from \eqref{conservation_finite}. In this case, the solution to \eqref{4.120a} is unique up to a constant. Thus, we can add zero mean condition $\int_{-1}^1\wh{\phi_c}\,dx_1=0$ to find a unique solution to \eqref{4.120a}. 
%		Then by Poincar\'{e}'s inequality and
To the end of this step, we assume that either $\pa\neq I$ or $|\bar{k}|\neq 0$. 
Next we let $\pa=I,\pa_{x_2},\pa_{x_3}$.  
		Taking inner product of \eqref{4.120a} with $\wh{\phi_c}$ over $x_1\in[-1,1]$, we have 
		\begin{align*}
%			\|\wh{\phi_c}|_{\bar{k}=0}\|_{L^2_{x_1}}\lesssim
			 \|\pa_{x_1}\wh{\phi_c}|_{\bar{k}=0}\|_{L^2_{x_1}}
%			. 
%		\end{align*}
%		Taking inner product of \eqref{4.120a} with $\wh{\phi_c}$ over $x_1\in[-1,1]$, we have 
%		\begin{align*}
%			\|\pa_{x_1}\wh{\phi_c}\|_{L^2_{x_1}}
			\lesssim \|\wh{\pa c}|_{\bar{k}=0}\|_{L^2_{x_1}}=0,
		\end{align*}
		when $|\bar{k}|=0$. Here we note that by our setting, $\pa\neq I$ in this case and hence, $\wh{\pa c}|_{\bar{k}=0}=0$. If $|\bar{k}|\ge 1$, we can directly obtain from standard elliptic estimate that 
		\begin{align*}
			\|\pa_{x_1}\wh{\phi_c}\|^2_{L^2_{x_1}}+|\bar{k}|^2\|\wh{\phi_c}\|^2_{L^2_{x_1}}\lesssim \|\wh{c}\|^2_{L^2_{x_1}}.
		\end{align*}
		The above two estimates imply that for any $\bar{k}$, 
		\begin{align}\label{4.11}
			\|\pa_{x_1}\wh{\phi_c}\|^2_{L^2_{x_1}}+|\bar{k}|\|\wh{\phi_c}\|^2_{L^2_{x_1}}\lesssim \|\wh{c}\|^2_{L^2_{x_1}}. 
		\end{align}
		Similarly, since derivative on $t$ doesn't affect the boundary value, we have 
		\begin{align}\label{4.11a}
			\|\pa_t\pa_{x_1}\wh{\phi_c}\|^2_{L^2_{x_1}}+|\bar{k}|^2\|\pa_t\wh{\phi_c}\|^2_{L^2_{x_1}}\lesssim \|\pa_t\wh{c}\|^2_{L^2_{x_1}}.
		\end{align}
		On the other hand, when $\pa=\pa_{x_1}$, \eqref{4.120a} is a Dirichlet boundary problem. Then taking inner product of \eqref{4.120a} with $\wh{\phi_c}$ over $x_1\in[-1,1]$, we have 
		\begin{align*}
			\|\pa_{x_1}\wh{\phi_c}\|^2_{L^2_{x_1}}+|\bar{k}|^2\|\wh{\phi_c}\|^2_{L^2_{x_1}}\lesssim (\widehat{\partial_{x_1} c}, \wh{\phi_c})_{L^2_{x_1}}
			&= (\widehat{c}, \partial_{x_1} \wh{\phi_c})_{L^2_{x_1}}\lesssim \|\wh{c}\|_{L^2_{x_1}}\|\pa_{x_1}\wh{\phi_c}\|_{L^2_{x_1}}. 
		\end{align*}
		This implies that 
		\begin{align}\label{4.12}
			\|\pa_{x_1}\wh{\phi_c}\|_{L^2_{x_1}}+|\bar{k}|\|\wh{\phi_c}\|_{L^2_{x_1}}\lesssim \|\wh{c}\|_{L^2_{x_1}}. 
		\end{align}
		Similarly, 
		\begin{align}\label{4.12a}
			\|\pa_t\pa_{x_1}\wh{\phi_c}\|_{L^2_{x_1}}+|\bar{k}|\|\pa_t\wh{\phi_c}\|_{L^2_{x_1}}\lesssim \|\pa_t\wh{c}\|_{L^2_{x_1}}.
		\end{align}	
		Now we can compute \eqref{100a}. For the left hand side of \eqref{100a}, we have 
		\begin{align*}
			&\quad\,-\int^T_0(\widehat{\partial \PP f},v\cdot\widehat{\nabla_{x_1,\bar{x}}\Phi_{c}})_{L^2_{x_1,v}} \,dt\\
			&= -\sum_{j,m=1}^3\int^T_0(\widehat{\partial a_\pm}+\widehat{\partial b}\cdot v+(|v|^2-3)\widehat{\partial c} ,v_jv_m(|v|^2-5)\mu{(\partial_{x_j}\partial_{x_m}\phi_{c})^\wedge})_{L^2_{x_1,v}} \,dt\\
			&= 10\sum_{j=1}^3\int^T_0(\widehat{\partial c} ,\widehat{-\partial^2_j\phi_{c}})_{L^2_{x_1,v}} \,dt = 10\int^T_0\|\widehat{\partial c}\|^2_{L^2_{x_1}} \,dt.
		\end{align*}
		For the right hand side of \eqref{100a}, by H\"{o}lder's inequality and the elliptic estimate \eqref{4.11} and \eqref{4.12}, we have  
		\begin{align*}
			|S_1|\lesssim \sum_{|\alpha|\le 1}\Big(\|\widehat{\partial^\alpha f}(T)\|_{L^2_{x_1,v}}+\|\widehat{\partial^\alpha f_0}\|_{L^2_{x_1,v}}\Big).
		\end{align*}
		For $S_2$, we see from \eqref{4.11a} and \eqref{4.12a} that 
		\begin{align*}\notag
			|S_2|&\le\int^T_0|(\widehat{\partial f},\partial_t\widehat{\Phi_c})_{L^2_{x_1,v}}|\,dt = \int^T_0|(\{\I-\P\}\widehat{\partial f},\partial_t\widehat{\Phi_c})_{L^2_{x_1,v}}|\,dt\\
			&\notag\lesssim \eta\int^T_0\|\partial_t\widehat{\nabla_x\phi_c}\|^2_{L^2_{x_1}}\,dt+C_\eta\int^T_0\|\{\I-\P\}\widehat{\partial f}\|^2_{L^2_{x_1}L^2_D}\,dt\\
			&\notag\lesssim \eta\int^T_0\|\partial_t\widehat{c}\|^2_{L^2_{x_1}}\,dt+C_\eta\int^T_0\|\{\I-\P\}\widehat{\partial f}\|^2_{L^2_{x_1}L^2_D}\,dt\\
			&\lesssim \eta \int^T_0\|\widehat{\nabla_xb}\|^2_{L^2_{x_1}}\,dt+C_\eta\sum_{|\alpha|\le 1} \int^T_0\|\{\I-\P\}\widehat{\pa^\alpha f}\|^2_{L^2_{x_1}L^2_D}\,dt+\int^T_0\|(\widehat{g},\zeta)_{L^2_v}\|^2_{L^2_{x_1}}\,dt,
		\end{align*}
		where we used the third equation of \eqref{19} in the last inequality. 
		Thanks to \eqref{99cc}, $S_3$ can be bounded as 
		\begin{align*}
			|S_3|\lesssim \eta\int^T_0\|\widehat{\partial c}\|^2_{L^2_{x_1}}\,dt+C_\eta\int^T_0\|\{\I-\P\}\widehat{\nabla_x f}\|^2_{L^2_{x_1}L^2_D}\,dt.
		\end{align*}
		For the term $S_4$, applying \eqref{4.11} and \eqref{4.12}, we have 
		\begin{align*}
			|S_4|\le \eta\int^T_0\|\widehat{ c}\|_{L^2_{x_1}}\,dt+C_\eta\int^T_0\|\{\I-\P\}\widehat{\pa f}\|_{L^2_{x_1}L^2_D}\,dt+C_\eta\int^T_0\|(\widehat{\partial g},\zeta)_{L^2_v}\|_{L^2_{x_1}}\,dt.
		\end{align*}
		For the term $S_5$, using Cauchy-Schwarz's inequality, we have 
		\begin{align*}
			|S_5|&\le \int^T_0((\partial{\nabla_x\phi})^\wedge\cdot v\mu^{1/2},\widehat{\Phi})_{L^2_{x_1,v}}\,dt\\
			&\le \eta \int^T_0\|\widehat{ c}\|_{L^2_{x_1}}\,dt + C_\eta \int^T_0\|\widehat{\partial E}\|_{L^2_{x_1}}\,dt.
		\end{align*}
		%	{\color{blue}For the boundary term $S_6$, we first consider the inflow boundary condition \eqref{inflow}. 
		%		Here 
		%		\begin{align*}
		%			\widehat{\Phi_c}(\pm 1) = (|v|^2-5)v_1\partial_{x_1}\widehat{\phi_c}(t,\pm 1,\bar{k})\mu^{1/2},
		%		\end{align*}
		%		and by the trace theory and \eqref{99cc}, 
		%		\begin{align*}
		%			|\partial_{x_1}\phi_c(t,\pm 1,\bar{k})|\lesssim \|\widehat{\phi_c}(t,\bar{k})\|_{H^2_{x_1}}\lesssim \|\widehat{\partial c}\|_{L^2_{x_1}}.
		%		\end{align*}
		%		Then by using Cauchy-Schwarz's inequality, we have 
		%		\begin{align*}
		%			|S_6|&\lesssim \eta\int^T_0\|\widehat{\partial c}\|_{L^2_{x_1}}^2\,dt+C_\eta\int^T_0\int_{\R^3}|v_1||\widehat{\partial f}{(1)}|^2\,dvdt+C_\eta\int^T_0\int_{\R^3}|v_1||\widehat{\partial f}{(-1)}|^2\,dvdt\\
		%			&\lesssim \eta\int^T_0\|\widehat{\partial c}\|_{L^2_{x_1}}^2\,dt + C_\eta|G^+(\widehat{f})|^2 + C_\eta\int^T_0\int_{v_1<0}|v_1||\widehat{\partial f}{(1)}|^2\,dvdt+C_\eta\int^T_0\int_{v_1>0}|v_1||\widehat{\partial f}{(-1)}|^2\,dvdt.
		%		\end{align*}
		%		Using \eqref{inflow} and \eqref{40abc}, we have 
		%		\begin{align*}
		%			&\quad\,\int^T_0\int_{v_1<0}|v_1||\widehat{\partial_{x_1} f}{(1)}|^2\,dvdt\\
		%			&\lesssim \int^T_0\int_{v_1<0}|v_1|^{-1}\Big(|\partial_t\widehat{u_+}|^2+|\bar{k}\cdot\bar{v}|^2|\widehat{u_+}|^2+|\widehat{\nabla_x\phi}(1)v\mu^{1/2}|^2+|L\widehat{u_+}|^2+|\widehat{g}(1)|^2\Big)\,dvdt\\
		%			&\lesssim E_{\bar{k}}(u_+).
		%		\end{align*}
		%		where we have used $\partial_{x_1}\phi(1)=0$.} {\color{red}$\partial_{\bar{x}}\phi$ is not determined by the boundary}
		For $S_6$, we need to use the specular reflection boundary condition \eqref{reflection}:
		\begin{align}\label{103a}
			\widehat{f}(1,v_1)|_{v_1\neq 0} = \widehat{f}(1,-v_1),\quad\widehat{f}(-1,-v_1)|_{v_1\neq 0}  = \widehat{f}(-1,v_1) .
		\end{align}
		Using the Neumann boundary condition \eqref{97a}, we know that 
		\begin{align*}
			\widehat{\partial_{x_1}\phi}(\pm 1,\bar{k})= 0. 
		\end{align*}
		By the above boundary value for $f$ and $\phi$, we have $g_\pm(-1,-v_1)|_{v_1\neq 0}=g_\pm(-1,v_1)$ and $g_\pm(1,-v_1)|_{v_1\neq 0}=g_\pm(1,v_1)$. Thus, using the equation \eqref{40abc} to define derivative $\partial_{x_1}f_\pm$, one has 
		\begin{align}\label{115}\notag
			-v_1\widehat{\partial_{x_1}f}(-1,\bar{k},-v_1)&=
			\Big(\partial_t{f_\pm} +\bar{v}\cdot\nabla_{\bar{x}}f_\pm\pm {\nabla_x\phi}\cdot v\mu^{1/2} - L_\pm {f} - {g_\pm}\Big)(-1,\bar{k},-v_1)\\
			&= v_1\widehat{\partial_{x_1}f}(-1,\bar{k},v_1),
		\end{align}
			and similarly, 
			\begin{align}\label{115a}
			-v_1\widehat{\partial_{x_1}f}(1,\bar{k},-v_1) &= v_1\widehat{\partial_{x_1}f}(1,\bar{k},v_1).
		\end{align}
		
		For the case $\partial=\partial_{x_1}$, by \eqref{4.120a} with boundary condition $\widehat{\phi_c}(\pm 1,\bar{k})=0$, we know that $\widehat{\partial_{\bar{x}}\phi_c}(\pm 1,\bar{k})=0$. Thus by the definition of $\widehat{\Phi_c}$, we have 
		\begin{align*}
			\widehat{\Phi_c}(- 1,\bar{k},-v_1,\bar{v}) = -\widehat{\Phi_c}(- 1,\bar{k},v_1,\bar{v}),\quad \widehat{\Phi_c}(1,\bar{k},-v_1,\bar{v}) = -\widehat{\Phi_c}(1,\bar{k},v_1,\bar{v}).
		\end{align*}
		Therefore, when $\partial=\partial_{x_1}$, by change of variable $v_1\mapsto -v_1$, \eqref{115} and \eqref{115a}, we have 
		\begin{align}\label{117}
			S_6 &= -\int^T_0(v_1\widehat{\partial_{x_1} f_\pm}(1,v_1),\widehat{\Phi_c}(1,v_1))_{L^2_v}\,dt \notag\\&\qquad\qquad\qquad\qquad\qquad\qquad+ \int^T_0(v_1\widehat{\partial_{x_1} f_\pm}(-1,v_1),\widehat{\Phi_c}(-1,v_1))_{L^2_v}\,dt\\
			&=\notag -\int^T_0(-v_1\widehat{\partial_{x_1}f_\pm}(1,-v_1),\widehat{\Phi_c}(1,-v_1))_{L^2_v}\,dt\notag\\&\notag\qquad\qquad\qquad\qquad\qquad\qquad+ \int^T_0(-v_1\widehat{\partial_{x_1}f_\pm}(-1,-v_1),\widehat{\Phi_c}(-1,-v_1))_{L^2_v}\,dt\\
			&=\label{118} \int^T_0(v_1\widehat{\partial_{x_1}f_\pm}(1,v_1),\widehat{\Phi_c}(1,v_1))_{L^2_v}\,dt- \int^T_0(v_1\widehat{\partial_{x_1}f_\pm}(-1,v_1),\widehat{\Phi_c}(-1,v_1))_{L^2_v}\,dt
			= 0. 
		\end{align}
		Here $S_6$ equal to zero because \eqref{117} and \eqref{118} are the same except the sign. 
		
		For the case $\partial=I,\partial_{x_2},\partial_{x_3}$, by boundary condition $\partial_{x_1}\widehat{\phi_c}(\pm 1,\bar{k})=0$, we know that
		\begin{align*}
			\widehat{\Phi_c}(- 1,\bar{k},-v_1,\bar{v}) = \widehat{\Phi_c}(- 1,\bar{k},v_1,\bar{v}),\quad \widehat{\Phi_c}(1,\bar{k},-v_1,\bar{v}) = \widehat{\Phi_c}(1,\bar{k},v_1,\bar{v}).
		\end{align*}
		On the other hand, from \eqref{103a} we have 
		\begin{align}\label{115b}
			v_1\widehat{\partial_{}f}(-1,\bar{k},-v_1) = v_1\widehat{\partial_{}f}(-1,\bar{k},v_1), \quad v_1\widehat{\partial_{}f}(1,\bar{k},-v_1) = v_1\widehat{\partial_{}f}(1,\bar{k},v_1).
		\end{align}
		Therefore, when $\partial=I,\partial_{x_2},\partial_{x_3}$, by change of variable $v_1\mapsto -v_1$, we have 
		\begin{align}\label{116}
			S_6 &=\notag -\int^T_0(v_1\widehat{\partial_{}f_\pm}(1,v_1),\widehat{\Phi_c}(1,v_1))_{L^2_v}\,dt+ \int^T_0(v_1\widehat{\partial_{}f_\pm}(-1,v_1),\widehat{\Phi_c}(-1,v_1))_{L^2_v}\,dt\\
			%	&=\notag -\int^T_0(-v_1\widehat{\partial_{}f_\pm}(1,-v_1),\widehat{\Phi_c}(1,-v_1))_{L^2_v}\,dt+ \int^T_0(-v_1\widehat{\partial_{}f_\pm}(-1,-v_1),\widehat{\Phi_c}(-1,-v_1))_{L^2_v}\,dt\\
			&=\notag \int^T_0(v_1\widehat{\partial_{}f_\pm}(1,v_1),\widehat{\Phi_c}(1,v_1))_{L^2_v}\,dt- \int^T_0(v_1\widehat{\partial_{}f_\pm}(-1,v_1),\widehat{\Phi_c}(-1,v_1))_{L^2_v}\,dt\\
			&= 0. 
		\end{align}
		
		Combining the above estimates for $S_j$'s $(1\le j\le 6)$, applying \eqref{c0} for the case of $|\alpha|=0$ and $\bar{k}=0$, taking summation of \eqref{100a} for $|\alpha|\le 1$ for the remaining cases and then taking the square root and summation over $\bar{k}\in\Z^2$, and finally letting $\eta$ suitably small, we obtain
%		 that for $\partial=I,\partial_{x_1},\partial_{x_2},\partial_{x_3}$, 
		\begin{multline}\label{122af}
			\sum_{|\alpha|\le 1}\|\widehat{\partial^\alpha c}\|^2_{L^1_{\bar{k}}L^2_TL^2_{x_1}}
			\lesssim \sum_{|\alpha|\le 1}\Big(\|\widehat{\partial^\alpha f}(T)\|_{L^1_{\bar{k}}L^2_{x_1,v}}+\|\widehat{\partial^\alpha f_0}\|_{L^1_{\bar{k}}L^2_{x_1,v}}\Big) \\
			\qquad+C_\eta\sum_{|\alpha|\le 1}\Big(\|\{\I-\P\}\widehat{\pa^\alpha f}\|_{L^1_{\bar{k}}L^2_TL^2_{x_1}L^2_D} + \|\widehat{\partial^\alpha E}\|_{L^1_{\bar{k}}L^2_TL^2_{x_1}}
			+\|(\widehat{\partial^\alpha g},\zeta)_{L^2_v}\|_{L^1_{\bar{k}}L^2_TL^2_{x_1}}\Big)\\
			+ \eta^{1/2}\|\widehat{\nabla_x b}\|_{L^1_{\bar{k}}L^2_TL^2_{x_1}}+\|\wh{E}\|_{L^1_{\bar{k}}L^\infty_TL^2_{x_1}}\|\wh{E}\|_{L^1_{\bar{k}}L^2_TL^2_{x_1}}.
		\end{multline}

		% Moreover, 
		%%\begin{equation*}
		%%\widehat{\partial_{\bar{x}} f}(x_1,v_1) = \widehat{\partial_{\bar{x}} f}(-x_1,-v_1),
		%%\end{equation*}
		%%and 
		%\begin{align*}
		%a_+(x_1)-a_-(x_1) = a_+(-x_1)-a_-(-x_1), \text{ for } x_1\in[-1,1].
		%\end{align*}
		%Together with \eqref{40abcd}, we have 
		%\begin{align*}
		%\widehat{\phi}(x_1,\bar{k}) = \widehat{\phi}(-x_1,\bar{k}),\quad
		%\widehat{\partial_{\bar{x}}\phi}(x_1,\bar{k}) = \widehat{\partial_{\bar{x}}\phi}(-x_1,\bar{k}),
		%\end{align*}
		%and hence 
		%\begin{align}
		%\widehat{\phi}(1,\bar{k}) = \widehat{\phi}(-1,\bar{k}),\quad
		%\widehat{\partial_{\bar{x}}\phi}(1,\bar{k}) = \widehat{\partial_{\bar{x}}\phi}(-1,\bar{k}).
		%\end{align}
		%
		%On the other hand, since $\widehat{c}$ is even with respect to $x_1$, we obtain from \eqref{4.120a} that $\widehat{\phi_c}$ is odd w.r.t. $x_1$ when $\partial=\partial_{x_1}$ and $\widehat{\phi_c}$ is even w.r.t. $x_1$ when $\partial=I,\partial_{x_2},\partial_{x_3}$
		%
		%
		%Then $\widehat{c}$ is even with respect to $x_1$ and 

		\medskip \noindent{\bf Step 2. Estimate of $\widehat{b}(t,x_1,\bar{k})$.}
		Now we consider the estimate of $\widehat{b}$. For this purpose we choose 
		\begin{align*}
			\widehat{\Phi}=\widehat{\Phi_b}=\sum^3_{m=1}\widehat{\Phi^{j,m}_b},\ j=1,2,3,
		\end{align*}
		where 
		\begin{equation*}
			\widehat{\Phi^{j,m}_b}=\left\{\begin{aligned}
				\big(|v|^2v_mv_j\widehat{\partial_{x_m}\phi_j}-\frac{7}{2}(v_m^2-1)\widehat{\partial_{x_j}\phi_j}\big)\mu^{1/2},\ j\neq m,\\
				\frac{7}{2}(v_j^2-1)\widehat{\partial_{x_j}\phi_j}\mu^{1/2},\qquad\qquad\qquad j=m.
			\end{aligned}\right.
		\end{equation*}
		Also, $\phi_j$ $(1\le j\le 3)$ solves  
		\begin{equation}\left\{\label{118a}
			\begin{aligned}
				&-\partial^2_{x_1}\widehat{\phi_j}+|\bar{k}|^2\widehat{\phi_j}(\bar{k})=\widehat{\partial b_j}(\bar{k}),\\&\partial_{x_1}\widehat{\phi_1}(\pm 1,\bar{k})=\widehat{\phi_2}(\pm 1,\bar{k})=\widehat{\phi_3}(\pm 1,\bar{k})=0,
			\end{aligned}\right.
		\end{equation}for the case $\partial=\partial_{x_1}$ and
		\begin{equation}\label{118b}\left\{\begin{aligned}
				&-\partial^2_{x_1}\widehat{\phi_j}+|\bar{k}|^2\widehat{\phi_j}(\bar{k})=\widehat{\partial b_j}(\bar{k}),\\&\widehat{\phi_1}(\pm 1,\bar{k})=\partial_{x_1}\widehat{\phi_2}(\pm 1,\bar{k})=\partial_{x_1}\widehat{\phi_3}(\pm 1,\bar{k})=0,
			\end{aligned}\right.
		\end{equation}
		for the case $\partial=I,\partial_{x_2},\partial_{x_3}$.
		When $\bar{k}=0$, $j=1$ and $\pa=\pa_{x_1}$, \eqref{118a} is Neumann boundary problem. When $j=2,3$ and $\pa=I$, \eqref{118b} is also Neumann boundary problem. Their existence are guaranteed by the fact that $\int_{-1}^1\int_{{\T}^2}\pa_{x_1}b_1\,d\bar{x}dx_1=b_1(1)-b_1(-1)=0$ and $\int_{-1}^1\int_{{\T}^2}b_j\,d\bar{x}dx_1=0$ for $j=2,3$, which follows from \eqref{reflection} and \eqref{conservation_finite} respectively. 
		Under this choice we have 
		\begin{align*}
			&\quad\,-\sum^3_{m=1}\int^T_0(\PP \widehat{\partial f},iv\cdot k\widehat{\Phi^{j,m}_b})_{L^2_{x_1,v}}dt\\
			%		&=-\sum^3_{m=1}\int^T_0\big((\widehat{\partial a_\pm}+\widehat{\partial b}\cdot v+(|v|^2-3)\widehat{\partial c})\mu^{1/2},v\cdot\widehat{\nabla_x\Phi^{j,m}_b}\big)_{L^2_{x_1,v}}dt\\
			&=-\sum^{3}_{m=1,m\neq j}\int^T_0(v_mv_j\mu^{1/2}\widehat{\partial b_j},|v|^2v_mv_j\mu^{1/2}\widehat{\partial_{x_m}^2\phi_j})_{L^2_{x_1,v}}dt\\
			&\qquad-\sum^{3}_{m=1,m\neq j}\int^T_0(v_mv_j\mu^{1/2}\widehat{\partial b_m},|v|^2v_mv_j\mu^{1/2}\widehat{\partial_{x_m}\partial_{x_j}\phi_j})_{L^2_{x_1,v}}dt\\
			&\qquad+7\sum^{3}_{m=1,m\neq j}\int^T_0(\widehat{\partial b_m},\widehat{\partial_{x_m}\partial_{x_j}\phi_j})_{L^2_{x_1}}dt-7\int^T_0(\widehat{\partial b_m},\widehat{\partial^2_{x_j}\phi_j})_{L^2_{x_1}}dt\\
			&= -7 \sum^3_{m=1}\int^T_0(\widehat{\partial b_j},\widehat{\partial_{x_m}^2\phi_j})_{L^2_{x_1}}dt=7\int^T_0\|\widehat{\partial b_j}\|^2_{L^2_{x_1}}dt.
		\end{align*}
		Note that from \eqref{25}, we have $\int_{\R^3}v_m^2(v_m^2-1)\mu\,dv=2$, $\int_{\R^3}v_m^2(v_j^2-1)\mu\,dv=0$, $\int_{\R^3}v_m^2v_j^2|v|^2\mu\,dv=7$ when $m\neq j$. Similar to the calculation for deriving \eqref{99cc}, \eqref{4.11}, \eqref{4.11a}, \eqref{4.12} and \eqref{4.12a}, we have that for $|\bar{k}|\ge 0$, 
		\begin{equation*}
			\|\partial_t\widehat{\pa_{x_1}\phi_{j}}\|_{L^2_{x_1}}+|\bar{k}|\|\partial_t\widehat{\phi_{j}}\|_{L^2_{x_1}}\lesssim \|\partial_t\widehat{b_j}\|_{L^2_{x_1}},
			%			 |\bar{k}|^2\|\partial_t\widehat{\phi_{j}}\|_{L^2_{x_1}}+|\bar{k}|\|\partial_t\widehat{\phi_{j}}\|_{H^1_{x_1}}\lesssim \|\partial_t\widehat{\partial b_j}\|_{L^2_{x_1,}}.
		\end{equation*} 
		\begin{equation*}
			\|\widehat{\pa_{x_1}\phi_{j}}\|_{L^2_{x_1}}+|\bar{k}|\|\widehat{\phi_{j}}\|_{L^2_{x_1}}\lesssim \|\widehat{b_j}\|_{L^2_{x_1}},
		\end{equation*}
		and
		\begin{align*}
			\|\pa^2_{x_1}\widehat{\phi_{j}}\|_{L^2_{x_1}}+|\bar{k}|\|\pa_{x_1}\widehat{\phi_{j}}\|_{L^2_{x_1}}+|\bar{k}|^2\|\widehat{\phi_{j}}\|_{L^2_{x_1}}\lesssim \|\widehat{\partial b_j}\|_{L^2_{x_1}}.
		\end{align*}
		%	By using the second equation of \eqref{19}, we obtain 
		%	\begin{equation}\begin{aligned}
		%			\|\partial_t\widehat{b}\|_{H^{-1}_{x_1}}+\|\partial_t\partial_{x_1}\widehat{b}\|_{H^{-1}_{x_1}}&\lesssim \|({\nabla_x (a_++a_-)})^\wedge\|_{L^2_{x_1}} + \|\widehat{\nabla_xc}\|_{L^2_{x_1}}\\
		%			&\qquad+\|\{\I-\P\}\widehat{\nabla_x f}\|_{L^2_{x_1}L^2_D}+ \|(\widehat{g},\zeta)_{L^2_v}\|_{L^2_{x_1}},\\
		%			\|\partial_t\widehat{\partial_{\bar{x}}b}\|_{L^2_{x_1}}&\lesssim |\bar{k}|\big(\|({\nabla_x (a_++a_-)})^\wedge\|_{L^2_{x_1}}+ \|\widehat{\nabla_xc}\|_{L^2_{x_1}}\\
		%			&\qquad+\|\{\I-\P\}\widehat{\nabla_x f}\|_{L^2_{x_1}L^2_D}+ \|(\widehat{g},\zeta)_{L^2_v}\|_{L^2_{x_1}}\big).
		%	\end{aligned}\end{equation}
		As a consequence, 
		\begin{align*}
			|S_2|&\le \int^T_0|(\{\I-\P\}\widehat{\partial f},\partial_t\widehat{\Phi^{j,m}_b})_{L^2_{x_1,v}}\,dt + \int^T_0|(\P\widehat{\partial f},\partial_t\widehat{\Phi^{j,m}_b})_{L^2_{x_1,v}}\,dt\\
			&\le \eta\int^T_0\|\partial_t\widehat{\na_x\phi_j}\|^2_{L^2_{x_1}}\,dt
			+C_\eta\int^T_0\|\{\I-\P\}\widehat{\partial f}\|^2_{L^2_{x_1}L^2_D}\,dt + C_\eta\int^T_0\|\widehat{\pa c}\|^2_{L^2_{x_1}}\,dt\\
			&\le \eta\int^T_0\|({\nabla_x (a_++a_-)})^\wedge\|^2_{L^2_{x_1}}\,dt
			+C_\eta\sum_{|\alpha|\le1}\int^T_0\|\{\I-\P\}\widehat{\pa^\alpha f}\|^2_{L^2_{x_1}L^2_D}\,dt\\
			&\qquad\qquad\qquad\qquad\qquad\qquad\qquad\qquad\qquad\qquad\qquad + C_\eta\int^T_0\|\widehat{\pa c}\|^2_{L^2_{x_1}}\,dt,
		\end{align*}
		where we used the second equation of \eqref{19} in the last inequality. 
		By the same argument as in the case of $\widehat{c}(t,x_1,\bar{k})$, we have 
		\begin{multline*}
			|S_1|+|S_3|+|S_4|+|S_5|\lesssim \sum_{|\alpha|\le 1}\Big(\|\widehat{\partial^\alpha f}(T)\|^2_{L^2_{x_1,v}}+\|\widehat{\partial^\alpha f_0}\|^2_{L^2_{x_1,v}}\Big)+ \eta\sum_{|\alpha|\le1}\int^T_0\|\widehat{\partial^\alpha b}\|^2_{L^2_{x_1}}\,dt\\
			+C_\eta \int^T_0\|\widehat{\partial E}\|^2_{L^2_{x_1}}\,dt+C_\eta\int^T_0\|\{\I-\P\}\widehat{\nabla_x f}\|^2_{L^2_{x_1}L^2_D}\,dt+C_\eta\int^T_0\|(\widehat{\partial g},\zeta)_{L^2_v}\|^2_{L^2_{x_1}}\,dt.
		\end{multline*}
		
		Now we consider the boundary term $S_6$. For the case $\partial=\partial_{x_1}$, by boundary condition \eqref{118a}, we have $\widehat{\partial_{x_j}\phi_j}(\pm 1)=0$ for $j=1,2,3$ and $\widehat{\partial_{x_m}\phi_j}(\pm 1)=0$ for $m,j=2,3$. Thus 
		\begin{align*}
			\widehat{\Phi_b}(-1,\bar{k},-v_1,\bar{v}) = -\widehat{\Phi_b}(-1,\bar{k},v_1,\bar{v}),\quad\widehat{\Phi_b}(1,\bar{k},-v_1,\bar{v}) = -\widehat{\Phi_b}(1,\bar{k},v_1,\bar{v}).
		\end{align*}
		Together with \eqref{115} and \eqref{115a}, by changing of variable $v_1\mapsto -v_1$, we know that 
		\begin{align*}
			S_6 &=\notag -\int^T_0(v_1\widehat{\partial_{x_1}f_\pm}(1,v_1),\widehat{\Phi_b}(1,v_1))_{L^2_v}\,dt+ \int^T_0(v_1\widehat{\partial_{x_1}f_\pm}(-1,v_1),\widehat{\Phi_b}(-1,v_1))_{L^2_v}\,dt\\
			&=\notag \int^T_0(v_1\widehat{\partial_{x_1}f_\pm}(1,v_1),\widehat{\Phi_b}(1,v_1))_{L^2_v}\,dt- \int^T_0(v_1\widehat{\partial_{x_1}f_\pm}(-1,v_1),\widehat{\Phi_b}(-1,v_1))_{L^2_v}\,dt= 0. 
		\end{align*}
		For the case $\partial=I,\partial_{x_2},\partial_{x_3}$, by \eqref{118b}, we know that $\partial_{x_1}\phi_j(\pm1)=\partial_{j}\phi_1(\pm1)=0$ for $j=2,3$. Thus 
		\begin{align*}
			\widehat{\Phi_b}(-1,\bar{k},-v_1,\bar{v}) = \widehat{\Phi_b}(-1,\bar{k},v_1,\bar{v}),\quad\widehat{\Phi_b}(1,\bar{k},-v_1,\bar{v}) = \widehat{\Phi_b}(1,\bar{k},v_1,\bar{v}),
		\end{align*}
		and hence by \eqref{115b} and change of variable $v_1\mapsto -v_1$, 
		\begin{align*}
			S_6 &=\notag -\int^T_0(v_1\widehat{\partial f_\pm}(1,v_1),\widehat{\Phi_b}(1,v_1))_{L^2_v}\,dt+ \int^T_0(v_1\widehat{\partial f_\pm}(-1,v_1),\widehat{\Phi_b}(-1,v_1))_{L^2_v}\,dt\\
			&=\notag \int^T_0(v_1\widehat{\partial f_\pm}(1,v_1),\widehat{\Phi_b}(1,v_1))_{L^2_v}\,dt- \int^T_0(v_1\widehat{\partial f_\pm}(-1,v_1),\widehat{\Phi_b}(-1,v_1))_{L^2_v}\,dt= 0. 
		\end{align*}
		Combining the above estimates, taking summation of \eqref{100a} over $|\alpha|\le1$, and then taking the square root and summation over $\bar{k}\in\Z^2$, and finally letting $\eta$ sufficiently small, we have 
		\begin{multline}\label{122bf}
			\sum_{|\alpha|\le1}\|\widehat{\partial^\alpha b}\|_{L^1_{\bar{k}}L^2_TL^2_{x_1}}\lesssim \sum_{|\alpha|\le1}\Big(\|\widehat{\partial^\alpha f}(T)\|_{L^1_{\bar{k}}L^2_{x_1,v}}+\|\widehat{\partial^\alpha f_0}\|_{L^1_{\bar{k}}L^2_{x_1,v}}\Big)\\+\eta^{1/2}\|({\nabla_x (a_++a_-)})^\wedge\|_{L^1_{\bar{k}}L^2_TL^2_{x_1}}
			+C_\eta\sum_{|\alpha|\le1}\Big(\|\{\I-\P\}\widehat{\pa^\alpha f}\|_{L^1_{\bar{k}}L^2_TL^2_{x_1}L^2_D} \\+ \|\widehat{\pa^\alpha c}\|_{L^1_{\bar{k}}L^2_TL^2_{x_1}} +\|(\widehat{\partial^\alpha g},\zeta)_{L^2_v}\|_{L^1_{\bar{k}}L^2_TL^2_{x_1}}+  \|\widehat{\partial^\alpha E}\|_{L^1_{\bar{k}}L^2_TL^2_{x_1}}\Big).
		\end{multline}

		%	{\color{red}I think the CPAM paper has some problems in this part. For instance, $v_1\partial_{x_1}\widehat{f}(1,v_1)=-v_1\partial_{x_1}\widehat{f}(1,-v_1)$ not $v_1\partial_{x_1}\widehat{f}(1,v_1)=v_1\partial_{x_1}\widehat{f}(1,-v_1)$.}
		
		\medskip\noindent{\bf Step 3. Estimate on $\widehat{a_+}(t,x_1,\bar{k})-\widehat{a_-}(t,x_1,\bar{k})$ and $\widehat{E}(t,x_1,\bar{k})$.} We choose the following two test functions
		\begin{align*}
			\widehat{\Phi} = \widehat{\tilde{\Phi}_{a\pm}} = (|v|^2-10)\big(v\cdot\widehat{\nabla_{x_1,\bar{x}}\phi_{a\pm}}(t,x_1,\bar{k})\big)\mu^{1/2},
		\end{align*}
		where $\phi_a = (\phi_{a+}(x_1,\bar{k}),\phi_{a_-}(x_1,\bar{k}))$ solves 
		\begin{equation}\begin{aligned}\label{102h}
				-\partial^2_{x_1}\widehat{\phi_{a+}}+|\bar{k}|^2\widehat{\phi_{a+}} = \widehat{\partial a_+}-\widehat{\partial a_-},\\	-\partial^2_{x_1}\widehat{\phi_{a-}}+|\bar{k}|^2\widehat{\phi_{a-}} = \widehat{\partial a_-}-\widehat{\partial a_+},
			\end{aligned}
		\end{equation}
		with boundary condition $\widehat{\phi_{a}}(\pm 1,\bar{k})=0$
		for the case $\partial=\partial_{x_1}$ and, 
		$\partial_{x_1}\widehat{\phi_{a}}(\pm 1,\bar{k})=0$
		for the case $\partial=I,\partial_{x_2},\partial_{x_3}$. When $|\bar{k}|=0$ and $\pa=I$, \eqref{102h} is a pure Neumann boundary problem and we need $\int^1_{-1}\int_{\T^2}a_+-a_-\,d\bar{x}dx_1$ to ensure the existence for \eqref{102h}, which follows from \eqref{conservation_finite}. In particular, as in the case of $\widehat{c}$, we add zero mean condition $\int_{-1}^1\widehat{\phi_a}(x_1,0)\,dx_1=0$ for Neumann boundary case. Thus, similar to the calculation from \eqref{99cc} to \eqref{4.12a}, we have 
		\begin{equation*}
			\|\pa^2_{x_1}\widehat{\phi_{a}}\|_{L^2_{x_1}}+|\bar{k}|\|\pa_{x_1}\widehat{\phi_{a}}\|_{L^2_{x_1}}+|\bar{k}|^2\|\widehat{\phi_{a}}\|_{L^2_{x_1}}\lesssim \|\widehat{\partial a_+}-\widehat{\partial a_-}\|_{L^2_{x_1}},
		\end{equation*}
		\begin{equation*}
			\|\pa_{x_1}\wh{\phi_a}\|_{L^2_{x_1}}+|\bar{k}|\|\wh{\phi_a}\|_{L^2_{x_1}}\lesssim \|\wh{a_+}-\wh{a_-}\|_{L^2_{x_1}}, 
		\end{equation*}and
		\begin{equation}\label{4.18a}
			\|\pa_t\pa_{x_1}\wh{\phi_a}\|_{L^2_{x_1}}+|\bar{k}|\|\pa_t\wh{\phi_a}\|_{L^2_{x_1}}\lesssim \|\pa_t(\wh{a_+}-\wh{a_-})\|_{L^2_{x_1}}.
		\end{equation}	
		Now we compute \eqref{100a} with summation on $\pm$. For the left hand side, taking summation on $\pm$, we have 
		\begin{align*}
			&\quad\,-\sum_{\pm}\int^T_0(\widehat{\partial \PP f},v\cdot\widehat{\nabla_{x_1,\bar{x}}\Phi_{a\pm}})_{L^2_{x_1,v}} \,dt\\
			&= -\sum_{\pm}\sum_{j,m=1}^3\int^T_0(\widehat{\partial a_\pm}+\widehat{\partial b}\cdot v+(|v|^2-3)\widehat{\partial c} ,v_jv_m(|v|^2-10)\mu{(\partial_{x_j}\partial_{x_m}\phi_{a\pm})^\wedge})_{L^2_{x_1,v}} \,dt\\
			&= \sum_{\pm}\sum_{j=1}^3\int^T_0(\widehat{\partial a_\pm} ,\widehat{-\partial^2_j\phi_{a\pm}})_{L^2_{x_1,v}} \,dt = \int^T_0\|\widehat{\partial a_+}-\widehat{\partial a_-}\|^2_{L^2_{x_1,v}} \,dt.
		\end{align*}
		%	Then we estimate the $S_j$ $(1\le j\le 6)$ term by term. From \eqref{102h}, we have 
		%	\begin{equation}\begin{aligned}\label{99a}
		%			\|\partial_t\widehat{\phi_{a\pm}}\|_{H^1_{x_1}}+|\bar{k}|\|\partial_t\widehat{\phi_{a\pm}}\|_{L^2_{x_1}}\lesssim \|\partial_t(\widehat{\partial a_+}-\widehat{\partial a_-})\|_{H^{-1}_{x_1}},\\ |\bar{k}|^2\|\partial_t\widehat{\phi_{a\pm}}\|_{L^2_{x_1}}+|\bar{k}|\|\partial_t\widehat{\phi_{a\pm}}\|_{H^1_{x_1}}\lesssim \|\partial_t(\widehat{\partial a_+}-\widehat{\partial a_-})\|_{L^2_{x_1}}.
		%		\end{aligned}
		%	\end{equation} 
		%	and
		%	\begin{align}\label{119}
		%		\|\widehat{\phi_{a}}\|_{H^2_{x_1}}+|\bar{k}|\|\widehat{\phi_{a}}\|_{H^1_{x_1}}+|\bar{k}|^2\|\widehat{\phi_{a}}\|_{L^2_{x_1}}\lesssim \|\widehat{\partial a_+}-\widehat{\partial a_-}\|_{L^2_{x_1}}.
		%	\end{align}
		%	Moreover, from the first equation of \eqref{98a}, we have 
		%	\begin{align}\label{99b}\notag
		%		\|\partial_t(\widehat{a_+}-\widehat{a_-})\|_{H^{-1}_{x_1}}+\|\partial_t(\widehat{\partial_{x_1} a_+}-\widehat{\partial_{x_1} a_-})\|_{H^{-1}_{x_1}}\lesssim \|\widehat{\nabla_x\cdot G}\|_{L^2_{x_1,v}},\\
		%		\|\partial_t(\widehat{\partial_{\bar{x}} a_+}-\widehat{\partial_{\bar{x}} a_-})\|_{L^2_{x_1}}\lesssim |\bar{k}|\|\widehat{\nabla_x\cdot G}\|_{L^2_{x_1,v}}.
		%	\end{align}
		%	
		%	With the above estimates, we now compute 
		For $S_2$, we decompose $f_\pm$ into $\PP f$ and $(\II-\PP)f$ and apply \eqref{4.18a} to obtain 
		\begin{align*}
			|\sum_\pm S_2|&\lesssim \int^T_0|(\{\I-\P\}\widehat{\partial f},\partial_t\widehat{\Phi_a})_{L^2_{x_1,v}}|\,dt + \int^T_0|(\P \widehat{\partial f},\partial_t\widehat{\Phi_a})_{L^2_{x_1,v}}|\,dt\\
			&\lesssim \int^T_0\|\partial_t\widehat{\nabla_{x}\phi_a}\|^2_{L^2_{x_1}}\,dt+\int^T_0\|\{\I-\P\}\widehat{\partial f}\|^2_{L^2_{x_1}L^2_D}+\int^T_0\|\widehat{\partial b}\|^2_{L^2_{x_1}}\,dt\\
			&\lesssim \sum_{|\alpha|\le 1}\int^T_0\|\{\I-\P\}\widehat{\partial^\alpha  f}\|^2_{L^2_{x_1}L^2_D}+\int^T_0\|\widehat{\partial b}\|^2_{L^2_{x_1}}\,dt,
		\end{align*}
		where we used the first equation of \eqref{98a}
		%Here we used $\|\widehat{\nabla_x\cdot G}\|_{L^2_{x_1,v}}\lesssim \|\{\I-\P\}\widehat{\nabla_x f}\|^2_{L^2_{x_1}L^2_{D}}$. Note that the term $|\bar{k}|\int^T_0\|\partial_t\widehat{\phi_a}\|^2_{L^2_{x_1,v}}\,dt$ is already $0$ when $\bar{k}=0$ and we can used \eqref{99a} and \eqref{99b} when $\bar{k}\neq 0$. 
		%	Thanks to \eqref{119}, $S_3$ can be controlled as 
		For $S_1$, $S_3$ and $S_4$, we apply the same argument as in Step 1 to obtain 
		\begin{multline*}
			|S_1|+|S_3|+|S_4|\lesssim \sum_{|\alpha|\le 1}\Big(\|\widehat{\partial^\alpha f}(T)\|_{L^2_{x_1,v}}+\|\widehat{\partial^\alpha f_0}\|_{L^2_{x_1,v}}\Big)+ \eta\sum_{|\alpha|\le1}\int^T_0\|\widehat{\partial^\alpha a_+}-\widehat{\partial^\alpha a_-}\|^2_{L^2_{x_1}}\,dt\\
			\qquad+C_\eta\sum_{|\alpha|\le1}\int^T_0\|\{\I-\P\}\widehat{\pa^\alpha f}\|^2_{L^2_{x_1}L^2_D}\,dt+C_\eta\int^T_0\|(\widehat{\partial g},\zeta)_{L^2_v}\|^2_{L^2_{x_1}}\,dt.
		\end{multline*}
		For $S_5$, notice from \eqref{102h} that $\phi_{a+} = -\phi_{a-}$.
		Thus, by direction calculations, 
		\begin{align*}
			S_5 &= \sum_\pm\mp\int^T_0((\partial{\nabla_x\phi})^\wedge\cdot v\mu^{1/2},\widehat{\Phi_{a\pm}})_{L^2_{x_1,v}}\,dt\\
			&= \sum_{\pm,j,m} \mp \int^T_0((\partial{\partial_{x_j}\phi})^\wedge \mu^{1/2},(|v|^2-10)v_jv_m\widehat{\partial_{x_m}\phi_{a\pm}}(t,x_1,\bar{k})\mu^{1/2})_{L^2_{x_1,v}}\,dt\\
			&= 5\sum_{\pm,j} \mp \int^T_0((\partial{\partial_{x_j}\phi})^\wedge ,\widehat{\partial_{x_j}\phi_{a\pm}}(t,x_1,\bar{k}))_{L^2_{x_1}}\,dt\\
			&= -10\sum_{j=1}^3  \int^T_0((\partial{\partial_{x_j}\phi})^\wedge ,\widehat{\partial_{x_j}\phi_{a+}}(t,x_1,\bar{k}))_{L^2_{x_1}}\,dt.
		\end{align*}
		When $\partial=\partial_{x_1}$, by equation \eqref{40abcd}, boundary condition $\widehat{\phi_{a}}(\pm 1,\bar{k})=0$ and Neumann boundary condition \eqref{97a}, we deduce from integration by parts that  
		\begin{align*}
			\quad\,\sum_{j=1}^3((\partial_{x_1}{\partial_{x_j}\phi})^\wedge ,\widehat{\partial_{x_j}\phi_{a+}})_{L^2_{x_1}}
			&= 
			%		(\widehat{\partial_{x_1}{\phi}}(1)\,|\,\widehat{\partial_{x_1}\phi_{a+}}(1))
			%		-(\widehat{\partial_{x_1}{\phi}}(-1)\,|\,\widehat{\partial_{x_1}\phi_{a+}}(-1))
			%		\\&\qquad+
			(\widehat{\partial_{x_1}\phi} ,-\widehat{\Delta_x\phi_{a+}})_{L^2_{x_1}}
%			+\sum_{j=2}^3(\widehat{\partial_{x_1}{\phi}} ,-\widehat{\partial^2_{x_j}\phi_{a+}})_{L^2_{x_1}}\\
%			&= (\widehat{\partial_{x_1}\phi},\partial_{x_1}(\widehat{a_+}-\widehat{a_-}))_{L^2_{x_1}}
			=
			(\widehat{\partial_{x_1}\phi},-\partial_{x_1}\widehat{\Delta_x\phi})_{L^2_{x_1}}\\
%			&=
%			%		-(\widehat{\partial_{x_1}{\phi}}(1)\,|\,\widehat{\partial^2_{x_1}\phi}(1))
%			%		+(\widehat{\partial_{x_1}{\phi}}(-1)\,|\,\widehat{\partial^2_{x_1}\phi}(-1))
%			%		\\&\qquad+
%			(\widehat{\partial_{x_1}\partial_{x_j}\phi},\widehat{\partial_{x_1}\partial_{x_j}\phi})_{L^2_{x_1}}+ \sum_{j=2}^3(\widehat{\partial_{x_1}\partial_{x_j}\phi},\widehat{\partial_{x_1}\partial_{x_j}\phi})_{L^2_{x_1}}\\
			&=\|\widehat{\partial_{x_1}\nabla_x\phi}\|^2_{L^2_{x_1}}.
		\end{align*}
		Similarly, when $\partial=I,\partial_{x_2},\partial_{x_3}$, note that $\partial_{x_1}\widehat{\phi_{a+}}(\pm 1)=0$ and $\widehat{\partial\partial_{x_1}\phi}(\pm1)=0$.
		%	by \eqref{40abcd} boundary condition $\partial_{x_1}\widehat{\phi_{a}}(\pm 1,\bar{k})=0$ and Neumann boundary condition \eqref{97a}, we have 
		Then
		\begin{align*}
			\sum_{j=1}^3((\partial{\partial_{x_j}\phi})^\wedge ,\widehat{\partial_{x_j}\phi_{a+}})_{L^2_{x_1}}
			%			&= 
			%			(\widehat{\partial{\phi}}(1)\,|\,\widehat{\partial_{x_1}\phi_{a+}}(1))
			%			-(\widehat{\partial{\phi}}(-1)\,|\,\widehat{\partial_{x_1}\phi_{a+}}(-1))
			%			\\&\qquad+
			%			(\widehat{\partial\phi} ,-\widehat{\partial^2_{x_1}\phi_{a+}})_{L^2_{x_1}}
			%			+\sum_{j=2}^3(\widehat{\partial{\phi}} ,-\widehat{\partial^2_{x_j}\phi_{a+}})_{L^2_{x_1}}\\
			&= (\widehat{\partial\phi},\partial(\widehat{a_+}-\widehat{a_-}))_{L^2_{x_1}}= (\widehat{\partial\phi},-\partial\widehat{\Delta_x\phi})_{L^2_{x_1}}\\
			%			&=
			%			-(\widehat{\partial{\phi}}(1)\,|\,\widehat{\partial\partial_{x_1}\phi}(1))
			%			+(\widehat{\partial{\phi}}(-1)\,|\,\widehat{\partial\partial_{x_1}\phi}(-1))
			%			\\&\qquad+
			%			(\widehat{\partial\partial_{x_1}\phi},\widehat{\partial\partial_{x_1}\phi})_{L^2_{x_1}}+ \sum_{j=2}^3(\widehat{\partial_{x_1}\partial_{x_j}\phi},\widehat{\partial\partial_{x_j}\phi})_{L^2_{x_1}}\\
			&=\|\widehat{\partial\nabla_x\phi}\|^2_{L^2_{x_1}}.
		\end{align*} Consequently, 
		\begin{align*}
			S_5 = -10\int^T_0\|\widehat{\partial\nabla_x\phi}\|^2_{L^2_{x_1}}\,dt.
		\end{align*}
		Applying the same arguments in \eqref{118} and \eqref{116}, one can obtain $S_6=0$. Therefore, combining the above estimates, taking summation of \eqref{100a} over $|\alpha|\le1$, and then taking the square root and summation over $\bar{k}\in\Z^2$, and finally letting $\eta>0$ suitably small, we have 
		\begin{multline}\label{122c}
			\sum_{|\alpha|\le1}\Big(\|\widehat{\partial^\alpha a_+}-\widehat{\partial^\alpha a_-}\|_{L^1_{\bar{k}}L^2_TL^2_{x_1}}  + \|\widehat{\partial^\alpha\nabla_x\phi}\|_{L^1_{\bar{k}}L^2_TL^2_{x_1}}
			\Big)
			\\\lesssim \sum_{|\alpha|\le1}\Big(\|\widehat{\partial^\alpha f}(T)\|_{L^1_{\bar{k}}L^2_TL^2_{x_1,v}}
			+\|\widehat{\partial^\alpha f_0}\|_{L^1_{\bar{k}}L^2_TL^2_{x_1,v}}+\|\widehat{\partial^\alpha b}\|_{L^1_{\bar{k}}L^2_TL^2_{x_1}}
			\\+\|\{\I-\P\}\widehat{\partial^\alpha  f}\|_{L^1_{\bar{k}}L^2_TL^2_{x_1}L^2_D}
			+\|(\widehat{\partial^\alpha g},\zeta)_{L^2_v}\|_{L^1_{\bar{k}}L^2_TL^2_{x_1}}\Big).
		\end{multline}
		
		\medskip \noindent{\bf Step 4. Estimate on $\widehat{a_+}(t,x_1,\bar{k})+\widehat{a_-}(t,x_1,\bar{k})$.}
		Similar to the above case, we choose the following two test functions
		\begin{align*}
			\widehat{\Phi} = \widehat{\Phi_{a\pm}} = (|v|^2-10)\big(v\cdot\widehat{\nabla_{x_1,\bar{x}}\phi_{a\pm}}(t,x_1,\bar{k})\big)\mu^{1/2},
		\end{align*}
		where $\phi_a = (\phi_{a+}(x_1,\bar{k}),\phi_{a_-}(x_1,\bar{k}))$ solves 
		\begin{equation}\label{102b}\begin{aligned}
				-\partial^2_{x_1}\widehat{\phi_{a+}}+|\bar{k}|^2\widehat{\phi_{a+}} = \widehat{\partial a_+}+\widehat{\partial a_-},\\	-\partial^2_{x_1}\widehat{\phi_{a-}}+|\bar{k}|^2\widehat{\phi_{a-}} = \widehat{\partial a_+}+\widehat{\partial a_-},
			\end{aligned}
		\end{equation}
		with boundary condition $\widehat{\phi_{a}}(\pm 1,\bar{k})=0$
		for the case $\partial=\partial_{x_1}$, and 
		$\partial_{x_1}\widehat{\phi_{a}}(\pm 1,\bar{k})=0$
		for the case $\partial=I,\partial_{x_2},\partial_{x_3}$. When $|\bar{k}|=0$ and $\pa=I$, \eqref{102b} is a pure Neumann boundary problem and we need $\int^1_{-1}\int_{\T^2}a_+-a_-\,d\bar{x}dx_1$ to ensure its existence and we assume $\phi_{a\pm}$ has zero mean in this case. 
		
		For the left hand side of \eqref{100a} , taking summation on $\pm$, we have 
		\begin{align*}
			&\quad\,-\sum_{\pm}\int^T_0(\widehat{\partial \PP f},v\cdot\widehat{\nabla_{x_1,\bar{x}}\Phi_{a\pm}})_{L^2_{x_1,v}} \,dt\\
			&= -\sum_{\pm}\sum_{j,m=1}^3\int^T_0(\widehat{\partial a_\pm}+\widehat{\partial b}\cdot v+(|v|^2-3)\widehat{\partial c} ,v_jv_m(|v|^2-10)\mu{(\partial_{x_j}\partial_{x_m}\phi_{a\pm})^\wedge})_{L^2_{x_1,v}} \,dt\\
			&= \sum_{\pm}\sum_{j=1}^3\int^T_0(\widehat{\partial a_\pm} ,\widehat{-\partial^2_j\phi_{a\pm}})_{L^2_{x_1}} \,dt = \int^T_0\|\widehat{\partial a_+}+\widehat{\partial a_-}\|^2_{L^2_{x_1}} \,dt.
		\end{align*}
		Following the same argument as in the Step 3 and using the first equation of \eqref{19}, we have 
		\begin{align*}
			|\sum_\pm S_2|\lesssim \sum_{|\alpha|\le 1}\int^T_0\|\{\I-\P\}\widehat{\partial^\alpha  f}\|^2_{L^2_{x_1}L^2_D}+\sum_{|\alpha|\le 1}\int^T_0\|\widehat{\partial^\alpha b}\|^2_{L^2_{x_1}}\,dt,
		\end{align*}
		and 
		\begin{multline*}
			|S_1|+|S_3|+|S_4|+|S_5|\lesssim \|\widehat{\partial f}(T)\|^2_{L^2_{x_1,v}}+\|\widehat{\partial f_0}\|^2_{L^2_{x_1,v}}+ \eta\int^T_0\|\widehat{\partial a_+}+\widehat{\partial a_-}\|^2_{L^2_{x_1}}\,dt\\+C_\eta \int^T_0\|\widehat{\partial E}\|^2_{L^2_{x_1}}\,dt+C_\eta\int^T_0\|\{\I-\P\}\widehat{\nabla_x f}\|^2_{L^2_{x_1}L^2_D}\,dt+C_\eta\int^T_0\|(\widehat{\partial g},\zeta)_{L^2_v}\|^2_{L^2_{x_1}}\,dt.
		\end{multline*}
		Also, similar to Step 3, $S_6$ vanishes by using the boundary condition for $\phi_{a}$. Thus, 
		\begin{multline}\label{122df}
			\sum_{|\alpha|\le1}\|\widehat{\partial^\alpha a_+}+\widehat{\partial^\alpha a_-}\|_{L^1_{\bar{k}}L^2_TL^2_{x_1}} 
			\lesssim  \sum_{|\alpha|\le 1}\Big(\|\widehat{\partial^\alpha f}(T)\|_{L^1_{\bar{k}}L^2_{x_1,v}}+\|\widehat{\partial^\alpha f_0}\|_{L^1_{\bar{k}}L^2_{x_1,v}}\\+ \|\widehat{\partial^\alpha E}\|_{L^1_{\bar{k}}L^2_TL^2_{x_1}}+\|\widehat{\partial^\alpha b}\|_{L^1_{\bar{k}}L^2_TL^2_{x_1}}+\|\{\I-\P\}\widehat{ \partial^\alpha f}\|_{L^1_{\bar{k}}L^2_TL^2_{x_1}L^2_D}\\+\|(\widehat{\partial^\alpha g},\zeta)_{L^2_v}\|_{L^1_{\bar{k}}L^2_TL^2_{x_1}}\Big).
		\end{multline}
		
		\medskip\noindent{\bf Step 5. Energy estimate.}
		Now we take the linear combination $\eqref{122c}+\kappa\times\eqref{122af}+\kappa^2\times\eqref{122bf}+\kappa^3\times\eqref{122df}$ and let $\kappa,\eta$ sufficiently small, then 
		\begin{multline}\label{121}
			\sum_{|\alpha|\le 1}\Big(\|\partial^\alpha (\widehat{a_+},\widehat{a_-},\widehat{b},\widehat{c})\|_{L^1_{\bar{k}}L^2_TL^2_{x_1}}+\int^T_0\|\widehat{\partial^\alpha \nabla_x\phi}\|_{L^1_{\bar{k}}L^2_TL^2_{x_1}}\Big)\\
			\lesssim\sum_{|\alpha|\le 1}\Big( \|\widehat{\partial^\alpha f}(T)\|_{L^1_{\bar{k}}L^2_{x_1}L^2_v}+\|\widehat{\partial^\alpha  f_0}\|_{L^1_{\bar{k}}L^2_{x_1}L^2_v}+\|(\widehat{\partial^\alpha g},\zeta)_{L^2_v}\|_{L^1_{\bar{k}}L^2_TL^2_{x_1}}\\+\|\{\I-\P\}\widehat{\partial^\alpha  f}\|_{L^1_{\bar{k}}L^2_TL^2_{x_1}L^2_D}\Big)
			+\|\wh{E}\|_{L^1_{\bar{k}}L^\infty_TL^2_{x_1}}\|\wh{E}\|_{L^1_{\bar{k}}L^2_TL^2_{x_1}}.
		\end{multline}
		Note that $|\widehat{\partial a_+}+\widehat{\partial a_-}|^2+|\widehat{\partial a_+}-\widehat{\partial a_-}|^2 = 2|\widehat{\partial a_+}|^2+2|\widehat{\partial a_+}|^2$. 
		
		\medskip
		Next we estimate $g_\pm$. 
		Using a similar argument as \eqref{84b}, we have 
		\begin{align*}
			&\quad\,\int_{\Z^2}\Big(\int^T_0\big\|\big((\partial(\nabla_x\phi\cdot \nabla_vf_\pm))^\wedge,\zeta(v)\big)_{L^2_{v}}\big\|^2_{L^2_{x_1}} dt\Big)^{1/2}d\Sigma(\bar{k})\\
			&\lesssim\notag \int_{\Z^2}\Big(\int^T_0\Big(\int_{\Z^2}\big(\|\widehat{\partial E}(\bar{k}-\bar{l})\|_{L^2_{x_1}}\|\widehat{f}(\bar{l})\|_{L^2_{x_1}L^2_{D}}\\
			&\qquad\qquad\qquad\qquad\qquad\qquad+\|\widehat{ E}(\bar{k}-\bar{l})\|_{L^2_{x_1}}\|\widehat{\partial f}(\bar{l})\|_{L^2_{x_1}L^2_{D}}\big)\,d\Sigma(\bar{l})\Big)^2dt\Big)^{1/2}d\Sigma(\bar{k})\\
			&\lesssim \sum_{|\alpha|\le 1}\|\widehat{\partial^\alpha  E}\|_{L^1_{\bar{k}}L^\infty_TL^2_{x_1}}\sum_{|\alpha|\le 1}\|\widehat{\partial^\alpha f}\|_{L^1_{\bar{k}}L^2_TL^2_{x_1}L^2_{D}}.
		\end{align*}
		Similar to \eqref{2.23} and \eqref{2.24}, we have 
		\begin{multline*}
			\int_{\Z^2}\Big(\int^T_0\big|\big(\frac{1}{2}((\partial (\nabla_x\phi\cdot vf_\pm))^\wedge,\zeta(v)\big)_{L^2_{x_1,v}}\big|^2 dt\Big)^{1/2}d\Sigma(\bar{k})
			\\\lesssim \sum_{|\alpha|\le 1}\|\widehat{\partial^\alpha E}\|_{L^1_{\bar{k}}L^\infty_TL^2_{x_1}}\sum_{|\alpha|\le 1}\|\widehat{\partial^\alpha f}\|_{L^1_{\bar{k}}L^2_TL^2_{x_1}L^2_{D}},
		\end{multline*}
		and 
		\begin{multline*}
			\int_{\Z^2}\Big(\int^T_0\big|\big((\partial \Gamma_\pm(f,f))^\wedge,\zeta(v)\big)_{L^2_{x_1,v}}\big|^2dt\Big)^{1/2}d\Sigma({\bar{k}})
			\\\lesssim \sum_{|\alpha|\le 1}\|\widehat{\partial^\alpha f}\|_{L^1_{\bar{k}}L^\infty_TL^2_{x_1}L^2_v}\sum_{|\alpha|\le 1}\|\widehat{\partial^\alpha f}\|_{L^1_{\bar{k}}L^2_TL^2_{x_1}L^2_{D}}.
		\end{multline*}
		Plugging the above three estimates into \eqref{121}, we have 
		\begin{multline*}\notag
			\quad\,\sum_{|\alpha|\le 1} \|\partial^\alpha(\widehat{a_+},\widehat{a_-},\widehat{b},\widehat{c})\|_{L^1_{\bar{k}}L^2_TL^2_{x_1}}+\sum_{|\alpha|\le 1}\|\widehat{\partial^\alpha E}\|_{L^1_{\bar{k}}L^2_TL^2_{x_1}}\\
			\notag\lesssim \sum_{|\alpha|\le 1}\big(\|\widehat{\partial^\alpha f}(T)\|_{L^1_{\bar{k}}L^2_{x_1,v}}+\|\widehat{\partial^\alpha f_0}\|_{L^1_{\bar{k}}L^2_{x_1,v}}\big)+\sum_{|\alpha|\le1}\|\{\I-\P\}\widehat{\partial^\alpha f}\|_{L^1_{\bar{k}}L^2_TL^2_{x_1}L^2_D}\\
			\notag\qquad+\sum_{|\alpha|\le 1}\big(\|\widehat{\partial^\alpha  E}\|_{L^1_{\bar{k}}L^\infty_TL^2_{x_1}}+\|\widehat{\partial^\alpha f}\|_{L^1_{\bar{k}}L^\infty_TL^2_{x_1}L^2_v}\big)\sum_{|\alpha|\le 1}\|\widehat{\partial^\alpha f}\|_{L^1_{\bar{k}}L^2_TL^2_{x_1}L^2_{D}}\\
			+\|\wh{E}\|_{L^1_{\bar{k}}L^\infty_TL^2_{x_1}}\|\wh{E}\|_{L^1_{\bar{k}}L^2_TL^2_{x_1}}.
		\end{multline*}This completes the proof of Theorem \ref{Thm41f}. 
		\qe\end{proof}

	\section{Proof of the main result in torus}\label{SecTorus}
	
	In this section, we shall obtain the global existence and large time behavior for system \eqref{1} in torus. 
	Assume $\gamma\ge -2$ for Landau case and $\gamma+2s\ge 1$, $1/2\le s<1$ for Boltzmann case. 
%	We begin with the trilinear estimate on $g$, which is given by \eqref{40a}. 
	
	%\subsection{Trilinear Estimates on Torus}\label{subsec41}

	%Assume $\gamma\ge -2$ for Landau case and $\gamma+2s\ge 1$, $1/2\le s<1$ for Boltzmann case. 
Recall that for VPL case, we let $\vt = -\gamma$ when $-2\le \gamma<-1$ and $q = 0$ when $\gamma\ge -1$. For VPB case, we let $q=0$ when $\gamma+2s\ge 1$, $\frac{1}{2}\le s<1$. Then $\vt\in[1,2]$ and we define weight $w$ as in \eqref{w2}, 
	where $q\ge 0$ and we restrict $0<q<1$ when $\vt=2$. 
	To obtain the microscopic estimate, we take Fourier transform of \eqref{1} on $x$ to obtain  
	\begin{align}\label{30a}
		\partial_t\widehat{f_\pm}+iv\cdot k\widehat{f_\pm}\pm\frac{1}{2}(\nabla_x\phi\cdot vf_\pm)^\wedge \mp (\nabla_x\phi\cdot\nabla_vf_\pm)^\wedge \pm \widehat{\nabla_x\phi}\cdot v\mu^{1/2} - L_\pm \widehat{f} = \widehat{\Gamma_\pm(f,f)}.
	\end{align}
	Denote $h=e^{\delta t}f$. Multiplying \eqref{30a} with $e^{\delta t}$ and $w$, we have 
	\begin{multline}\label{30b}
		\partial_t(w\widehat{h_\pm})+\frac{qN\<v\>^\vt}{(1+t)^{N+1}}w\widehat{h_\pm}+iv\cdot kw\widehat{h_\pm}\pm\frac{1}{2}(\nabla_x\phi\cdot vwh_\pm)^\wedge\mp (\nabla_x\phi\cdot w\nabla_vh_\pm)^\wedge \\ \pm e^{\delta t}\widehat{\nabla_x\phi}\cdot vw\mu^{1/2} - wL_\pm \widehat{h} = e^{-\delta t}w\widehat{\Gamma_\pm(h,h)}+\delta wh_\pm.
	\end{multline}
	Taking the inner product of \eqref{30b} with $w\widehat{h_\pm}$ over $\R^3_v$ and taking the real part, we have
	\begin{align}\label{31a}\notag
		&\quad\,\frac{1}{2}\partial_t|w\widehat{h_\pm}|^2_{L^2_v}+\frac{qN}{(1+t)^{N+1}}|w\<v\>^{\frac{\vt}{2}}\widehat{h_\pm}|^2_{L^2_v}\pm\frac{1}{2}\Re\big((\nabla_x\phi\cdot vwh_\pm)^\wedge,w\widehat{h_\pm}\big)_{L^2_v}\\
		&\mp \notag\Re\big((\nabla_x\phi\cdot w\nabla_vh_\pm)^\wedge,w\widehat{h_\pm}\big)_{L^2_v} \pm \Re\big(e^{\delta t}\widehat{\nabla_x\phi}\cdot vw\mu^{1/2},w\widehat{h_\pm}\big)_{L^2_v} - \Re\big(w^2L_\pm \widehat{h},\widehat{h_\pm}\big)_{L^2_v}\\
		& = \Re\big(e^{-\delta t}w^2\widehat{\Gamma_\pm(h,h)},\widehat{h_\pm}\big)_{L^2_v}+\delta |wh_\pm|^2_{L^2_v}.
	\end{align} 
	We will take summation on \eqref{31a} over $\pm$, integration over $t\in[0,T]$, the square root and integration over $k\in\Z^3$. So, we write the following estimates to control the trouble terms.

	For the third term on the left hand side of \eqref{31a}, recalling $E=-\nabla_x\phi$, we have 
	\begin{align*}\notag
		&\quad\,\Big|\int^T_0\frac{1}{2}\big((\nabla_x\phi\cdot vh_\pm)^\wedge,w^2\widehat{h_\pm}\big)_{L^2_v}\,dt\Big|\\
		%	&\lesssim\notag \int^T_0\int_{\R^3}\int_{\Z^3}|\widehat{E}(k-l)|\<v\>w^2|\widehat{h}(l)|\,d\Sigma(l)|\widehat{h}(k)|\,dvdt\\
		&\notag\lesssim \int^T_0\int_{\Z^3}|e^{\delta t}\widehat{E}(k-l)||e^{-\frac{\delta t}{2}}w\<v\>^{\frac{1}{2}}\widehat{h}(l)|_{L^2_v}\,d\Sigma(l)|e^{-\frac{\delta t}{2}}w\<v\>^{\frac{1}{2}}\widehat{h}(k)|_{L^2_v}\,dt\\
		&\lesssim \int^T_0C_\eta\Big(\int_{\Z^3}|e^{\delta t}\widehat{E}(k-l)||e^{-\frac{\delta t}{2}}w\<v\>^{\frac{1}{2}}\widehat{h}(l)|_{L^2_v}\,d\Sigma(l)\Big)^2+\eta^2|e^{-\frac{\delta t}{2}}w\<v\>^{\frac{1}{2}}\widehat{h}(k)|^2_{L^2_v}\,dt.
	\end{align*}
	Taking the square root and summation over $k\in\Z^3$, we have 
	\begin{align}\notag\label{32a}
		&\quad\,\int_{\Z^3}\Big|\int^T_0\frac{1}{2}\big((\nabla_x\phi\cdot vh_\pm)^\wedge,w^2\widehat{h_\pm}\big)_{L^2_v}\,dt\Big|^{\frac{1}{2}}d\Sigma(k)\\
		&\notag\lesssim C_\eta\int_{\Z^3}\int_{\Z^3}\Big(\int^T_0|e^{\delta t}\widehat{E}(k-l)|^2|e^{-\frac{\delta t}{2}}w\<v\>^{\frac{1}{2}}\widehat{h}(l)|_{L^2_v}^2\,dt\Big)^{\frac{1}{2}}\,d\Sigma(l)d\Sigma(k)\\
		&\notag\qquad\qquad\qquad\qquad\qquad\qquad\qquad\qquad\qquad\qquad+\eta\|e^{-\frac{\delta t}{2}}w\<v\>^{\frac{1}{2}}\widehat{h}(k)\|_{L^1_kL^2_TL^2_v}\\
		&\notag\lesssim C_\eta\int_{\Z^3}\int_{\Z^3}\sup_{0\le t\le T}|e^{\delta t}\widehat{E}(k-l)|\Big(\int^T_0|e^{-\frac{\delta t}{2}}w\<v\>^{\frac{1}{2}}\widehat{h}(l)|_{L^2_v}^2\,dt\Big)^{\frac{1}{2}}\,d\Sigma(l)d\Sigma(k)\\
		&\notag\qquad\qquad\qquad\qquad\qquad\qquad\qquad\qquad\qquad\qquad+\eta\|e^{-\frac{\delta t}{2}}w\<v\>^{\frac{1}{2}}\widehat{h}(k)\|_{L^1_kL^2_TL^2_v}\\
		&\lesssim C_\eta\|e^{\delta t}\widehat{E}\|_{L^1_kL^\infty_T}\|e^{-\frac{\delta t}{2}}w\<v\>^{\frac{1}{2}}\widehat{h}\|_{L^1_kL^2_TL^2_v}+\eta\|e^{-\frac{\delta t}{2}}w\<v\>^{\frac{1}{2}}\widehat{h}\|_{L^1_kL^2_TL^2_v},
	\end{align}where we used Young's inequality for integration and Fubini's Theorem. 
	For Landau case, the forth term on the left hand of \eqref{31a} with integration on $t\in[0,T]$ can be controlled by 
	\begin{align*}
		&\quad\,\Big|\int^T_0\big((\nabla_x\phi\cdot \nabla_vh_\pm)^\wedge,w^2\widehat{h_\pm}\big)_{L^2_v}\,dt\Big|\\
		&\lesssim \int^T_0\int_{\Z^3}|e^{\delta t}\widehat{E}(k-l)||w\<v\>^{\gamma/2}\nabla_v\widehat{h}(l)|_{L^2_v}\,d\Sigma(l)|e^{-\delta t}w\<v\>^{-\gamma/2}\widehat{h}(k)|_{L^2_v}\,dt\\
		&\lesssim \int^T_0C_\eta\Big(\int_{\Z^3}|e^{\delta t}\widehat{E}(k-l)||e^{-\frac{\delta t}{2}}w\<v\>^{\gamma/2}\nabla_v\widehat{h}(l)|_{L^2_v}\,d\Sigma(l)\Big)^2+\eta^2|e^{-\frac{\delta t}{2}}w\<v\>^{-\gamma/2}\widehat{f}(k)|^2_{L^2_v}\,dt.
	\end{align*}
	Taking the square root and integration over $k\in\Z^3$, similar to \eqref{32a}, we have 
	\begin{align}\label{32b}\notag
		&\quad\,\int_{\Z^3}\Big|\int^T_0\frac{1}{2}\big((\nabla_x\phi\cdot \nabla_vh_\pm)^\wedge,w^2\widehat{h_\pm}\big)_{L^2_v}\,dt\Big|^{1/2}\,d\Sigma(k)\\
		&\lesssim C_\eta\|e^{\delta t}\widehat{E}\|_{L^1_kL^\infty_T}\|e^{-\frac{\delta t}{2}}\widehat{h}\|_{L^1_kL^2_TL^2_{D,w}}+\eta\|e^{-\frac{\delta t}{2}}w\<v\>^{-\gamma/2}\widehat{h}\|_{L^1_kL^2_TL^2_v}.
	\end{align}
	For Boltzmann case, we deal with the forth term on the left hand of \eqref{31a} by interpolation. Indeed, since $q=0$ for Boltzmann case, we have 
	\begin{align}\label{116a}\notag
		&\quad\,\Big|\int^T_0\big((\nabla_x\phi\cdot \nabla_vh_\pm)^\wedge,\widehat{h_\pm}\big)_{L^2_v}\,dt\Big|\\
		&\notag= \Big|\int^T_0\int_{\Z^3}\big(\widehat{E}(k-l) \nabla_v\widehat{h}(l),\widehat{h}(k)\big)_{L^2_v}\,d\Sigma(l)dt\Big|\\
		%-\big(\widehat{E}(k-l) \nabla_vw\widehat{h}(l),w\widehat{h}(k)\big)_{L^2_v}\,d\Sigma(l)dt\Big|\\
		&\lesssim \int^T_0\int_{\Z^3}|e^{\delta t}\widehat{E}(k-l)||e^{-\frac{\delta t}{2}}\<v\>^{\frac{\gamma}{2}}\<D_v\>^s\widehat{h}(l)|_{L^2_v}|e^{-\frac{\delta t}{2}}\<v\>^{-\frac{\gamma}{2}}\<D_v\>^{1-s}\widehat{h}(k)|_{L^2_v}\,d\Sigma(l)dt.
		%	&\qquad 
		%	+\int^T_0\int_{\Z^3}|e^{\delta t}\widehat{E}(k-l)| \Big|e^{-\frac{\delta t}{2}}\frac{q\<v\>^{\vt-1}}{(1+t)^N}w\widehat{h}(l)\Big|_{L^2_v}|e^{-\frac{\delta t}{2}}w\widehat{h}(k)|_{L^2_v}\,d\Sigma(l)dt
	\end{align}
	Here we shall deal with the term $|e^{-\frac{\delta t}{2}}\<v\>^{-\frac{\gamma}{2}}\<D_v\>^{1-s}\widehat{h}(k)|_{L^2_v}$ in \eqref{116a}. By Young's inequality, 
	\begin{align*}
		\<v\>^{-\frac{\gamma}{2}}\<\eta\>^{1-s}&\lesssim \big(\<v\>^{\frac{\gamma(1-s)}{2s}}\<\eta\>^{1-s}\big)^{\frac{s}{1-s}}+\big(\<v\>^{-\frac{\gamma}{2s}}\big)^{\frac{s}{2s-1}}\\
		&\lesssim \<v\>^{\frac{\gamma}{2}}\<\eta\>^s+\<v\>^{-\frac{\gamma}{2(2s-1)}}, 
	\end{align*}where $\eta$ is the Fourier variable of $v$. 
	Similar calculation can be applied on derivatives of $\<v\>^{-\frac{\gamma}{2}}\<\eta\>^{1-s}$. Thus, 
	$\<v\>^{-\frac{\gamma}{2}}\<\eta\>^{1-s}$ belongs to symbol class $S(\<v\>^{\frac{\gamma}{2}}\<\eta\>^s+\<v\>^{-\frac{\gamma}{2(2s-1)}})$.
	Applying \cite[Lemma 2.4 and Corollary 2.5]{Deng2020a}, we have 
	\begin{align*}
		|\<v\>^{-\frac{\gamma}{2}}\<D_v\>^{1-s}(\widehat{h})(k)|_{L^2_v}\lesssim |\<v\>^{\frac{\gamma}{2}}\<D_v\>^s\widehat{h}(k)|_{L^2_v}+|\<v\>^{-\frac{\gamma}{2(2s-1)}}\widehat{h}(k)|_{L^2_v}.
	\end{align*}
	Recall that $\frac{1}{2} \ge -\frac{\gamma}{2(2s-1)}$ for Boltzmann case in our setting. Thus, substituting the above estimate into \eqref{116a}, taking square root and summation over $k\in\Z^3$ of the resultant estimate, we have 
	\begin{align}\label{217}
		&\notag\quad\,\int_{\Z^3}\Big|\int^T_0\big((\nabla_x\phi\cdot \nabla_vh_\pm)^\wedge,\widehat{h_\pm}\big)_{L^2_v}\,dt\Big|^{1/2}\,d\Sigma(k)\\
		&\notag\lesssim \int_{\Z^3}\Big(\int^T_0\Big(\int_{\Z^3}|e^{\delta t}\widehat{E}(k-l)||e^{-\frac{\delta t}{2}}\<v\>^{\frac{\gamma}{2}}\<D_v\>^s\widehat{h}(l)|_{L^2_v}\,d\Sigma(l)\Big)^2\\
		&\notag\qquad\qquad\qquad\qquad+\eta^2|e^{-\frac{\delta t}{2}}\<v\>^{\frac{\gamma}{2}}\<D_v\>^s\widehat{h}(k)|^2_{L^2_v}+\eta^2|e^{-\frac{\delta t}{2}}\<v\>^{\frac{1}{2}}\widehat{h}(k)|_{L^2_v}\,dt \\
		%	&\notag\qquad 
		%	+\int^T_0\Big(\int_{\Z^3}|e^{\delta t}\widehat{E}(k-l)| |e^{-\frac{\delta t}{2}}q\<v\>^{\frac{\vt}{2}}\widehat{h}(l)|_{L^2_v}\,d\Sigma(l)\Big)^2+\eta^2|e^{-\frac{\delta t}{2}}\widehat{h}(k)|^2_{L^2_v}\,dt\Big)^{1/2}\,d\Sigma(k)\\
		&\lesssim \|e^{\delta t}\widehat{E}\|_{L^1_kL^\infty_T}\|e^{-\frac{\delta t}{2}}h\|_{L^1_kL^2_TL^2_{D}}
		+\eta\big(\|e^{-\frac{\delta t}{2}}\widehat{h}\|_{L^1_kL^2_TL^2_{D}}+\|e^{-\frac{\delta t}{2}}\<v\>^{\frac{1}{2}}\widehat{h}\|_{L^1_kL^2_TL^2_v}\big). 
	\end{align}
	For the fifth term on the left hand of \eqref{31a} when $q=0$, we take the summation over $\pm$ to obtain 
	\begin{multline}\label{33a}
		\sum_\pm\pm\big(\widehat{\nabla_x\phi}\cdot v\mu^{1/2},\widehat{f_\pm}\big)_{L^2_v}= \widehat{\nabla_x\phi}\cdot\overline{\widehat{G}}
		= \widehat{\phi}\,\overline{-ik\cdot\widehat{G}}
		\\
		=\widehat{\phi}\,\overline{\partial_t(\widehat{a_+}-\widehat{a_-})}
		=\widehat{\phi}\,\overline{ik\cdot\partial_t\widehat{E}}
		=\widehat{\nabla_x\phi}\cdot\overline{\partial_t\widehat{\nabla_x\phi}},
	\end{multline}where the first inequality follows from the definition \eqref{98a} of $G$, the third and fourth inequalities follow from \eqref{27}. Taking the real part of \eqref{33a} and multiplying the weight $e^{2\delta t}$, we have 
	\begin{align}\label{32c}
		\sum_\pm\pm\Re\big(e^{\delta t}\widehat{\nabla_x\phi}\cdot v\mu^{1/2},\widehat{h_\pm}\big)_{L^2_v}&=\frac{1}{2}\partial_t|e^{\delta t}\widehat{E}|^2 - \delta|e^{\delta t}\widehat{E}|^2. 
	\end{align}
	For the fifth left-hand term of \eqref{31a} when $q>0$, we write an upper bound:
	\begin{align}\label{32cc}
		\big|\Re\big(e^{\delta t}\widehat{\nabla_x\phi}\cdot v\mu^{1/2},w^2\widehat{h_\pm}\big)_{L^2_v}\big|\lesssim |e^{\delta t}\widehat{E}|^2+|\<v\>^{\frac{\gamma+2s}{2}}\widehat{h_\pm}|^2_{L^2_{v}}.
	\end{align}
	For the sixth term on the left hand of \eqref{31a}, when $q=0$, we take the summation over $\pm$ to obtain 
	\begin{align}\label{32d}
		\sum_\pm (L_\pm \widehat{h},\widehat{h_\pm})_{L^2_v}\ge \lambda|\{\I-\P\}\widehat{h}|_D^2.
	\end{align}
	When $q>0$, we use \eqref{36c} to deduce that  
	\begin{align}\label{32dd}
		\sum_\pm (w^2L_\pm \widehat{h},\widehat{h_\pm})_{L^2_v}\ge \lambda|\widehat{h}|_{L^2_{D,w}}^2-C|\widehat{h}|^2_{L^2_{B_C}},
	\end{align}for some $\lambda,C>0$. 
	For the first term on the right hand of \eqref{31a}, taking integration over $t\in[0,T]$, by \eqref{35a}, we have 
	\begin{align*}
		&\quad\,\Big|\int^T_0\big(e^{-\delta t}w^2\widehat{\Gamma_\pm(h,h)},\widehat{h_\pm}\big)_{L^2_v}dt\Big|\\
		&= \int^T_0\int_{\Z^3}\big(w^2\Gamma_\pm(\widehat{h}(k-l),\widehat{h}(l)),\widehat{h_\pm}(k)\big)_{L^2_v}d\Sigma(l)dt\\
		&\lesssim \int^T_0\int_{\Z^3}|w\widehat{h}(k-l)|_{L^2_v}|\widehat{h}(l)|_{L^2_{D,w}}|\widehat{h}(k)|_{L^2_{D,w}}d\Sigma(l)dt\\
		&\lesssim \int^T_0C_\eta\Big(\int_{\Z^3}|w\widehat{h}(k-l)|_{L^2_v}|\widehat{h}(l)|_{L^2_{D,w}}d\Sigma(l)\Big)^2+\eta^2|\widehat{h}|^2_{L^2_{D,w}}dt.
	\end{align*}
	Taking the square root and integration over $k\in\Z^3$, using the trick of \eqref{32a}, we have 
	\begin{multline}\label{45a}
		\quad\,\int_{\Z^3}\Big|\int^T_0\big(e^{-\delta t}w^2\widehat{\Gamma_\pm(h,h)},\widehat{h_\pm}\big)_{L^2_v}dt\Big|^{1/2}\,dk
		\\\lesssim C_\eta\|w\widehat{h}\|_{L^1_kL^\infty_TL^2_v}\|\widehat{h}\|_{L^1_kL^2_TL^2_{D,w}}+\eta\|\widehat{h}\|^2_{L^1_kL^2_TL^2_{D,w}}.
	\end{multline}
	Now we take the summation on \eqref{31a} over $\pm$, integration over $t\in[0,T]$, the square root and then integration over $k\in\Z^3$. If $q=0$ in \eqref{w2}, then combining estimates \eqref{32a}, \eqref{32b}, \eqref{217}, \eqref{32c}, \eqref{32d} and \eqref{45a}, we have 
	\begin{align}\label{46a}\notag
		&\quad\,\|\widehat{h}\|_{L^1_kL^\infty_TL^2_v}+\|e^{\delta t}\widehat{E}\|_{L^1_kL^\infty_T}+\|\{\I-\P\}\widehat{h}\|_{L^1_kL^2_TL^2_{D}}\\
		&\lesssim\|\widehat{f_0}\|_{L^1_kL^2_v}+\|\widehat{E_0}\|_{L^1_k}+\notag
		\delta^{1/2}\big(\|e^{\delta t}\widehat{E}\|_{L^1_kL^2_T}+\|\widehat{h}\|_{L^1_kL^2_TL^2_v}\big)+C_\eta\|\widehat{h}\|_{L^1_kL^\infty_TL^2_v}\|\widehat{h}\|_{L^1_kL^2_TL^2_{D}}\\
		&\notag\quad+C_\eta\|e^{\delta t}\widehat{E}\|_{L^1_kL^\infty_T}\big(\|\<v\>^{\gamma/2}\nabla_v\widehat{h}\|_{L^1_kL^2_TL^2_v}+\|e^{-\frac{\delta t}{2}}\<v\>^{1/2}\widehat{h}\|_{L^1_kL^2_TL^2_v}\big)\\
		&\quad+\eta\big(\|e^{-\frac{\delta t}{2}}\<v\>^{-\gamma/2}\widehat{h}\|_{L^1_kL^2_TL^2_v}
		+\|e^{-\frac{\delta t}{2}}\<v\>^{1/2}\widehat{h}\|_{L^1_kL^2_TL^2_v}\big)+\eta\|\widehat{h}\|_{L^1_kL^2_TL^2_{D}}.
	\end{align}
	If $q\neq 0$, then combining estimates \eqref{32a}, \eqref{32b}, \eqref{217}, \eqref{32cc}, \eqref{32dd} and \eqref{45a}, we have 
	\begin{align}\label{46b}\notag
		&\quad\,\|w\widehat{h}\|_{L^1_kL^\infty_TL^2_v}+\sqrt{qN}\|{\<v\>^{\frac{\vt}{2}}}{(1+t)^{-\frac{N+1}{2}}}w\widehat{h}\|_{L^1_kL^2_TL^2_v}
		+\|\widehat{h}\|_{L^1_kL^2_TL^2_{D,w}}\\
		&\lesssim\|w\widehat{f_0}\|_{L^1_kL^2_v}+\|\widehat{E_0}\|_{L^1_k}+\notag
		\delta^{1/2}\big(\|e^{\delta t}\widehat{E}\|_{L^1_kL^2_T}+\|w\widehat{h}\|_{L^1_kL^2_TL^2_v}\big)\\
		&\notag\quad+C_\eta\|e^{\delta t}\widehat{E}\|_{L^1_kL^\infty_T}\big(\|\widehat{h}\|_{L^1_kL^2_TL^2_{D,w}}+\|e^{-\frac{\delta t}{2}}w\<v\>^{1/2}\widehat{h}\|_{L^1_kL^2_TL^2_v}\big)\\
		&\notag\quad+\eta\big(\|e^{-\frac{\delta t}{2}}w\<v\>^{-\gamma/2}\widehat{f}\|_{L^1_kL^2_TL^2_v}
		+\|e^{-\frac{\delta t}{2}}w\<v\>^{1/2}\widehat{h}\|_{L^1_kL^2_TL^2_v}\big) + \|\widehat{h}\|_{L^1_kL^2_TL^2_D}\\
		&\quad+C_\eta\|w\widehat{h}\|_{L^1_kL^\infty_TL^2_v}\|\widehat{h}\|_{L^1_kL^2_TL^2_{D,w}}+\eta\|\widehat{h}\|_{L^1_kL^2_TL^2_{D,w}}+\|e^{\delta t}\widehat{E}\|_{L^1_kL^2_T} .
	\end{align}
	Noticing $-\gamma\le \vt$, we take combination $\eqref{46a}+\delta^{1/2}\times\eqref{46b}$ and let $\eta\ll\delta<1$ sufficiently small to obtain 
	\begin{align}\label{51aa}\notag
		&\quad\,\|\widehat{h}\|_{L^1_kL^\infty_TL^2_v}+\|e^{\delta t}\widehat{E}\|_{L^1_kL^\infty_T}+\|\{\I-\P\}\wh h\|_{L^1_kL^2_TL^2_{D}}+\delta^{1/2}\|w\widehat{h}\|_{L^1_kL^\infty_TL^2_v}\\
		&\quad\notag+\sqrt{\delta qN}\|{\<v\>^{\frac{\vt}{2}}}{(1+t)^{-\frac{N+1}{2}}}w\widehat{h}\|_{L^1_kL^2_TL^2_v}+\delta^{1/2}\|\widehat{h}\|_{L^1_kL^2_TL^2_{D,w}}\\
		&\lesssim\|w\widehat{f_0}\|_{L^1_kL^2_v}+\|\widehat{E_0}\|_{L^1_k}\notag+
		\delta^{1/2}\big(\|e^{\delta t}\widehat{E}\|_{L^1_kL^2_T}+\|w\widehat{h}\|_{L^1_kL^2_TL^2_v}\big)\\
		&\notag\quad+C_\eta\|e^{\delta t}\widehat{E}\|_{L^1_kL^\infty_T}\big(\|\widehat{h}\|_{L^1_kL^2_TL^2_{D,w}}+\|e^{-\frac{\delta t}{2}}w\<v\>^{1/2}\widehat{h}\|_{L^1_kL^2_TL^2_v}\big)\\
		&\quad+C_\eta\|w\widehat{h}\|_{L^1_kL^\infty_TL^2_v}\|\widehat{h}\|_{L^1_kL^2_TL^2_{D,w}}+\eta\|\widehat{h}\|_{L^1_kL^2_TL^2_{D}}.
	\end{align}
	The terms with velocity weight such as $\|e^{-\frac{\delta t}{2}}w\<v\>^{1/2}\widehat{h}\|_{L^1_kL^2_TL^2_v}$ are controlled by extra dissipation term $\|{\<v\>^{\frac{\vt}{2}}}{(1+t)^{-\frac{N+1}{2}}}w\widehat{h}\|_{L^1_kL^2_TL^2_v}$.

	\medskip
	
	Now we are in a position to prove Theorem \ref{Main}. 
	\begin{proof}[Proof of Theorem \ref{Main}]
		Recall that $h=e^{\delta t}f$. 
		Adding the weight $e^{\delta t}$ in \eqref{40}, we have 
		\begin{align*}
			\partial_t\widehat{h_\pm}+iv\cdot k\widehat{h_\pm} \pm e^{\delta t}\widehat{\nabla_x\phi}\cdot v\mu^{1/2} - L_\pm \widehat{h} = e^{\delta t}\widehat{g_\pm} + \delta \widehat{h_\pm}.
		\end{align*}
		Note that, compared to \eqref{40}, the only extra term is $\delta \widehat{h_\pm}$. 
		Then following the same argument in Theorem \ref{Thm21}, we deduce the macroscopic estimate: 
		\begin{multline}\label{53a}
			\|e^{\delta t}(\widehat{a_+},\widehat{a_-},\widehat{b},\widehat{c})\|_{L^1_kL^2_T}+\|e^{\delta t}\widehat{E}\|_{L^1_kL^2_T}\lesssim \|\widehat{h}\|_{L^1_kL^\infty_TL^2_v}+\|\widehat{f_0}\|_{L^1_kL^2_v}\\+\|\{\I-\P\}\widehat{h}\|_{L^1_kL^2_TL^2_{D}} 
			+\|e^{\delta t}\widehat{E}\|_{L^1_kL^\infty_T}\|e^{\delta t}\widehat{E}\|_{L^1_kL^2_T}	\\+\big(\|e^{\delta t}\widehat{E}\|_{L^1_kL^\infty_T}+\|\widehat{h}\|_{L^1_kL^\infty_TL^2_v}\big)\|\widehat{h}\|_{L^1_kL^2_TL^2_{D}}+\delta^{1/2}\|\widehat{h}\|_{L^1_kL^2_TL^2_v},
		\end{multline}
	where we have an extra term $\delta^{1/2}\|\widehat{h}\|_{L^1_kL^2_TL^2_v}$ compared to Theorem \ref{Thm21}. 
		Now we take the combination $\eqref{51aa}+\kappa\times\eqref{53a}$ with sufficiently small $\kappa,\delta,\eta>0$ to obtain 
		\begin{align}\label{52a}
			&\notag\quad\,\|\widehat{h}\|_{L^1_kL^\infty_TL^2_v}+\|e^{\delta t}\widehat{E}\|_{L^1_kL^\infty_T}+\|\wh h\|_{L^1_kL^2_TL^2_{D}}+\|e^{\delta t}\widehat{E}\|_{L^1_kL^2_T}\\
			&\quad\notag+\delta^{1/2}\|w\widehat{h}\|_{L^1_kL^\infty_TL^2_v}+\sqrt{\delta qN}\|{\<v\>^{\frac{\vt}{2}}}{(1+t)^{-\frac{N+1}{2}}}w\widehat{h}\|_{L^1_kL^2_TL^2_v}+\delta^{1/2}\|\widehat{h}\|_{L^1_kL^2_TL^2_{D,w}}\\
			&\lesssim\|w\widehat{f_0}\|_{L^1_kL^2_v}+\|\widehat{E_0}\|_{L^1_k}+\notag
			\delta^{1/2}\|w\widehat{h}\|_{L^1_kL^2_TL^2_v}+\|w\widehat{h}\|_{L^1_kL^\infty_TL^2_v}\|\widehat{h}\|_{L^1_kL^2_TL^2_{D,w}}\\
			&\quad+\|e^{\delta t}\widehat{E}\|_{L^1_kL^\infty_T}\big(\|\widehat{h}\|_{L^1_kL^2_TL^2_{D,w}}+\|e^{-\frac{\delta t}{2}}w\<v\>^{1/2}\widehat{h}\|_{L^1_kL^2_TL^2_v}+\|e^{\delta t}\widehat{E}\|_{L^1_kL^2_T}+\|\widehat{h}\|_{L^1_kL^2_TL^2_{D}}\big).
		\end{align}
		Next we discuss the result in the following two cases. 
		
		\medskip \noindent{\bf Case I: $\gamma\ge -1$ for VPL case and $\gamma+2s\ge 1$, $1/2\le s<1$ for VPB case.} 
		In this case, we have $|\<v\>^{-\frac{\gamma}{2}}(\cdot)|_{L^2_v}\lesssim |\<v\>^{\frac{1}{2}}(\cdot)|_{L^2_v}\lesssim |\cdot|_{L^2_{D}}$ and $\|\widehat{h}\|_{L^1_kL^2_TL^2_v}\lesssim \|\widehat{h}\|_{L^1_kL^2_TL^2_{D}}$. 
		Choosing $q=0$ in \eqref{w2} and $\delta$ in \eqref{52a} sufficiently small, we obtain 
		\begin{equation}\label{6.3}
			\E_T+\D_T
			\lesssim\|\widehat{f_0}\|_{L^1_kL^2_v}+\|\widehat{E_0}\|_{L^1_k}
			+\E_T\D_T,
		\end{equation}where $\E_T$ and $\D_T$ are defined by \eqref{defe} and \eqref{defd} respectively. 
		Under the smallness of $\|\widehat{f_0}\|_{L^1_kL^2_v}+\|\widehat{E_0}\|_{L^1_k}$, it's now standard to apply the continuity argument to obtain 
		\begin{align*}
			\E_T+\D_T
			\lesssim\|\widehat{f_0}\|_{L^1_kL^2_v}+\|\widehat{E_0}\|_{L^1_k}.
		\end{align*}
		This completes the case $\gamma\ge -1$ for VPL systems and $\gamma+2s\ge 1$, $\frac{1}{2}\le s<1$ for VPB systems.

		\medskip \noindent{\bf Case II: $-2\le\gamma<-1$ for VPL case.} We need to apply the estimate with time-velocity weight. 
		In this case, we have $\|w\widehat{h}\|_{L^1_kL^2_TL^2_v}\lesssim \|\widehat{h}\|_{L^1_kL^2_TL^2_{D,w}}$ and 
		\begin{equation*}
			\|e^{-\frac{\delta t}{2}}w\<v\>^{1/2}\widehat{h}\|_{L^1_kL^2_TL^2_v}\lesssim \sqrt{\delta qN}\|{\<v\>^{\frac{\vt}{2}}}{(1+t)^{-\frac{N+1}{2}}}w\widehat{h}\|_{L^1_kL^2_TL^2_v}.
		\end{equation*} Then letting $\delta>0$ small enough in \eqref{52a}, we obtain 
		\begin{align}\label{6.4}
			\E_T+\D_T+\E_{T,w}+\D_{T,w}
			&\lesssim\|w\widehat{f_0}\|_{L^1_kL^2_v}+\|\widehat{E_0}\|_{L^1_k}+(\E_{T,w}+\E_T)(\D_{T,w}+\D_T),
		\end{align}
		where $\E_{T,w}$ and $\D_{T,w}$ are defined by \eqref{ETw} and \eqref{DTw} respectively. Under the smallness of $\|w\widehat{f_0}\|_{L^1_kL^2_v}+\|\widehat{E_0}\|_{L^1_k}$, we obtain \eqref{energytorus} by using continuity argument and local existence from Section \ref{Sec_loc}. 
		\qe\end{proof}

	Next we prove the propagation of initial regularity. 
	\begin{proof}
		[Proof of Theorem \ref{spatial_regurlarity}]
		Following closely the proofs of from \eqref{32a} to \eqref{217} and Theorem \ref{Thm21}, one can show the analogous higher order trilinear estimates. That is, for $m\ge 0$, 
		\begin{multline*}
			\int_{\Z^3}\Big|\int^T_0\frac{1}{2}\big((\nabla_x\phi\cdot vwf_\pm)^\wedge,\<k\>^{2m}\widehat{wf_\pm}\big)_{L^2_v}\,dt\Big|^{1/2}d\Sigma(k)\\
			\qquad\lesssim{C_\eta}{} \|e^{\delta t}\widehat{\nabla_x\phi}\|_{L^1_{k,m}L^\infty_T}\|e^{-\frac{\delta t}{2}}\<v\>^{1/2}\widehat{wf}\|_{L^1_{k,m}L^2_TL^2_v}+{\eta}{}\|e^{-\frac{\delta t}{2}}\<v\>^{1/2}\widehat{wf}\|_{L^1_{k,m}L^2_TL^2_v},
		\end{multline*} 
	\begin{multline*}
			\int_{\Z^3}\Big|\int^T_0\big((\nabla_x\phi\cdot w\nabla_vf_\pm)^\wedge,\<k\>^{2m}\widehat{wf_\pm}\big)_{L^2_v}\,dt\Big|^{1/2}\,d\Sigma(k)\\
			\qquad\lesssim {C_\eta}\|e^{\delta t}\widehat{\nabla_x\phi}\|_{L^1_{k,m}L^\infty_T}\|e^{-\frac{\delta t}{2}}\<v\>^{\frac{\gamma}{2}}(w\nabla_v{f})^\wedge\|_{L^1_{k,m}L^2_TL^2_v}+\eta\|e^{-\frac{\delta t}{2}}\<v\>^{-\frac{\gamma}{2}}\widehat{wf}\|_{L^1_{k,m}L^2_TL^2_v},
		\end{multline*} 
	\begin{multline*}
			\int_{\Z^3}\Big|\int^T_0(\widehat{w\nabla_x\phi}\cdot v\mu^{1/2},\<k\>^{2m}\widehat{wf_\pm})_{L^2_v}dt\Big|^{1/2}d\Sigma(k)\\\le C_\eta\|\widehat{\nabla_x\phi}\|_{L^1_{k,m}L^2_T}+\eta\|\widehat{wf_\pm}\mu^{1/8}\|_{L^1_{k,m}L^2_TL^2_v},
		\end{multline*} 
		and 
		\begin{multline*}
			\quad\,\|e^{\delta t}(\widehat{a_+},\widehat{a_-},\widehat{b},\widehat{c})\|_{L^1_{k,m}L^2_T}+\|\widehat{E}\|_{L^1_{k,m}L^2_T}\lesssim \|\widehat{h}\|_{L^1_{k,m}L^\infty_TL^2_v}+\|\widehat{f_0}\|_{L^1_{k,m}L^2_v}\\+\|\{\I-\P\}\widehat{h}\|_{L^1_{k,m}L^2_TL^2_{D}} +\|e^{\delta t}\widehat{E}\|_{L^1_kL^\infty_T}\|e^{\delta t}\widehat{E}\|_{L^1_kL^2_T}\\
			+\big(\|e^{\delta t}\widehat{E}\|_{L^1_{k,m}L^\infty_T}+\|\widehat{h}\|_{L^1_{k,m}L^\infty_TL^2_v}\big)\|\widehat{h}\|_{L^1_{k,m}L^2_TL^2_{D}}+\delta^{1/2}\|\widehat{h}\|_{L^1_kL^2_TL^2_v}.
		\end{multline*}
		Taking the $L^2_v$ inner product of of \eqref{30b} with $\<k\>^{2m}\widehat{wh_\pm}$, we have 
		\begin{align*}\notag
			\frac{1}{2}\partial_t&|\<k\>^m\widehat{wh_\pm}|^2_{L^2_v}  +\frac{qN}{(1+t)^{N+1}}|w\<k\>^m\<v\>^{\frac{\vt}{2}}\widehat{h_\pm}|^2_{L^2_v}  - \Re(\<k\>^{2m}\widehat{wL_\pm h},\widehat{wh_\pm})_{L^2_v} \\&\le \mp \frac{1}{2}\Re(\<k\>^{2m}(\nabla_x\phi\cdot vwh_\pm)^\wedge,\widehat{wh_\pm})_{L^2_v}\pm \Re(\<k\>^{2m}(\nabla_x\phi\cdot w\nabla_vh_\pm)^\wedge,\widehat{wh_\pm})_{L^2_v}\\&\qquad \mp \Re(\<k\>^{2m}e^{\delta t}\widehat{w\nabla_x\phi}\cdot v\mu^{1/2},\widehat{wh_\pm})_{L^2_v} +\Re(\<k\>^{2m}e^{-\delta t}(w\Gamma_{\pm}(h,h))^\wedge,\widehat{wh_\pm})_{L^2_v}.
		\end{align*}
		Using the same argument deriving \eqref{6.4}, under the smallness of $\|\widehat{wf_0}\|_{L^1_{k,m}L^\infty_TL^2_v}+\|e^{\delta t}\widehat{E_0}\|_{L^1_{k,m}L^\infty_T}$, we have 
		\begin{align*}
			&\quad\,\|\widehat{wh}\|_{L^1_{k,m}L^\infty_TL^2_v}  +\sqrt{\delta qN}\|\<v\>^{\frac{\vt}{2}}(1+t)^{-\frac{N+1}{2}}\widehat{wh}\|_{L^1_{k,m}L^2_TL^2_v}  +\|\widehat{h}\|_{L^1_{k,m}L^2_TL^2_{D,w}} \\
			&\qquad\notag
			+\|\widehat{h}\|_{L^1_{k,m}L^\infty_TL^2_v}+\|e^{\delta t}\widehat{E}\|_{L^1_{k,m}L^\infty_T}+\|\wh h\|_{L^1_{k,m}L^2_TL^2_{D}}+\|e^{\delta t}\widehat{E}\|_{L^1_{k,m}L^2_T}\\
			&\lesssim \|\widehat{wf_0}\|_{L^1_{k,m}L^2_v}+\|\widehat{E_0}\|_{L^1_{k,m}},
		\end{align*}
		which implies \eqref{regularity} for the case $-1>\gamma\ge-2$. The case of $\gamma\ge-1$ follows closely to the argument that was used to derive \eqref{6.3} and we omit the proof for brevity. This completes the proof of Theorem \ref{spatial_regurlarity}.
\end{proof}

	\section{Proof of the main result in finite channel}\label{SecFinite}
	
	In this section, we write $\widehat{\cdot}=\F_{\bar{x}}$ to be the Fourier transform on $\bar{x}\in\T^2$. 
	Recall that we define $h=e^{\delta t}f$. Let $\partial=\partial^\alpha$ with $|\alpha|\le 1$. Acting $\partial$ on \eqref{1} and multiplying $we^{\delta t}$, we have 
	\begin{multline*}\notag\partial_tw\pa h_\pm +\frac{qN\<v\>^\vt}{(1+t)^{N+1}}w\partial h_\pm + v\cdot\nabla_xw(\partial h_\pm) \pm \frac{1}{2}w\partial(\nabla_x\phi\cdot vh_\pm)    
		\mp w\partial(\nabla_x\phi\cdot \nabla_vh_\pm)\\
		\pm e^{\delta t}\partial\nabla_x\phi\cdot vw\mu^{1/2} - wL_\pm \partial h = we^{-\delta t}\partial(\Gamma_{\pm}(h,h)) + \delta wh_\pm.
	\end{multline*}
	Taking Fourier transform $\widehat{\cdot}$ over $\T^2_x$ and the inner product with $w\widehat{h_\pm}$ over $[-1,1]\times\R_v^3$ and the real part, we have 
	\begin{align}\notag
		&\quad\,\frac{1}{2}\partial_t|\widehat{w\partial h_\pm}|^2_{L^2_{x_1,v}}  +\frac{1}{2}\int_{\R^3}\big(v_1|\widehat{w\partial h_\pm}(1)|^2-v_1|\widehat{w\partial h_\pm}(-1)|^2\,\big)dv\\&\notag\qquad+\frac{qN}{(1+t)^{N+1}}\Big|\<v\>^{\frac{\vt}{2}}w\widehat{\partial h_\pm}\Big|^2_{L^2_{x_1,v}} \\&\le \notag
		\Re(\widehat{wL_\pm \partial h},\widehat{w\partial h_\pm})_{L^2_{x_1,v}}\mp \frac{1}{2}\Re(w(\partial(\nabla_x\phi\cdot vh_\pm))^\wedge,\widehat{w\partial h_\pm})_{L^2_{x_1,v}}\\&\notag\qquad\pm \Re(w(\partial(\nabla_x\phi\cdot \nabla_vh_\pm))^\wedge,\widehat{w\partial h_\pm})_{L^2_{x_1,v}}\mp \Re(e^{\delta t}\widehat{w\partial\nabla_x\phi}\cdot v\mu^{1/2},\widehat{w\partial h_\pm})_{L^2_{x_1,v}} \\&\qquad +\Re(e^{-\delta t}w(\partial\Gamma_{\pm}(h,h))^\wedge,\widehat{w\partial h_\pm})_{L^2_{x_1,v}} + \delta |wh_\pm|^2_{L^2_{x_1,v}}.\label{62a}
	\end{align}

	Following carefully the argument from \eqref{32a} to \eqref{45a}, by replacing Fourier transform on torus $\T^3$ by Fourier transform on $\bar{x}\in\T^2$, we have the following estimates. 
	\begin{Lem} \label{Lem33a}
		Assume $\gamma\ge -2$ in Landau case and $\gamma+2s\ge 1$, $\frac{1}{2}\le s<1$ in Boltzmann casae. Let $\partial=\partial^\alpha$ with $|\alpha|\le 1$. For any $T>0$, $\eta>0$, we have 
		\begin{multline}\label{84a}
			\int_{\Z^2}\Big|\int^T_0\frac{1}{2}\big(w^2(\partial(\nabla_x\phi\cdot vh_\pm))^\wedge,\widehat{\partial h_\pm}\big)_{L^2_{x_1,v}}\,dt\Big|^{1/2}d\Sigma(\bar{k})\\
			\le{C_\eta}{} \sum_{|\alpha|\le 1} \|e^{\delta t}\widehat{\partial^\alpha \nabla_x\phi}\|_{L^1_{\bar{k}}L^\infty_TL^2_{x_1}}\sum_{|\alpha|\le 1}\|e^{-\frac{\delta t}{2}}\<v\>^{1/2}w\widehat{\partial^\alpha h}\|_{L^1_{\bar{k}}L^2_TL^2_{x_1,v}}\\
			+{\eta}{}\sum_{|\alpha|\le 1}\|e^{-\frac{\delta t}{2}}\<v\>^{1/2}w\widehat{\partial^\alpha h}\|_{L^1_{\bar{k}}L^2_TL^2_{x_1,v}},
		\end{multline}and
		\begin{multline*}
			\notag\int_{\Z^2}\Big|\int^T_0(e^{\delta t}w^2\widehat{\partial\nabla_x\phi}\cdot v\mu^{1/2},\widehat{\partial h_\pm})_{L^2_{x_1,v}}dt\Big|^{1/2}d\Sigma({\bar{k}})\\
			\qquad\qquad\qquad\qquad\le C_\eta\|e^{\delta t}\widehat{\partial\nabla_x\phi}\|_{L^1_{\bar{k}}L^2_T}+\eta\|w\widehat{\partial h_\pm}\mu^{1/8}\|_{L^1_{\bar{k}}L^2_TL^2_{x_1,v}}.
		\end{multline*}
		For the boundedness of $\Gamma_\pm$, we have 
		%	\begin{align*}\notag
		%		&\quad\,\int_{\Z^2}\Big(\int^T_0\big|(w(\partial(\Gamma_{\pm}(f,g)))^\wedge,w\widehat{\partial h_\pm})_{L^2_{x_1,v}}\big|dt\Big)^{1/2}d\Sigma({\bar{k}}) \\&\le\notag C_\eta\sum_{|\alpha|\le 1} \|\mu^{1/16}\widehat{\partial^\alpha f}\|_{L^1_{\bar{k}}L^\infty_TL^2_{x_1,v}}\sum_{|\alpha|\le 1} \|\widehat{\partial^\alpha g}\|_{L^1_{\bar{k}}L^2_TL^2_{x_1}L^2_{D,w}}\\
		%		&\qquad+C_\eta\sum_{|\alpha|\le 1} \|\widehat{\partial^\alpha f}\|_{L^1_{\bar{k}}L^2_TL^2_{x_1}L^2_{D}}
		%		\sum_{|\alpha|\le 1} \|w\widehat{\partial^\alpha g}\|_{L^1_{\bar{k}}L^\infty_TL^2_{x_1,v}}\\
		%		&\qquad+\eta\sum_{|\alpha|\le 1} \big(\|\<v\>^{\frac{\gamma+4}{2}}w\widehat{\partial^\alpha h}\|_{L^1_{\bar{k}}L^2_TL^2_{x_1,v}}+\|\widehat{\partial^\alpha h}\|_{L^1_{\bar{k}}L^2_TL^2_{x_1}L^2_{D,w}}\big),
		%	\end{align*}when $-1\le\gamma+2\le 0$, and 
		\begin{multline*}
			\int_{\Z^2}\Big(\int^T_0\big|(
			w^2(\partial\Gamma_{\pm}(f,g))^\wedge,\widehat{h_\pm})_{L^2_{x_1,v}}\big|dt\Big)^{1/2}d\Sigma({\bar{k}})\\ 
			\le C_\eta\sum_{|\alpha|\le 1}\|w\widehat{\partial^\alpha f}\|_{L^1_{\bar{k}}L^\infty_TL^2_{x_1,v}}\sum_{|\alpha|\le 1}\|\widehat{\partial^\alpha g}\|_{L^1_{\bar{k}}L^2_TL^2_{x_1}L^2_{D,w}}+\eta\sum_{|\alpha|\le 1}\|\widehat{\partial^\alpha h}\|^2_{L^1_{\bar{k}}L^2_TL^2_{x_1}L^2_{D,w}}.
		\end{multline*}
		Moreover, for the term $\na_x\phi\cdot\na_vh_\pm$, we have 
		\begin{multline*}
			\notag\int_{\Z^2}\Big|\int^T_0\big((\partial(\nabla_x\phi\cdot w\nabla_vh_\pm))^\wedge,w\widehat{\partial h_\pm}\big)_{L^2_{x_1,v}}\,dt\Big|^{1/2}\,d\Sigma({\bar{k}})\\
			\notag\qquad\qquad\le {C_\eta}\sum_{|\alpha|\le 1}\|e^{\delta t}\widehat{\partial^\alpha \nabla_x\phi}\|_{L^1_{\bar{k}}L^\infty_TL^2_{x_1}}\sum_{|\alpha|\le 1}\|e^{-\frac{\delta t}{2}}\<v\>^{\gamma/2}(w\nabla_v{\partial^\alpha h})^\wedge\|_{L^1_{\bar{k}}L^2_TL^2_{x_1,v}}\\
			\notag\qquad\qquad\qquad\qquad\qquad\qquad\qquad\qquad+\eta\sum_{|\alpha|\le 1}\|e^{-\frac{\delta t}{2}}\<v\>^{-\gamma/2}w\widehat{\partial^\alpha h}\|_{L^1_{\bar{k}}L^2_TL^2_{x_1,v}},
		\end{multline*}
		for Landau case, and 
		\begin{align*}
			&\notag\quad\,\int_{\Z^3}\Big|\int^T_0\big((\nabla_x\phi\cdot \nabla_vh_\pm)^\wedge,\widehat{h_\pm}\big)_{L^2_v}\,dt\Big|^{1/2}\,d\Sigma(k)\\
			&\lesssim \|e^{\delta t}\widehat{E}\|_{L^1_kL^\infty_T}\|e^{-\frac{\delta t}{2}}h\|_{L^1_kL^2_TL^2_{D}}
			+\eta\big(\|e^{-\frac{\delta t}{2}}\widehat{h}\|_{L^1_kL^2_TL^2_{D}}+\|e^{-\frac{\delta t}{2}}\<v\>^{\frac{1}{2}}\widehat{h}\|_{L^1_kL^2_TL^2_v}\big). 
		\end{align*}for Boltzmann case. 
	\end{Lem}
	\begin{proof}
		Here we only prove \eqref{84a}, since the other estimates are similar to the argument from \eqref{32a} to \eqref{45a} by using the Banach algebra structure of $H^1_{x_1}$. 
		We also assume $\partial=\partial_{x_1}$, since the other cases are similar. 
		Noticing  
		\begin{align*}
			\partial(\nabla_x\phi\cdot vf_\pm) = \partial\nabla_x\phi\cdot vf_\pm + \nabla_x\phi\cdot v\partial f_\pm,
		\end{align*}	one has 
		\begin{align*}\notag
			&\quad\,\Big|\int^T_0\frac{1}{2}\big((w^2\partial{(\nabla_x\phi\cdot vf_\pm)})^\wedge,\widehat{\partial f_\pm}\big)_{L^2_{x_1,v}}\,dt\Big|\\
			&\lesssim\notag \int^T_0\int_{\R^3}\int^1_{-1}\int_{\Z^2}\Big(|\widehat{\partial\nabla_x\phi}({\bar{k}}-\bar{l})|\<v\>|\widehat{wf}(\bar{l})|+|\widehat{\nabla_x\phi}({\bar{k}}-\bar{l})|\<v\>|w\widehat{\partial f}(\bar{l})|\Big)\,d\Sigma(\bar{l})\\&\qquad\qquad\qquad\qquad\qquad\qquad\qquad\qquad\qquad\qquad\qquad\qquad\times|w\widehat{\partial f}({\bar{k}})|\,dx_1dvdt\\
%			&\lesssim\notag \int^T_0\int^1_{-1}\int_{\Z^2}\Big(|\widehat{\partial\nabla_x\phi}({\bar{k}}-\bar{l})||\<v\>^{1/2}\widehat{wf}(\bar{l})|_{L^2_v}+|\widehat{\nabla_x\phi}({\bar{k}}-\bar{l})||\<v\>^{1/2}w\widehat{\partial f}(\bar{l})|_{L^2_v}\Big)\,d\Sigma(\bar{l})\\&\qquad\qquad\qquad\qquad\qquad\qquad\qquad\qquad\qquad\qquad\qquad\qquad\times|\<v\>^{1/2}w\widehat{\partial f}({\bar{k}})|_{L^2_v}\,dx_1dt\\
			&\lesssim \int^T_0\int^1_{-1}{C_\eta}{}\Big(\int_{\Z^2}\Big(|e^{\delta t}\widehat{\partial\nabla_x\phi}({\bar{k}}-\bar{l})||e^{-\frac{\delta t}{2}}\<v\>^{1/2}\widehat{wf}(\bar{l})|_{L^2_v}\\&\qquad\qquad\qquad\qquad\qquad
			+|e^{\delta t}\widehat{\nabla_x\phi}({\bar{k}}-\bar{l})||e^{-\frac{\delta t}{2}}\<v\>^{1/2}w\widehat{\partial f}(\bar{l})|_{L^2_v}\Big)\,d\Sigma(\bar{l})\Big)^2\\
			&\qquad\qquad\qquad\qquad\qquad\qquad\qquad\qquad\qquad+{\eta^2}{}|e^{-\frac{\delta t}{2}}\<v\>^{1/2}w\widehat{\partial f}({\bar{k}})|^2_{L^2_v}\,dx_1dt.
		\end{align*}
		Taking the square root and summation over ${\bar{k}}\in\Z^2$, we have 
		\begin{align*}\notag
			&\quad\,\int_{\Z^2}\Big|\int^T_0\frac{1}{2}\big(w(\partial(\nabla_x\phi\cdot vf_\pm))^\wedge,w\widehat{\partial f_\pm}\big)_{L^2_{x_1,v}}\,dt\Big|^{1/2}d\Sigma({\bar{k}})\\
			&\notag\lesssim {C_\eta}{} \int_{\Z^2}\int_{\Z^2}\Big(\int^T_0\int^1_{-1}|e^{\delta t}\widehat{\partial\nabla_x\phi}({\bar{k}}-\bar{l})|^2|e^{-\frac{\delta t}{2}}\<v\>^{1/2}\widehat{wf}(\bar{l})|_{L^2_v}^2\,dx_1dt\Big)^{1/2}\,d\Sigma(\bar{l})d\Sigma({\bar{k}})\\
			&\quad+{C_\eta}{} \int_{\Z^2}\int_{\Z^2}\Big(\int^T_0\int^1_{-1}|e^{\delta t}\widehat{\nabla_x\phi}({\bar{k}}-\bar{l})|^2|e^{-\frac{\delta t}{2}}\<v\>^{1/2}w\widehat{\partial f}(\bar{l})|_{L^2_v}^2\,dx_1dt\Big)^{1/2}\,d\Sigma(\bar{l})d\Sigma({\bar{k}})\\
			&\quad+{\eta}{}\|e^{-\frac{\delta t}{2}}\<v\>^{1/2}w\widehat{\partial f}({\bar{k}})\|_{L^1_{\bar{k}}L^2_TL^2_{x_1,v}}\\
			&\notag\lesssim C_\eta \int_{\Z^2}\int_{\Z^2}\sup_{0\le t\le T}|e^{\delta t}\widehat{\partial\nabla_x\phi}({\bar{k}}-\bar{l})|\Big(\int^T_0|e^{-\frac{\delta t}{2}}\<v\>^{1/2}\widehat{wf}(\bar{l})|_{L^2_v}^2\,dt\Big)^{1/2}\,d\Sigma(\bar{l})d\Sigma({\bar{k}})\\
			&\notag\quad+{C_\eta}{} \int_{\Z^2}\int_{\Z^2}\sup_{0\le t\le T}|e^{\delta t}\widehat{\nabla_x\phi}({\bar{k}}-\bar{l})|\Big(\int^T_0|e^{-\frac{\delta t}{2}}\<v\>^{1/2}w\widehat{\partial f}(\bar{l})|_{L^2_v}^2\,dt\Big)^{1/2}\,d\Sigma(\bar{l})d\Sigma({\bar{k}})\\
			&\quad+{\eta}{}\|e^{-\frac{\delta t}{2}}\<v\>^{1/2}w\widehat{\partial f}({\bar{k}})\|_{L^1_{\bar{k}}L^2_TL^2_{x_1,v}}\\
			&\lesssim{C_\eta}{} \|e^{\delta t}\widehat{\partial\nabla_x\phi}\|_{L^1_{\bar{k}}L^\infty_TL^2_{x_1}}\|e^{-\frac{\delta t}{2}}\<v\>^{1/2}\widehat{wf}\|_{L^1_{\bar{k}}L^2_TL^2_{x_1,v}}\\
			&\quad+\|e^{\delta t}\widehat{\nabla_x\phi}\|_{L^1_{\bar{k}}L^\infty_TL^2_{x_1}}\|e^{-\frac{\delta t}{2}}\<v\>^{1/2}w\widehat{\partial f}\|_{L^1_{\bar{k}}L^2_TL^2_{x_1,v}}+{\eta}{}\|e^{-\frac{\delta t}{2}}\<v\>^{1/2}w\widehat{\partial f}\|_{L^1_{\bar{k}}L^2_TL^2_{x_1,v}}.
		\end{align*}Here we have used Young's inequality and Fubini's Theorem. 
		\qe\end{proof}

	We are ready to prove the result in finite channel. 
	\begin{proof}[Proof of Theorem \ref{Channel1}]
		%	\medskip \noindent{\bf Global existence.}
		Define $h=e^{\delta t}f$ and let $\partial=\partial^\alpha$ with $|\alpha|\le 1$. Applying $\partial$ on \eqref{1} and multiplying $we^{\delta t}$, we have 
		\begin{multline*}\notag\partial_tw\pa h_\pm +\frac{qN\<v\>^\vt}{(1+t)^{N+1}}w\partial h_\pm + v\cdot\nabla_xw(\partial h_\pm) \pm \frac{1}{2}w\partial(\nabla_x\phi\cdot vh_\pm)    
			\mp w\partial(\nabla_x\phi\cdot \nabla_vh_\pm)\\
			\pm e^{\delta t}\partial\nabla_x\phi\cdot vw\mu^{1/2} - wL_\pm \partial h = we^{-\delta t}\partial(\Gamma_{\pm}(h,h)) + \delta wh_\pm.
		\end{multline*}
		With time weight $e^{\delta t}$, the only extra term is $\delta wh_\pm$. Then following the argument we used in Theorem \ref{Thm41f}, we have 
		\begin{align}\notag\label{7.1}
			&\quad\,\sum_{|\alpha|\le 1} \|e^{\delta t}\partial^\alpha(\widehat{a_+},\widehat{a_-},\widehat{b},\widehat{c})\|_{L^1_{\bar{k}}L^2_TL^2_{x_1}}+\sum_{|\alpha|\le 1}\|e^{\delta t}\widehat{\partial^\alpha E}\|_{L^1_{\bar{k}}L^2_TL^2_{x_1}}\\
			\notag&\lesssim \sum_{|\alpha|\le 1}\big(\|\widehat{\partial^\alpha h}(T)\|_{L^1_{\bar{k}}L^2_{x_1,v}}+\|\widehat{\partial^\alpha f_0}\|_{L^1_{\bar{k}}L^2_{x_1,v}}\big)\\
			&\notag\quad+\sum_{|\alpha|\le 1}\big(\|e^{\delta t}\widehat{\partial^\alpha  E}\|_{L^1_{\bar{k}}L^\infty_TL^2_{x_1}}+\|\widehat{\partial^\alpha h}\|_{L^1_{\bar{k}}L^\infty_TL^2_{x_1}L^2_v}\big)\sum_{|\alpha|\le 1}\|\widehat{\partial^\alpha h}\|_{L^1_{\bar{k}}L^2_TL^2_{x_1}L^2_{D}}\\
			&\quad+\sum_{|\alpha|\le1}\|\{\I-\P\}\widehat{\partial^\alpha h}\|_{L^1_{\bar{k}}L^2_TL^2_{x_1}L^2_D} + \delta^{1/2}\|\partial^\alpha h\|_{L^1_{\bar{k}}L^2_TL^2_{x_1}L^2_{v}}.
		\end{align}
		Note from \eqref{103a} and \eqref{115} that 
		\begin{align}\label{146}
			|\partial f_\pm(-1,\bar{k},-v_1,\bar{v})|=|\partial f_\pm(-1,\bar{k},v_1,\bar{v})|,\quad |\partial f_\pm(1,\bar{k},-v_1,\bar{v})|=|\partial f_\pm(1,\bar{k},v_1,\bar{v})|,
		\end{align}on $v_1\neq 0$. 
		Thus, by change of variable $v_1\mapsto -v_1$, 
		\begin{align}\label{135}
			\int_{\R^3}\big(v_1|\widehat{w\partial f_\pm}(1)|^2-v_1|\widehat{w\partial f_\pm}(-1)|^2\,\big)dv = 0. 
		\end{align}
		%Applying $\partial$ to \eqref{1} and taking Fourier transform $\widehat{\cdot}$ with respect to $\bar{x}=(x_2,x_3)$, one has 
		%\begin{multline}\label{144a}\partial_t\widehat{\partial f_\pm} + v_1\partial_{x_1}\widehat{\partial f_\pm}+i\bar{v}\cdot\bar{k}\widehat{\partial f_\pm}\pm \frac{1}{2}(\partial(\nabla_x\phi\cdot vf_\pm))^\wedge  \mp (\partial(\nabla_x\phi\cdot \nabla_vf_\pm))^\wedge  
		%	\pm \widehat{\partial\nabla_x\phi}\cdot v\mu^{1/2}  \\= (wL_\pm \partial f)^\wedge+(\partial(\Gamma_{\pm}(f,f)))^\wedge.
		%\end{multline}
		%Taking inner product with $\widehat{\partial f_\pm}$ over $[-1,1]\times \R^3_v$, we have 
		Then \eqref{62a} becomes 
		\begin{multline}
			\quad\,\frac{1}{2}\partial_t|\widehat{w\partial h_\pm}|^2_{L^2_{x_1,v}}  +\frac{qN}{(1+t)^{N+1}}\Big|\<v\>^{\frac{\vt}{2}}w\widehat{\partial h_\pm}\Big|^2_{L^2_{x_1,v}} \le 
			\Re(\widehat{wL_\pm \partial h},\widehat{w\partial h_\pm})_{L^2_{x_1,v}}\\
			\mp \frac{1}{2}\Re(w(\partial(\nabla_x\phi\cdot vh_\pm))^\wedge,\widehat{w\partial h_\pm})_{L^2_{x_1,v}}\pm \Re(w(\partial(\nabla_x\phi\cdot \nabla_vh_\pm))^\wedge,\widehat{w\partial h_\pm})_{L^2_{x_1,v}}
			\\\qquad\mp \Re(e^{\delta t}\widehat{w\partial\nabla_x\phi}\cdot v\mu^{1/2},\widehat{w\partial h_\pm})_{L^2_{x_1,v}} +\Re(e^{-\delta t}w(\partial\Gamma_{\pm}(h,h))^\wedge,\widehat{w\partial h_\pm})_{L^2_{x_1,v}}.\label{134}
		\end{multline}
		%\begin{align}\notag
		%	&\quad\,\frac{1}{2}\|\widehat{\partial f_\pm}\|^2_{L^2_{x_1,v}} +\int_{\R^3}\big(v_1|\widehat{w\partial f_\pm}(1)|^2-v_1|\widehat{w\partial f_\pm}(-1)|^2\,\big)dv+\lambda\|(\II-\PP)\widehat{\partial f}\|^2_{L^2_{x_1}L^2_{D}}\\
		%	&\notag\le \mp \frac{1}{2}\Re((\partial(\nabla_x\phi\cdot vf_\pm))^\wedge,\widehat{\partial f_\pm})_{L^2_{x_1,v}}\pm \Re((\partial(\nabla_x\phi\cdot \nabla_vf_\pm))^\wedge,\widehat{\partial f_\pm})_{L^2_{x_1,v}} \\&\qquad\mp \Re(\widehat{\partial\nabla_x\phi}\cdot v\mu^{1/2},\widehat{\partial f_\pm})_{L^2_{x_1,v}} +\Re((\partial\Gamma_{\pm}(f,f))^\wedge,\widehat{\partial f_\pm})_{L^2_{x_1,v}}.
		%\end{align}
		If $q\neq 0$, the second to fifth terms on the right hand of \eqref{134} can be estimated by using Lemma \ref{Lem33a}. 
		If $q=0$, we further look at the forth term of the right hand of \eqref{134}:
		\begin{align*}
			&\quad\,\sum_\pm \pm\Re\big(\widehat{\partial\nabla_x\phi}\cdot v\mu^{1/2},\widehat{\partial f_\pm}\big)_{L^2_{x_1,v}}= \Re(\widehat{\partial\nabla_x\phi},\widehat{\partial G})_{L^2_{x_1}}\\
			&= \Re\big(\widehat{\partial\phi}(1)\,|\,\widehat{\partial G}(1)\big)-\Re\big(\widehat{\partial\phi}(-1)\,|\,\widehat{\partial G}(-1)\big)-\Re\big(\widehat{\partial\phi},\widehat{\partial\nabla_x G}\big)_{L^2_{x_1}},
		\end{align*}
		where $G$ is defined in \eqref{93a}. Note from \eqref{Neumann_finite}, \eqref{103a} and \eqref{115} we have the following. $\widehat{\partial\phi}(\pm1)=0$ when $\partial=\partial_{x_1}$. For the case $\partial=I,\partial_{x_2},\partial_{x_3}$, we know that 
		$v_1\mu^{1/2}f_\pm(v_1)$ is odd with respect to $v_1$ and hence, $\partial G(\pm 1)=0$. In any cases, we have $\big(\widehat{\partial\phi}(1)\,|\,\widehat{\partial G}(1)\big)=\big(\widehat{\partial\phi}(-1)\,|\,\widehat{\partial G}(-1)\big)=0$. 
		Then multiplying $e^{2\delta t}$, we have from \eqref{98a}$_1$ and \eqref{2} that 
		\begin{align*}
			&\quad\,e^{2\delta t}\sum_\pm \pm\Re(\widehat{\partial_{x_1}\nabla_x\phi}\cdot v\mu^{1/2},\widehat{\partial f_\pm})_{L^2_{x_1,v}}\\
			&= -e^{2\delta t}\Re\big(\widehat{\partial\phi},{(\partial\partial_t(a_+-a_-))^\wedge}\big)_{L^2_{x_1}}
			= -e^{2\delta t}\Re\big(\widehat{\partial\phi},-{(\partial\partial_t(\Delta_x\phi))^\wedge}\big)_{L^2_{x_1}}\\
			&=
%			e^{2\delta t}\Big(\Re\big(\widehat{\partial\phi}(1)\,|\,{(\partial\partial_t\partial_{x_1}\phi(1))^\wedge}\big)-\Re\big(\widehat{\partial\phi}(-1)\,|\,{(\partial\partial_t\partial_{x_1}\phi(-1))^\wedge}\big)
		-\Re\big(\widehat{\partial\nabla_x\phi},{(\partial\partial_t\nabla_x\phi)^\wedge}\big)_{L^2_{x_1}}
%		\Big)
			= \frac{1}{2}\partial_t\|e^{\delta t}\partial\nabla_x\phi\|^2_{L^2_{x_1}}-\delta\|e^{\delta t}\partial\nabla_x\phi\|^2_{L^2_{x_1}}. 
		\end{align*}
		Together with Lemma \ref{Lem33a}, 
		taking summation on $\pm$ of \eqref{134}, integration on $t\in[0,T]$, absolute value, square root and summation over $\bar{k}\in\Z^2$ and $|\alpha|\le 1$, we have that when $q=0$, 
		\begin{align}\notag\label{133}
			&\quad\,\sum_{|\alpha|\le 1}\|\widehat{\partial^\alpha h}\|_{L^1_{\bar{k}}L^\infty_TL^2_{x_1,v}} +\sum_{|\alpha|\le 1}\|\{\I-\P\}\widehat{\partial^\alpha h}\|_{L^1_{\bar{k}}L^2_TL^2_{x_1}L^2_{D}}+\sum_{|\alpha|\le 1}\|e^{\delta t}\wh{\partial^\alpha E}\|_{L^1_{\bar{k}}L^\infty_TL^2_{x_1}}\\
			&\notag\lesssim \sum_{|\alpha|\le 1}\|\widehat{\partial^\alpha f_0}\|_{L^1_{\bar{k}}L^\infty_TL^2_{x_1,v}}+\sum_{|\alpha|\le 1}\|\widehat{\partial^\alpha E_0}\|_{L^1_{\bar{k}}L^2_{x_1}}+\eta\sum_{|\alpha|\le 1}\|e^{-\frac{\delta t}{2}}\<v\>^{-\gamma/2}\widehat{\partial^\alpha h}\|_{L^1_{\bar{k}}L^2_TL^2_{x_1,v}}\\
			&\notag\qquad+\Big(C_\eta \E_T+\eta\Big)\Big(\sum_{|\alpha|\le 1}\|e^{-\frac{\delta t}{2}}\<v\>^{1/2}\widehat{\partial^\alpha h}\|_{L^1_{\bar{k}}L^2_TL^2_{x_1,v}}+\D_T\Big) \\
%			&\notag\qquad+ \eta
%%			\Big(\sum_{|\alpha|\le 1}\|e^{-\frac{\delta t}{2}}\<v\>^{1/2}\widehat{\partial^\alpha h}\|_{L^1_{\bar{k}}L^2_TL^2_{x_1,v}}+
%			\sum_{|\alpha|\le 1}\|e^{-\frac{\delta t}{2}}\<v\>^{-\gamma/2}\widehat{\partial^\alpha h}\|_{L^1_{\bar{k}}L^2_TL^2_{x_1,v}}\\
%			+\D_T\Big)\\
			&\qquad+\delta^{1/2} \sum_{|\alpha|\le 1}\big(\|\partial^\alpha h_\pm\|_{L^1_{\bar{k}}L^2_TL^2_{x_1,v}}+\|e^{\delta t}\partial^\alpha\nabla_x\phi\|^2_{L^1_{\bar{k}}L^2_TL^2_{x_1}}\big),
			%	&\notag\qquad+{C_\eta} \sum_{|\alpha|\le 1} \|\widehat{\partial^\alpha \nabla_x\phi}\|_{L^1_{\bar{k}}L^\infty_TL^2_{x_1}}\sum_{|\alpha|\le 1}\Big(\|\<v\>^{1/2}\widehat{\partial^\alpha f}\|_{L^1_{\bar{k}}L^2_TL^2_{x_1,v}}+\|\<v\>^{\gamma/2}(\nabla_v{\partial^\alpha f})^\wedge\|_{L^1_{\bar{k}}L^2_TL^2_{x_1,v}}\Big)\\
			%	&\notag\qquad+{\eta}{}\sum_{|\alpha|\le 1}\|\<v\>^{1/2}\widehat{\partial^\alpha f}\|_{L^1_{\bar{k}}L^2_TL^2_{x_1,v}}+\eta\sum_{|\alpha|\le 1}\|\<v\>^{-\gamma/2}\widehat{\partial^\alpha f}\|_{L^1_{\bar{k}}L^2_TL^2_{x_1,v}}
			%	 \\&\qquad +C_\eta\sum_{|\alpha|\le 1}\|\widehat{\partial^\alpha f}\|_{L^1_{\bar{k}}L^\infty_TL^2_{x_1,v}}\sum_{|\alpha|\le 1}\|\widehat{\partial^\alpha f}\|_{L^1_{\bar{k}}L^2_TL^2_{x_1}L^2_{D}}+\eta\sum_{|\alpha|\le 1}\|\widehat{\partial^\alpha f}\|^2_{L^1_{\bar{k}}L^2_TL^2_{x_1}L^2_{D}}.
		\end{align}
		where $\E_T$ and $\D_T$ are given by \eqref{141} and \eqref{142} respectively. 
		Similarly, when $q\neq 0$, we use \eqref{36c} to find that 
		\begin{align}\label{7.6}\notag
			&\sum_{|\alpha|\le 1}\Big(\|w\widehat{\partial^\alpha h}\|_{L^1_{\bar{k}}L^\infty_TL^2_{x_1,v}}+\sqrt{qN}\|{\<v\>^{\frac{\vt}{2}}}{(1+t)^{-\frac{N+1}{2}}}\widehat{\partial^\alpha{h}}\|_{L^1_kL^2_TL^2_v} +\|\widehat{\partial^\alpha h}\|_{L^1_{\bar{k}}L^2_TL^2_{x_1}L^2_{D,w}}\Big)\\
			&\notag\lesssim \sum_{|\alpha|\le 1}\Big(\|w\widehat{\partial^\alpha f_0}\|_{L^1_{\bar{k}}L^\infty_TL^2_{x_1,v}}+\|\widehat{\partial^\alpha E_0}\|_{L^1_{\bar{k}}L^2_{x_1}}\Big)+\eta\sum_{|\alpha|\le 1}\|e^{-\frac{\delta t}{2}}w\<v\>^{-\gamma/2}\widehat{\partial^\alpha h}\|_{L^1_{\bar{k}}L^2_TL^2_{x_1,v}}\\&\quad\notag
			+\Big(C_\eta \E_{T,w}+\eta\Big)\Big(\sum_{|\alpha|\le 1}\|e^{-\frac{\delta t}{2}}w\<v\>^{1/2}\widehat{\partial^\alpha h}\|_{L^1_{\bar{k}}L^2_TL^2_{x_1,v}}+\D_{T,w}\Big) \\
%			&\notag\quad+ \eta
%%			\Big(\sum_{|\alpha|\le 1}\|e^{-\frac{\delta t}{2}}w\<v\>^{1/2}\widehat{\partial^\alpha h}\|_{L^1_{\bar{k}}L^2_TL^2_{x_1,v}}+
%\sum_{|\alpha|\le 1}\|e^{-\frac{\delta t}{2}}w\<v\>^{-\gamma/2}\widehat{\partial^\alpha h}\|_{L^1_{\bar{k}}L^2_TL^2_{x_1,v}}\\
%%+\D_{T,w}\Big)\\
			&\notag\quad+\delta^{1/2}\sum_{|\alpha|\le 1}\big(\|\partial^\alpha h_\pm\|_{L^1_{\bar{k}}L^2_TL^2_{x_1,v}}+\|e^{\delta t}\partial^\alpha\nabla_x\phi\|^2_{L^1_{\bar{k}}L^2_TL^2_{x_1}}\big)\\
			&\quad+\sum_{|\alpha|\le 1}\Big(\|\widehat{\partial^\alpha h}\|_{L^1_{\bar{k}}L^2_TL^2_{x_1}L^2_{D}} + \|e^{\delta t}\widehat{\partial^\alpha E}\|_{L^1_{\bar{k}}L^2_TL^2_{x_1}}\Big),
		\end{align}
		where $\E_{T,w}$ and $\D_{T,w}$ are given by \eqref{143} and \eqref{144}. Notice that $\|\partial^\alpha h_\pm\|_{L^1_{\bar{k}}L^2_TL^2_{x_1,v}}\lesssim \|\partial^\alpha h_\pm\|_{L^1_{\bar{k}}L^2_TL^2_{x_1}L^2_D}$. Then taking linear combination $\eqref{133}+\kappa\times\eqref{7.1} + \kappa^2\times\eqref{7.6}$ and letting $\delta,\eta,\kappa>0$ suitably small, we have 
		\begin{multline}\label{7.7}
			\E_T+\D_T+\E_{T,w}+\D_{T,w} \lesssim \sum_{|\alpha|\le 1}\Big(\|w\widehat{\partial^\alpha f_0}\|_{L^1_{\bar{k}}L^\infty_TL^2_{x_1,v}}+\|\widehat{\partial^\alpha E_0}\|_{L^1_{\bar{k}}L^2_{x_1}}\Big)\\
			+\Big(C_\eta (\E_T+\E_{T,w})+\eta\Big)\Big(\sum_{|\alpha|\le 1}\|e^{-\frac{\delta t}{2}}w\<v\>^{1/2}\widehat{\partial^\alpha h}\|_{L^1_{\bar{k}}L^2_TL^2_{x_1,v}}+\D_T+\D_{T,w}\Big) \\
			+ \eta
%			\Big(\sum_{|\alpha|\le 1}\|e^{-\frac{\delta t}{2}}w\<v\>^{1/2}\widehat{\partial^\alpha h}\|_{L^1_{\bar{k}}L^2_TL^2_{x_1,v}}+
			\sum_{|\alpha|\le 1}\|e^{-\frac{\delta t}{2}}w\<v\>^{-\gamma/2}\widehat{\partial^\alpha h}\|_{L^1_{\bar{k}}L^2_TL^2_{x_1,v}}.
%			+\D_T+\D_{T,w}\Big),
		\end{multline}
		Then we discuss the {\it a priori} estimates in two cases. 
		
		\medskip \noindent{\bf Case I: $\gamma\ge -1$ for VPL case and $\gamma+2s\ge 1$, $1/2\le s<1$ for VPB case.} 
		In this case, we have 
		\begin{align*}
			\|e^{-\frac{\delta t}{2}}\<v\>^{\gamma/2}\widehat{\partial^\alpha f}\|_{L^1_{\bar{k}}L^2_TL^2_{x_1,v}}
			\le \|e^{-\frac{\delta t}{2}}\<v\>^{1/2}\widehat{\partial^\alpha f}\|_{L^1_{\bar{k}}L^2_TL^2_{x_1,v}}
			\lesssim \D_T. 
		\end{align*} 
		Letting $q=0$ and $\eta>0$ small enough in \eqref{7.7}, we have
		\begin{align*}
			\E_T + \D_T &\lesssim \sum_{|\alpha|\le 1}\big(\|\widehat{\partial f_0}\|_{L^1_{\bar{k}}L^\infty_TL^2_{x_1,v}}+\|\widehat{\partial E_0}\|^2_{L^1_{\bar{k}}L^2_{x_1}}\big) + \E_T\D_T.
		\end{align*}
		This concludes the proof when $\gamma\ge -1$ in Landau case and $\gamma+2s\ge 1$, $\frac{1}{2}\le s<1$ in Boltzmann case by using the standard continuity argument under the smallness of \eqref{124}. 

		\medskip \noindent{\bf Case II: $-2\le\gamma<-1$ for VPL case.}
		%Next we assume $-2\le\gamma<-1$ in Landau case. Then 
		In this case, recall that we choose $\vt=-\gamma$. Then we have 
		\begin{multline*}
			\|e^{-\frac{\delta t}{2}}w\<v\>^{1/2}\widehat{\partial^\alpha h}\|_{L^1_{\bar{k}}L^2_TL^2_{x_1,v}}
			\lesssim\|e^{-\frac{\delta t}{2}}w\<v\>^{\gamma/2}\widehat{\partial^\alpha h}\|_{L^1_{\bar{k}}L^2_TL^2_{x_1,v}}\\
			\lesssim \|{\<v\>^{\frac{\vt}{2}}}{(1+t)^{-\frac{N+1}{2}}}\widehat{\partial^\alpha{h}}\|_{L^1_kL^2_TL^2_v} \le \D_{T,w}.
		\end{multline*}
		Letting $\eta>0$ in \eqref{7.7} small enough, we have 
		\begin{multline}\label{147}
			\E_{T,w}+\D_{T,w}+\E_T + \D_T \lesssim \sum_{|\alpha|\le 1}\big(\|\widehat{w\partial^\alpha f_0}\|^2_{L^1_{\bar{k}}L^2_{x_1,v}}+\|\widehat{\partial^\alpha E_0}\|^2_{L^1_{\bar{k}}L^2_{x_1}}\big)\\+ (\E_T+ \E_{T,w})(\D_T +\D_{T,w}).
		\end{multline}
		Once \eqref{147} is obtained, then \eqref{13} follows from the standard continuity argument and local existence from Section \ref{Sec_loc} under the smallness of $\sum_{|\alpha|\le 1}\big(\|\widehat{w\partial^\alpha f_0}\|^2_{L^1_{\bar{k}}L^2_{x_1,v}}+\|\widehat{\partial^\alpha E_0}\|^2_{L^1_{\bar{k}}L^2_{x_1}}\big)$; cf. \cite{Guo2002a}. 
		This concludes the proof of the global existence and large-time behavior of mild solutions. 

		The uniqueness of the initial boundary value problem \eqref{1} can be proved by applying the similar method as the previous energy estimates and is now quite standard. Also the local solution we extend from section \ref{localsolution} is unique. So we omit these analogous details. 
%		For the positivity, noticing that on the boundary \begin{align*}
%			F(t,\pm1,\bar{x},v) = \mu+\mu^{1/2}f_\pm(t,\pm1,\bar{x},v)\ge 0,
%		\end{align*}
		The positivity of the solutions to VPL systems can be obtained from \cite[Lemma 12, page 800]{Guo2012}. The positivity of solutions to VPB systems can be guaranteed by \cite{Guo2002}. This completes the proof of Theorem \ref{Channel1}. 
		\qe\end{proof}

	\begin{proof}[Proof of Theorem \ref{Channel2}]
		We will show that the regularity of the initial data can propagate along time. Let $m\ge 0$. Since $\bar{k}$ doesn't concern the boundary $x_1\in[-1,1]$, following closely the argument proving Theorem \ref{Thm41f}, Lemma \ref{Lem33a}, we have the following estimates:
		\begin{align*}\notag
			&\quad\,\sum_{|\alpha|\le 1} \|e^{\delta t} \partial^\alpha(\widehat{a_+},\widehat{a_-},\widehat{b},\widehat{c})\|_{L^1_{\bar{k},m}L^2_TL^2_{x_1}}+\sum_{|\alpha|\le 1}\|e^{\delta t}\widehat{\partial^\alpha E}\|_{L^1_{\bar{k},m}L^2_TL^2_{x_1}}\\
			\notag&\lesssim \sum_{|\alpha|\le 1}\big(\|\widehat{\partial^\alpha h}(T)\|_{L^1_{\bar{k},m}L^2_{x_1,v}}+\|\widehat{\partial^\alpha f_0}\|_{L^1_{\bar{k},m}L^2_{x_1,v}}\big)+\sum_{|\alpha|\le1}\|\{\I-\P\}\widehat{\partial^\alpha h}\|_{L^1_{\bar{k},m}L^2_TL^2_{x_1}L^2_D}\\
			&\qquad+\sum_{|\alpha|\le 1}\big(\|e^{\delta t}\widehat{\partial^\alpha  E}\|_{L^1_{\bar{k},m}L^\infty_TL^2_{x_1}}+\|\widehat{\partial^\alpha h}\|_{L^1_{\bar{k},m}L^\infty_TL^2_{x_1}L^2_v}\big)\sum_{|\alpha|\le 1}\|\widehat{\partial^\alpha h}\|_{L^1_{\bar{k},m}L^2_TL^2_{x_1}L^2_{D}},
		\end{align*}and
		\begin{multline*}\notag
			\quad\,\int_{\Z^2}\Big|\int^T_0\frac{1}{2}\big((w\partial(\nabla_x\phi\cdot vh_\pm))^\wedge,\<\bar{k}\>^{2m}\widehat{w\partial h_\pm}\big)_{L^2_{x_1,v}}\,dt\Big|^{1/2}d\Sigma(\bar{k})\\
			\lesssim{C_\eta}{} \sum_{|\alpha|\le 1} \|e^{\delta t}\widehat{\partial^\alpha \nabla_x\phi}\|_{L^1_{\bar{k},m}L^\infty_TL^2_{x_1}}\sum_{|\alpha|\le 1}\|e^{-\frac{\delta t}{2}}\<v\>^{1/2}\widehat{w\partial^\alpha h}\|_{L^1_{\bar{k},m}L^2_TL^2_{x_1,v}}\\\qquad\qquad\qquad\qquad\qquad+{\eta}{}\sum_{|\alpha|\le 1}\|e^{-\frac{\delta t}{2}}\<v\>^{1/2}\widehat{w\partial^\alpha h}\|_{L^1_{\bar{k},m}L^2_TL^2_{x_1,v}},
		\end{multline*}
	\begin{multline*}
			\notag\quad\,\int_{\Z^2}\Big|\int^T_0\big((\partial(\nabla_x\phi\cdot w\nabla_vh_\pm))^\wedge,\<\bar{k}\>^{2m}\widehat{w\partial h_\pm}\big)_{L^2_{x_1,v}}\,dt\Big|^{1/2}\,d\Sigma({\bar{k}})\\
			\notag\lesssim {C_\eta}\sum_{|\alpha|\le 1}\|e^{\delta t}\widehat{\partial^\alpha \nabla_x\phi}\|_{L^1_{\bar{k},m}L^\infty_TL^2_{x_1}}\sum_{|\alpha|\le 1}\|e^{-\frac{\delta t}{2}}\<v\>^{\gamma/2}(w\nabla_v{\partial^\alpha h})^\wedge\|_{L^1_{\bar{k},m}L^2_TL^2_{x_1,v}}\\\qquad\qquad\qquad\qquad\qquad+\eta\sum_{|\alpha|\le 1}\|e^{-\frac{\delta t}{2}}\<v\>^{-\gamma/2}\widehat{w\partial^\alpha h}\|_{L^1_{\bar{k},m}L^2_TL^2_{x_1,v}},
		\end{multline*}
	\begin{multline*}
			\notag\quad\,\int_{\Z^2}\Big|\int^T_0(\widehat{w\partial\nabla_x\phi}\cdot v\mu^{1/2},\<\bar{k}\>^{2m}\widehat{w\partial h_\pm})_{L^2_{x_1,v}}dt\Big|^{1/2}d\Sigma({\bar{k}})\\\qquad\qquad\qquad\qquad\le C_\eta\|\widehat{\partial\nabla_x\phi}\|_{L^1_{\bar{k},m}L^2_T}+\eta\|\widehat{w\partial h_\pm}\mu^{1/8}\|_{L^1_{\bar{k},m}L^2_TL^2_{x_1,v}}.
		\end{multline*}
		For the boundedness of $\Gamma_\pm$, we have 
		%	\begin{align*}\notag
		%		&\quad\,\int_{\Z^2}\Big(\int^T_0\big|((w\partial(\Gamma_{\pm}(f,g)))^\wedge,\<\bar{k}\>^{2m}\widehat{w\partial h_\pm})_{L^2_{x_1,v}}\big|dt\Big)^{1/2}d\Sigma({\bar{k}}) \\&\le\notag C_\eta\sum_{|\alpha|\le 1} \|\mu^{1/16}\widehat{\partial^\alpha f}\|_{L^1_{\bar{k},m}L^\infty_TL^2_{x_1,v}}\sum_{|\alpha|\le 1} \|\widehat{\partial^\alpha g}\|_{L^1_{\bar{k},m}L^2_TL^2_{x_1}L^2_{D,w}}\\
		%		&\qquad+C_\eta\sum_{|\alpha|\le 1} \|\widehat{\partial^\alpha f}\|_{L^1_{\bar{k},m}L^2_TL^2_{x_1}L^2_{D}}
		%		\sum_{|\alpha|\le 1} \|(w\partial^\alpha g)^\wedge\|_{L^1_{\bar{k},m}L^\infty_TL^2_{x_1,v}}\\
		%		&\qquad+\eta\sum_{|\alpha|\le 1} \big(\|\<v\>^{\frac{\gamma+4}{2}}(w\partial^\alpha h)^\wedge\|_{L^1_{\bar{k},m}L^2_TL^2_{x_1,v}}+\|\widehat{\partial^\alpha h}\|_{L^1_{\bar{k},m}L^2_TL^2_{x_1}L^2_{D,w}}\big),
		%	\end{align*}when $-1\le\gamma+2\le 0$, and 
		\begin{multline*}\notag
			\int_{\Z^2}\Big(\int^T_0\big|(
			w^2\partial(\Gamma_{\pm}(f,g))^\wedge,\<\bar{k}\>^{2m}\widehat{h_\pm})_{L^2_{x_1,v}}\big|dt\Big)^{1/2}d\Sigma({\bar{k}})\\ \le\notag C_\eta\sum_{|\alpha|\le 1}\|w\widehat{\partial^\alpha f}\|_{L^1_{\bar{k},m}L^\infty_TL^2_{x_1,v}}\sum_{|\alpha|\le 1}\|\widehat{\partial^\alpha f}\|_{L^1_{\bar{k},m}L^2_TL^2_{x_1}L^2_{D,w}}+\eta\sum_{|\alpha|\le 1}\|\widehat{\partial^\alpha f}\|^2_{L^1_{\bar{k},m}L^2_TL^2_{x_1}L^2_{D,w}}.
		\end{multline*}
		%when $\gamma\ge -3$. 
		
		Consequently, noticing $\<\bar{k}\>^{m}$ doesn't affect the boundary terms, one can apply similar calculation for deriving \eqref{147} to obtain that 
		\begin{align*}
			&\sum_{|\alpha|\le 1}\Big(\|\widehat{\partial^\alpha h}\|_{L^1_{\bar{k},m}L^\infty_TL^2_{x_1,v}} +\|\wh{\partial^\alpha E}\|_{L^1_{\bar{k},m}L^\infty_TL^2_{x_1}}+\|\widehat{\partial^\alpha h}\|_{L^1_{\bar{k},m}L^2_TL^2_{x_1}L^2_{D}}+\|\widehat{\partial^\alpha E}\|_{L^1_{\bar{k},m}L^2_TL^2_{x_1}}\\
			&\ +\|w\widehat{\partial^\alpha h}\|_{L^1_{\bar{k},m}L^\infty_TL^2_{x_1,v}}+\sqrt{qN}\|{\<v\>^{\frac{\vt}{2}}}{(1+t)^{-\frac{N+1}{2}}}\widehat{\partial^\alpha{h}}\|_{L^1_kL^2_TL^2_v} +\|\widehat{\partial^\alpha h}\|_{L^1_{\bar{k},m}L^2_TL^2_{x_1}L^2_{D,w}}\Big)\\
			&\lesssim \sum_{|\alpha|\le 1}\big(\|\widehat{w\partial^\alpha f_0}\|^2_{L^1_{\bar{k},m}L^\infty_TL^2_{x_1,v}}+\|\widehat{\partial^\alpha E_0}\|^2_{L^1_{\bar{k},m}L^2_{x_1}}\big),
		\end{align*}provided that $\varepsilon_0>0$ in \eqref{1.23} is suitably small. 
		This completes the proof of Theorem \ref{Channel2} for VPL systems when $-2\le\gamma<-1$. The proof of other cases follows similarly and we omit the details for brevity.  
		\qe\end{proof}

	\section{Local existence}\label{Sec_loc}
	
	In this section, we are concerned with the local-in-time existence of solutions to problem \eqref{1}. 
	For brevity of presentation, we only give the proof for Vlasov-Poisson-Landau equation with the specular reflection boundary condition in the finite channel, since the other cases are similar. 
	
	\begin{Thm}\label{localsolution}
		Let $\gamma\ge -2$ and $w$ be given by \eqref{w2}. Then there exists $\ve_0>0$, $T_0>0$ such that if $F_0(x_1,\bar{x},v)=\mu+\mu^{1/2}f_0(x_1,\bar{x},v)\ge 0$ and 
		\begin{align*}
			\sum_{|\alpha|\le 1}\big(\|\widehat{w\partial^\alpha f_0}\|^2_{L^1_{\bar{k}}L^2_{x_1,v}}+\|\widehat{\partial^\alpha E_0}\|^2_{L^1_{\bar{k}}L^2_{x_1}}\big)\le\ve_0,
		\end{align*}
		then the specular reflection boundary problem for VPL systems \eqref{1}, \eqref{Neumann_finite}, \eqref{specular_finite} and \eqref{VPL} admits a unique solution $f(t,x,v)$ on $t\in[0,T_0]$, $x\in\Omega=[-1,1]\times\T^2$, $v\in\R^3$, satisfying estimate 
		\begin{align}
			\label{169}\E_{{T_0}}+\D_{T_0}+\E_{{T_0},w}+\D_{{T_0},w}\lesssim  \sum_{|\alpha|\le 1}\big(\|\widehat{w\partial^\alpha f_0}\|_{L^1_{\bar{k}}L^2_{x_1,v}}+\|\widehat{\partial^\alpha E_0}\|_{L^1_{\bar{k}}L^2_{x_1}}\big),
		\end{align}
		where $\E_{{T_0}}$, $\D_{T_0}$, $\E_{{T_0},w}$ and $\D_{{T_0},w}$ are defined by \eqref{141}, \eqref{142}, \eqref{143} and \eqref{144} respectively.
%		 The constant depends on $T_0$ and is independent of $f$. 
	\end{Thm}
	
	We begin with the following linear inhomogeneous problem on the finite channel:
	\begin{equation}\label{154}\left\{
		\begin{split}
			&\partial_tf_\pm + v\cdot \nabla_xf_\pm  \pm \frac{1}{2}\nabla_x\psi\cdot vf_\pm  \mp\nabla_x\psi\cdot\nabla_vf_\pm \\
			&\qquad\qquad\qquad\pm \nabla_x\phi\cdot v\mu^{1/2} - A_\pm f = \Gamma_{\pm}(g,h)+Kh,\\% \text{ in }(0,\infty)\times[-1,1]\times\T^3\times\R^3.
			&-\Delta_x\phi = \int_{\R^3}(f_+-f_-)\mu^{1/2}\,dv,\\ %\text{ in }(0,\infty)\times[-1,1]\times\T^3.
			&f(0,x,v) = f_0(x,v),\quad E(0,x)=E_0(x),\\ 
			&\widehat{f}(t,-1,\bar{k},v_1,\bar{v})|_{v_1>0} = \widehat{f}(t,-1,\bar{k},-v_1,\bar{v}),\\
			&\widehat{f}(t,1,\bar{k},v_1,\bar{v})|_{v_1<0} = \widehat{f}(t,1,\bar{k},-v_1,\bar{v}),\\
			&\partial_{x_1}\phi = 0,\ \text{ on } x_1=\pm 1, 
		\end{split}\right.
	\end{equation}
	for a given $h=h(t,x,v)$ and $\psi=\psi(t,x)$.

	\begin{Lem}\label{Lem82}
		There exist $\varepsilon_0>0$, $T_0>0$ such if 
		%	\begin{align*}
		%	\wh{f_0},\wh{\nabla_xf_0}\in L^1_{\bar{k}}L^2_{x_1,v},\quad \wh{h},\wh{\nabla_xh}\in L^1_{\bar{k}}L^\infty_{T_0}L^2_{x_1,v}\cap L^1_{\bar{k}}L^2_TL^2_{x_1}L^2_{D},\quad \wh{\nabla_x\psi}\in L^1_{\bar{k}}L^\infty_{T_0}L^2_{x_1},
		%	\end{align*}
		%	and it holds that 
		\begin{multline}\label{156a}
			\sum_{|\alpha|\le 1}\Big\{\|w\widehat{\partial^\alpha f_0}\|_{L^1_{\bar{k}}L^2_{x_1,v}}+\|w\widehat{\partial^\alpha h}\|_{L^1_{\bar{k}}L^2_{T_0}L^2_{x_1}L^2_{D}}\\
			+\|w\widehat{\partial^\alpha h}\|_{L^1_{\bar{k}}L^\infty_{T_0}L^2_{x_1,v}}+\|\widehat{\partial^\alpha\nabla_x\psi}\|_{L^1_{\bar{k}}L^\infty_{T_0}L^2_{x_1}}\Big\}\le \varepsilon_0,
		\end{multline}
		then the initial boundary value problem \eqref{154} admits a unique weak solution $f=f(t,x,v)$ on $[-1,1]\times \T^2\times\R^3$ satisfying 
		\begin{multline}\label{158a}
			\E_{{T_0}}+\D_{T_0}+\E_{{T_0},w}+\D_{{T_0},w}\\\lesssim  \sum_{|\alpha|\le 1}\big(\|\widehat{w\partial^\alpha f_0}\|_{L^1_{\bar{k}}L^2_{x_1,v}}+\|\widehat{\partial^\alpha E_0}\|_{L^1_{\bar{k}}L^2_{x_1}}\big)+T_0^{1/2}\sum_{|\alpha|\le 1}\|\widehat{\partial^\alpha h}\|_{L^1_{\bar{k}}L^\infty_{T_0}L^2_{x_1,v}}, 
		\end{multline}
		where $\E_{{T_0}}$, $\D_{T_0}$, $\E_{{T_0},w}$ and $\D_{{T_0},w}$ are defined by \eqref{141}, \eqref{142}, \eqref{143} and \eqref{144} respectively. 
		The weak solution $f$ is defined by %{\color{red} is this make sense?}
		\begin{align*}
			&
			\int^{T_0}_0(\partial_t\widehat{f_\pm}+v\cdot \widehat{\nabla_xf_\pm} \pm \frac{1}{2}(\nabla_x\psi\cdot v{f_\pm})^\wedge \mp(\nabla_x\psi\cdot\nabla_v{f_\pm})^\wedge
			\pm\widehat{\nabla_x\phi}\cdot v\mu^{1/2},\widehat{g_\pm})_{L^2_{x_1,\bar{k},v}}\,dt\\&\quad - \int^{T_0}_0(\widehat{A_\pm f},\widehat{g_\pm})_{L^2_{x_1,\bar{k},v}}\,dt = \int^{T_0}_0(\widehat{\Gamma_{\pm}(h,f)},\widehat{g_\pm})_{L^2_{x_1,\bar{k},v}}+(\widehat{Kh},\widehat{g_\pm})_{L^2_{x_1,\bar{k},v}}\,dt, 
		\end{align*} for any $g$ belonging to 
		\begin{align*}
			\big\{g=(g_+,g_-)\in L^\infty_{\bar{k}}L^\infty_{T_0} L^2_{x_1,v}:\ &\widehat{g}(t,-1,\bar{k},v_1,\bar{v})|_{v_1>0} = \widehat{g}(t,-1,\bar{k},-v_1,\bar{v}),\\
			&\widehat{g}(t,1,\bar{k},v_1,\bar{v})|_{v_1<0} = \widehat{g}(t,1,\bar{k},-v_1,\bar{v})\big\}.
		\end{align*}
	\end{Lem}
	\begin{proof}
		Let $\eta_v$ and $\eta_x$ be the standard mollifier in $\R^3$ and $[-1,1]\times\T^2$: $\eta_v,\eta_x\in C^\infty_c$, $0\le \eta_v,\eta_x\le 1$, $\int\zeta_vdv=\int\zeta_xdx=1$. For $\varepsilon>0$, let $\eta^\varepsilon_v(v) = \ve^{-3}\zeta_v{(\ve^{-1}v)}$ and $\eta_x^\ve(x)=\ve^{-3}\zeta_x{(\ve^{-1}x)}$. Then we mollify the initial data as $f_0^\ve=f_0*\eta^\ve_v*\eta^\ve_x$, $E^\ve_0=E_0*\eta_x^\ve$. Note that $f_0^\ve$, $E^\ve_0$ are still periodic with respect to $\bar{x}\in\T^2$. 
		Since $\int^{1}_{-1}|\widehat{\eta_x^\ve}|\,dx_1\le\int^{1}_{-1}\int_{\T^2}\eta_x(x)\,d\bar{x}dx_1 = 1$ and $|\widehat{\eta_v^\ve}|\le \int\eta_v\,dv=1$, we have 
		\begin{align*}
			\|\widehat{f_0^\ve}\|_{L^1_{\bar{k}}L^2_{x_1,v}}\le \|(f_0*\eta^\ve_v*\eta^\ve_x)^\wedge\|_{L^1_{\bar{k}}L^2_{x_1,v}}\lesssim \|\eta_v\|_{L^1_{v}}\|\widehat{\eta_x^\ve}\|_{L^1_{x_1}}\|\widehat{f_0}\|_{L^1_{\bar{k}}L^2_{x_1,v}}\le \|\widehat{f_0}\|_{L^1_{\bar{k}}L^2_{x_1,v}}, 
		\end{align*}
		and similarly,
		\begin{align*}
			\|\widehat{\nabla_xf_0^\ve}\|_{L^1_{\bar{k}}L^2_{x_1,v}}&\le \|\widehat{\nabla_x f_0}\|_{L^1_{\bar{k}}L^2_{x_1,v}},\\ 
			\|\widehat{E_0^\ve}\|_{L^1_{\bar{k}}L^2_{x_1}}&\le \|\widehat{ E_0}\|_{L^1_{\bar{k}}L^2_{x_1}},\\
			\|\widehat{\nabla_xE_0^\ve}\|_{L^1_{\bar{k}}L^2_{x_1}}&\le \|\widehat{\nabla_x E_0}\|_{L^1_{\bar{k}}L^2_{x_1}}.
		\end{align*}
		We would like to add the vanishing term $\ve\<v\>^{30-8|\beta|}\partial^{2\alpha}_{2\beta}f$ to the first equation of \eqref{154}. Here we pick $30$ as merely a large constant. But since there's boundary on $x_1$, in order to eliminate the boundary effect, we directly consider the weak form of the solution. For any $g_\pm$, we consider
		\begin{multline}\label{156}
			(\partial_tf_\pm,g_\pm)_{L^2_{x,v}} +\ve\sum_{|\alpha|+|\beta|\le 3}(\<v\>^{30-8|\beta|}\partial^\alpha_\beta f_\pm,\partial^\alpha_\beta g_\pm)_{L^2_{x,v}}+(v\cdot \nabla_xf_\pm,g_\pm)_{L^2_{x,v}}\\   \pm \Big(\frac{1}{2}\nabla_x\psi\cdot vf_\pm ,g_\pm\Big)_{L^2_{x,v}} \mp\big(\nabla_x\psi\cdot\nabla_vf_\pm,g_\pm\big)_{L^2_{x,v}}  \pm \big(\nabla_x\phi\cdot v\mu^{1/2},g_\pm\big)_{L^2_{x,v}}\\ - (A_\pm f,g_\pm)_{L^2_{x,v}} = (\Gamma_{\pm}(h,f),g_\pm)_{L^2_{x,v}}+(Kh,g_\pm)_{L^2_{x,v}}. 
		\end{multline}
		Denote \eqref{156} by 
		\begin{equation*}			(\partial_tf_\pm,g_\pm)_{L^2_{x,v}}+ \mathbf{B}_\pm[f,g]=(Kh,g_\pm)_{L^2_{x,v}}.
		\end{equation*} 
		Then $\mathbf{B}=(\mathbf{B}_+,\mathbf{B}_-)$ is a bilinear operator on $\mathcal{H}\times\mathcal{H}$, where 
		\begin{align*}
			\mathcal{H} = \big\{f=(f_+,f_-)\in  L^2_{x,v}:\ &\<v\>^{15-4|\beta|}\partial^\alpha_\beta f\in L^2_{x,v},\ \forall\, |\alpha|+|\beta|\le 3,\\&\widehat{f}(t,-1,\bar{k},v_1,\bar{v})|_{v_1>0} = \widehat{f}(t,-1,\bar{k},-v_1,\bar{v}),\\
			&\widehat{f}(t,1,\bar{k},v_1,\bar{v})|_{v_1<0} = \widehat{f}(t,1,\bar{k},-v_1,\bar{v})\big\}.
		\end{align*} 
		Note that the terms involving both $\psi$ and $h$ can be controlled as 
		\begin{align*}
			&\quad(\pm \frac{1}{2}\nabla_x\psi\cdot vf_\pm ,g_\pm)_{L^2_{x,v}}+(\mp\nabla_x\psi\cdot\nabla_vf_\pm,g_\pm)_{L^2_{x,v}}+(\Gamma_{\pm}(h,f),g_\pm)_{L^2_{x,v}}\\
			&\lesssim \|\nabla_x\psi\|_{L^\infty_x}\|\<v\>^{1/2}f\|_{L^2_{x,v}}\|\<v\>^{1/2}g\|_{L^2_{x,v}} + \|\nabla_x\psi\|_{L^\infty_x}\|\<v\>^{\gamma/2}\nabla_vf\|_{L^2_{x,v}}\|\<v\>^{-\gamma/2}g\|_{L^2_{x,v}} \\&\qquad\qquad\qquad\qquad\qquad\qquad\qquad\qquad\quad+ \|h\|_{L^\infty_xL^2_v}\|f\|_{L^2_{x}L^2_D}\|g\|_{L^2_{x}L^2_D}\\
			&\lesssim  \|\widehat{\nabla_x\psi}\|_{L^1_{\bar{k}}L^2_{x_1}}\|\<v\>^{1/2}f\|_{L^2_{x,v}}\|\<v\>^{1/2}g\|_{L^2_{x,v}} + \|\widehat{\nabla_x\psi}\|_{L^1_{\bar{k}}L^2_{x_1}}\|\<v\>^{\gamma/2}\nabla_vf\|_{L^2_{x,v}}\|\<v\>^{-\gamma/2}g\|_{L^2_{x,v}} \\&\qquad\qquad\qquad\qquad\qquad\qquad\qquad\qquad\quad + \|\widehat{h}\|_{L^1_{\bar{k}}L^2_{x_1}L^2_v}\|f\|_{L^2_{x}L^2_D}\|g\|_{L^2_{x}L^2_D}.
		\end{align*} 
		Then using \eqref{156a} to control the upper bound of $\|\widehat{\nabla_x\psi}\|_{L^1_{\bar{k}}L^2_{x_1}}$ and $\|\widehat{h}\|_{L^1_{\bar{k}}L^2_{x_1}L^2_v}$, it's direct to obtain from \eqref{36a} that 
		%	\begin{align*}
		%		\int^T_0\mathbf{B}[f,g]\,dt\lesssim \sum_{|\alpha|+|\beta|\le 3}\int^T_0\|\<v\>^{15-4|\beta|}\partial^\alpha_\beta f\|_{L^2_{x,v}}\sum_{|\alpha|+|\beta|\le 3}\|\<v\>^{15-4|\beta|}\partial^\alpha_\beta g\|_{L^2_{x,v}},
		%	\end{align*}
		%and 
		\begin{multline*}
			\int^T_0\sum_\pm\mathbf{B}_\pm[f,f]\,dt\ge \ve\sum_{|\alpha|+|\beta|\le 3}\int^T_0\|\<v\>^{15-4|\beta|}\partial^\alpha_\beta f\|^2_{L^2_{x,v}}\,dt\\ + \lambda\int^T_0\|f\|_{L^2_xL^2_{D}}^2\,dt-CT\sup_{0\le t\le T}\|f\|_{L^2_xL^2_v}^2 ,
			%	\\
			%	&\qquad- C\sum_{|\alpha|+|\beta|\le 1}\int^T_0\|\<v\>^{\gamma+4}\partial^\alpha_\beta f_\pm\|^2_{L^2_{x,v}}\,dt\\
			%	&\ge {\ve}\sum_{|\alpha|+|\beta|\le 3}\int^T_0\|\<v\>^{15-4|\beta|}\partial^\alpha_\beta f_\pm\|^2_{L^2_{x,v}}\,dt - C\int^T_0\|f_\pm\|^2_{L^2_{x,v}}\,dt,
		\end{multline*}
	for some $\lambda>0$. 
		%where the second inequality follows from interpolation. 
		Notice that the boundary term generating from $(v\cdot \nabla_xf_\pm,f_\pm)_{L^2_{x,v}}$ vanishes by using the same argument as \eqref{135}. 
		Also, $f_0^\ve$ is smooth and $K$ is a linear bounded operator $L^2_{x,v}$.
		Then by the standard theory of linear evolution equations on Hilbert space $\mathcal{H}$, there exists $T_0>0$ and unique solution $f^\ve\in\mathcal{H}$ to equation \eqref{156} and \eqref{154}$_2$-\eqref{154}$_6$ on time $[0,T_0]$, which is smooth on $(t,\bar{x},v)$, where $\phi$ is solved by \eqref{154}$_2$ and $\eqref{154}_6$. Then $f^\ve$ solves
		\begin{align}\label{159}\notag
			&%-(f_\pm^{\ve}(0),g_\pm(0))_{L^2_{x,v}}+
			\int^{T_0}_0\Big((\partial_tf_\pm^{\ve},g_\pm)_{L^2_{x,v}} +\ve\sum_{|\alpha|+|\beta|\le 3}(\<v\>^{30-8|\beta|}\partial^\alpha_\beta f_\pm^{\ve},\partial^\alpha_\beta g_\pm)_{L^2_{x,v}}\Big)dt\\&\notag+\int^{T_0}_0(v\cdot \nabla_xf_\pm^{\ve}   \pm \frac{1}{2}\nabla_x\psi\cdot vf_\pm^{\ve} \mp\nabla_x\psi\cdot\nabla_vf_\pm^{\ve}  \pm \nabla_x\phi^\ve\cdot v\mu^{1/2} - A_\pm f^\ve,g_\pm)_{L^2_{x,v}}\,dt\\ & = \int^{T_0}_0(\Gamma_{\pm}(h,f^\ve),g_\pm)_{L^2_{x,v}}+(Kh,g_\pm)_{L^2_{x,v}}\,dt,  
		\end{align}for any $g\in \mathcal{H}$. % with $g(T_0)=0$. 
		Using identity $(\cdot,\cdot)_{L^2_{\bar{x}}}=(\widehat{\cdot},\widehat{\cdot})_{L^2_{\bar{k}}}$ with Fourier transform $\widehat{\cdot}$ over $\bar{x}\in\T^2$, we have 
		\begin{align}\notag\label{157}
			&
			%-(\widehat{f_\pm^{\ve}}(0),\widehat{g_\pm}(0))_{L^2_{x_1,v}}+
			\int^{T_0}_0(\partial_t\widehat{f_\pm^{\ve}},\widehat{g_\pm})_{L^2_{x_1,v}}\,dt +\ve\sum_{|\alpha|+|\beta|\le 3}\int^{T_0}_0(\<v\>^{30-8|\beta|}\widehat{\partial^{\alpha}_{\beta} f_\pm^{\ve}},\widehat{\partial^{\alpha}_\beta g_\pm})_{L^2_{x_1,v}}dt\\
			&\notag+\int^{T_0}_0(v\cdot \widehat{\nabla_xf_\pm^{\ve}} \pm \frac{1}{2}(\nabla_x\psi\cdot v{f_\pm^{\ve}})^\wedge \mp(\nabla_x\psi\cdot\nabla_v{f_\pm^{\ve}})^\wedge
			\pm\widehat{\nabla_x\phi^\ve}\cdot v\mu^{1/2} - \widehat{A_\pm f^\ve},\widehat{g_\pm})_{L^2_{x_1,v}}\,dt\\& = \int^{T_0}_0(\widehat{\Gamma_{\pm}(h,f^\ve)},\widehat{g_\pm})_{L^2_{x_1,v}}+(\widehat{Kh},\widehat{g_\pm})_{L^2_{x_1,v}}\,dt, 
		\end{align}
		and with weight $w$ involved, we have 
		\begin{align}\notag\label{157a}
			%	&(\partial_t\widehat{wf_\pm^{\ve}},\widehat{wg_\pm})_{L^2_{x_1,v}} +\ve\sum_{|\alpha|+|\beta|\le 3}(\<v\>^{30-8|\beta|}(-1)^{|\alpha-\alpha_1|+|\beta|}\widehat{w^2\partial^{2\alpha-\alpha_1}_{2\beta} f_\pm^{\ve}},\widehat{\partial^{\alpha_1} g_\pm})_{L^2_{x_1,v}}\\
			&\int^{T_0}_0(\partial_t\widehat{wf_\pm^{\ve}},\widehat{wg_\pm})_{L^2_{x_1,v}}\,dt +\ve\sum_{|\alpha|+|\beta|\le 3}\int^{T_0}_0(\<v\>^{30-8|\beta|}\widehat{w\partial^{\alpha}_{\beta} f_\pm^{\ve}},\widehat{w\partial^{\alpha}_\beta g_\pm})_{L^2_{x_1,v}}\,dt\\
			&\notag+\ve\int^{T_0}_0\sum_{|\alpha|+|\beta|\le 3}\big(\<v\>^{30-8|\beta|}\widehat{\partial^{\alpha}_{\beta} f_\pm^{\ve}},\sum_{\substack{\beta'<\beta}}(\partial_{\beta-\beta'}w^2{\partial^{\alpha}_{\beta'} g_\pm})^\wedge\big)_{L^2_{x_1,v}}\,dt\\
			&\notag+\int^{T_0}_0(v\cdot \widehat{w\nabla_xf_\pm^{\ve}} \pm \frac{1}{2}(\nabla_x\psi\cdot v{wf_\pm^{\ve}})^\wedge \mp(\nabla_x\psi\cdot\nabla_v{wf_\pm^{\ve}})^\wedge \pm\widehat{w\nabla_x\phi^\ve}\cdot v\mu^{1/2},\widehat{wg_\pm})_{L^2_{x_1,v}}\,dt \\&- \int^{T_0}_0(\widehat{wA_\pm f^\ve},\widehat{wg_\pm})_{L^2_{x_1,v}}\,dt = \int^{T_0}_0({(w\Gamma_{\pm}(h,f^\ve))^\wedge}+\widehat{wKh},\widehat{wg_\pm})_{L^2_{x_1,v}}\,dt, 
		\end{align}for any $g\in \mathcal{H}$, 
		where $\alpha=(\alpha_1,\alpha_2,\alpha_3)$. Note that after Fourier transform on $\T^2$, we take only inner product $L^2_{x_1,v}$.
		% {\color{red}I skip some steps here, is it okay?}
		
		Next we derive the identities on derivative. Let $|\alpha_1|= 1$. We consider the equation 
		\begin{align*}
			\partial_tf^{\alpha_1}_\pm + v\cdot \nabla_xf^{\alpha_1}_\pm  \pm \frac{1}{2}\nabla_x\psi\cdot vf^{\alpha_1}_\pm  \mp\nabla_x\psi\cdot\nabla_vf^{\alpha_1}_\pm \mp E^\alpha\cdot v\mu^{1/2} - A_\pm \partial^{\alpha_1} f\\ = \Gamma_{\pm}(\partial^{\alpha_1} h,f)+\Gamma_{\pm}(h,f^{\alpha_1} )\mp \frac{1}{2}\partial^{\alpha_1} \nabla_x\psi\cdot vf_\pm  \pm\partial^{\alpha_1} \nabla_x\psi\cdot\nabla_vf_\pm + K\partial^{\alpha_1} h,
		\end{align*}
		with initial data $f^{\alpha_1}(0,x,v)=\partial^{\alpha_1} f_0(x,v)$, $E^{\alpha}(0,x)=\partial^{\alpha}E_0(x)$ and \eqref{154}$_3$-\eqref{154}$_6$. Using the same argument we used to derive \eqref{159}, there exists $f^{\ve,\alpha_1}\in\mathcal{H}$ such that 
		\begin{align}\label{162}\notag
			&\int^{T_0}_0(\partial_tf_\pm^{\ve,\alpha_1},g_\pm)_{L^2_{x,v}}\,dt +\ve\sum_{|\alpha|+|\beta|\le 3}\int^{T_0}_0(\<v\>^{30-8|\beta|}\partial^\alpha_\beta f_\pm^{\ve,\alpha_1},\partial^\alpha_\beta g_\pm)_{L^2_{x,v}}\,dt\\&\notag+\int^{T_0}_0(v\cdot \nabla_xf_\pm^{\ve,\alpha_1}   \pm \frac{1}{2}\nabla_x\psi\cdot vf_\pm^{\ve,\alpha_1}\mp\nabla_x\psi\cdot\nabla_vf_\pm^{\ve,\alpha_1}  \mp E^{\ve,\alpha_1}\cdot v\mu^{1/2},g_\pm)_{L^2_{x,v}}\,dt \\&\notag = \int^{T_0}_0( A_\pm f^{\ve,\alpha_1}+\Gamma_{\pm}(\partial^{\alpha_1} h,f^\ve)+\Gamma_{\pm}(h,f^{\ve,\alpha_1} ),g_\pm)_{L^2_{x,v}}\,dt\\&
			+\int^{T_0}_0(\mp \frac{1}{2}\partial^{\alpha_1} \nabla_x\psi\cdot vf^\ve_\pm  \pm\partial^{\alpha_1} \nabla_x\psi\cdot\nabla_vf^\ve_\pm+K\partial^{\alpha_1} h,g_\pm)_{L^2_{x,v}}\,dt,
		\end{align}for $g\in\mathcal{H}$. Then $f^{\ve,\alpha_1}=\partial^{\alpha_1} f^\ve$, $E^{\ve,\alpha_1}=\partial^{\alpha_1} E^\ve$ in the weak sense by using \eqref{159} and \eqref{162}. 
		%{\color{red} is this true? it's similar to CPAM paper}
		Then similar to \eqref{157} and \eqref{157a}, we have 
		\begin{align}\notag\label{163}
			&
			\int^{T_0}_0(\partial_t\widehat{\partial^{\alpha_1} f_\pm^{\ve}},\widehat{g_\pm})_{L^2_{x_1,v}}\,dt +\ve\sum_{|\alpha|+|\beta|\le 3}\int^{T_0}_0(\<v\>^{30-8|\beta|}\widehat{\partial^{\alpha}_{\beta} \partial^{\alpha_1} f_\pm^{\ve}},\widehat{\partial^{\alpha}_\beta g_\pm})_{L^2_{x_1,v}}dt\\
			&\notag+\int^{T_0}_0(v\cdot \widehat{\nabla_x\partial^{\alpha_1} f_\pm^{\ve}} \pm \frac{1}{2}(\nabla_x\psi\cdot v{\partial^{\alpha_1} f_\pm^{\ve}})^\wedge \mp(\nabla_x\psi\cdot\nabla_v{\partial^{\alpha_1} f_\pm^{\ve}})^\wedge,\widehat{g_\pm})_{L^2_{x_1,v}}\,dt\\
			&\notag+\int^{T_0}_0(
			\pm\widehat{\partial^\alpha\nabla_x\phi^\ve}\cdot v\mu^{1/2}-\widehat{A_\pm \partial^{\alpha_1} f^{\ve}},\widehat{g_\pm})_{L^2_{x_1,v}}\,dt
			 =\notag \int^{T_0}_0((\Gamma_{\pm}(\partial^{\alpha_1} h,f^\ve))^\wedge,\widehat{g_\pm})_{L^2_{x_1,v}}\,dt\\
			&\notag+ \int^{T_0}_0((\Gamma_{\pm}(h,f^{\ve,\alpha_1} ))^\wedge
			\mp \frac{1}{2}(\partial^{\alpha_1} \nabla_x\psi\cdot vf^\ve_\pm)^\wedge ,\widehat{g_\pm})_{L^2_{x_1,v}}\,dt\\& +\int^{T_0}_0(\pm(\partial^{\alpha_1} \nabla_x\psi\cdot\nabla_vf^\ve_\pm)^\wedge,\widehat{g_\pm})_{L^2_{x_1,v}}\,dt+\int^{T_0}_0(K\widehat{\partial^{\alpha_1} h},\widehat{g_\pm})_{L^2_{x_1,v}}\,dt, 
		\end{align}
		and the weighted form is 
		\begin{align}\notag\label{164}
			&\int^{T_0}_0(\partial_t\widehat{w\partial^{\alpha_1} f_\pm^{\ve}},\widehat{wg_\pm})_{L^2_{x_1,v}}\,dt +\ve\sum_{|\alpha|+|\beta|\le 3}\int^{T_0}_0(\<v\>^{30-8|\beta|}\widehat{w\partial^{\alpha}_{\beta} \partial^{\alpha_1} f_\pm^{\ve}},\widehat{w\partial^{\alpha}_\beta g_\pm})_{L^2_{x_1,v}}\,dt\\
			&\notag+\ve\int^{T_0}_0\sum_{|\alpha|+|\beta|\le 3}\big(\<v\>^{30-8|\beta|}\widehat{\partial^{\alpha}_{\beta} \partial^{\alpha_1} f_\pm^{\ve}},\sum_{\substack{\beta'<\beta}}(\partial_{\beta-\beta'}w^2{\partial^{\alpha}_{\beta'} g_\pm})^\wedge\big)_{L^2_{x_1,v}}\,dt\\
			&\notag+\int^{T_0}_0(v\cdot ({w\nabla_x\partial^{\alpha_1} f_\pm^{\ve}})^\wedge \pm \frac{1}{2}(\nabla_x\psi\cdot v{w\partial^{\alpha_1} f_\pm^{\ve}})^\wedge \mp(\nabla_x\psi\cdot\nabla_v{w\partial^{\alpha_1} f_\pm^{\ve}})^\wedge,\widehat{wg_\pm})_{L^2_{x_1,v}}\,dt\\
			&\notag+\int^{T_0}_0(\pm\widehat{w\partial^{\alpha_1} \nabla_x\phi^\ve}\cdot v\mu^{1/2}-w\widehat{A_\pm \partial^{\alpha_1} f^{\ve}},\widehat{wg_\pm})_{L^2_{x_1,v}}\,dt \\
			%	&- \int^{T_0}_0(\widehat{wA_\pm \partial^{\alpha_1} f^{\ve}},\widehat{wg_\pm})_{L^2_{x_1,v}}\,dt 
			&=\notag \int^{T_0}_0((\Gamma_{\pm}(\partial^{\alpha_1} h,f^\ve))^\wedge,\widehat{g_\pm})_{L^2_{x_1,v}}\,dt+\int^{T_0}_0((\Gamma_{\pm}(h,\partial^{\alpha_1} f^{\ve} ))^\wedge,\widehat{g_\pm})_{L^2_{x_1,v}}\,dt
			\\&\notag+\int^{T_0}_0(\mp \frac{1}{2}(w\partial^{\alpha_1} \nabla_x\psi\cdot vf^\ve_\pm)^\wedge  \pm(w\partial^{\alpha_1} \nabla_x\psi\cdot\nabla_vf^\ve_\pm)^\wedge,\widehat{wg_\pm})_{L^2_{x_1,v}}\,dt\\&+\int^{T_0}_0(\widehat{wK\partial^{\alpha_1} h},\widehat{wg_\pm})_{L^2_{x_1,v}}\,dt. 
		\end{align}
		Following the arguments from \eqref{146} to \eqref{147}, choosing $g = f$ in \eqref{157}, \eqref{157a}, \eqref{163} and \eqref{164}, we can obtain the similar energy estimate to \eqref{147}. It suffices to deal with the third terms in \eqref{157a} and \eqref{164}. Noticing $|\partial^\alpha_\beta w^2|\le \<v\>^{2|\beta|}w^2$, we have  
		\begin{align*}
			&\ve\int_{\Z^2}\Big(\int^{T_0}_0\Big|\sum_{|\alpha|+|\beta|\le 3}\big(\<v\>^{30-8|\beta|}\widehat{\partial^{\alpha}_{\beta}\pa f_\pm},\sum_{\substack{\beta'<\beta}}(\partial_{\beta-\beta'}w^2{\partial^{\alpha}_{\beta'}\pa f_\pm})^\wedge\big)_{L^2_{x_1,v}}\Big|\,dt\Big)^{1/2}\,d\Sigma(\bar{k})\\
			&= \ve\sum_{|\alpha|+|\beta|\le 3}\|\<v\>^{15-4|\beta|}w\widehat{\partial^{\alpha}_{\beta}\pa f_\pm}\|_{L^1_{\bar{k}}L^2_{T_0}L^2_{x_1,v}}\sum_{|\alpha'|+|\beta'|\le 2}\|\<v\>^{15-4|\beta'|-2}w\widehat{\partial^{\alpha'}_{\beta'}\pa f_\pm}\|_{L^1_{\bar{k}}L^2_{T_0}L^2_{x_1,v}}\\
			&\lesssim \frac{\ve}{2}\sum_{|\alpha|+|\beta|\le 3}\|\<v\>^{15-4|\beta|}w\widehat{\partial^{\alpha}_{\beta}\pa f_\pm}\|_{L^1_{\bar{k}}L^2_{T_0}L^2_{x_1,v}}^2+\ve C\sum_{|\alpha|+|\beta|\le 2}\|\<v\>^{15-4|\beta|}\widehat{\partial^{\alpha}_{\beta}\pa f_\pm}\|^2_{L^1_{\bar{k}}L^2_{T_0}L^2_{x_1,v}},
		\end{align*}
		for some $C>0$. 
		The first and second right-hand terms will be absorbed by the second left-hand terms of \eqref{164} and \eqref{157} respectively. 
		Therefore, we take the linear combination $\eqref{157}+\kappa\times\eqref{157a}+\sum_{|\alpha_1|\le 1}\big(\eqref{163}+\kappa\times\eqref{164}\big)$ with $g=f$ and $\kappa$ suitably small. Following the arguments from \eqref{146} to \eqref{147} and taking summation over $\pm$, square root and summation on $\bar{k}\in \Z^2$ of the resultant estimate, we have the following estimate (Note that we use \eqref{36a} to estimate $A_\pm$ while $K$ is linear bounded on $L^2_{x,v}$):
		\begin{align*}
			&\quad\,\E_{{T_0},w}(f^\ve)+\D_{{T_0},w}(f^\ve)+\E_{{T_0}}(f^\ve)+\D_{T_0}(f^\ve)\\&\qquad+\ve\sum_{|\alpha_1|\le 1}\sum_{|\alpha|+|\beta|\le 3}\int_{\Z^2}\Big(\int^{T_0}_0\|\<v\>^{15-4|\beta|}({w\partial^{\alpha}_{\beta} \partial^{\alpha_1} f_\pm^{\ve}})^\wedge\|^2_{L^2_{x_1,v}}\,dt\Big)^{1/2}\,d\Sigma(\bar{k})\\
			&\lesssim \sum_{|\alpha_1|\le 1}\Big(\|\widehat{\partial^{\alpha_1}\nabla_x\psi}\|_{L^1_{\bar{k}}L^\infty_{T_0}L^2_{x_1}}+\|\widehat{\partial^{\alpha_1} h}\|_{L^1_{\bar{k}}L^2_{T_0}L^2_{x_1}L^2_{D}}+\|\widehat{\partial^{\alpha_1} h}\|_{L^1_{\bar{k}}L^\infty_{T_0}L^2_{x_1,v}}\Big)\\&\qquad\qquad\qquad\qquad\qquad\times\Big(\E_{{T_0},w}(f^\ve)+\D_{{T_0},w}(f^\ve)+\E_{{T_0}}(f^\ve)+\D_{T_0}(f^\ve)\Big)\\
			&\qquad +\sum_{|\alpha_1|\le 1}\big(\|\widehat{w\partial^{\alpha_1} f_0}\|^2_{L^1_{\bar{k}}L^2_{x_1,v}}+\|\widehat{\partial^{\alpha_1} E_0}\|^2_{L^1_{\bar{k}}L^2_{x_1}}\big)\\&\qquad+T_0^{1/2}\sum_{|\alpha_1|\le 1}\|\widehat{\partial^{\alpha_1} f}\|_{L^1_{\bar{k}}L^\infty_{T_0}L^2_{x_1,v}}+T_0^{1/2}\sum_{|\alpha_1|\le 1}\|\widehat{\partial^{\alpha_1} h}\|_{L^1_{\bar{k}}L^\infty_{T_0}L^2_{x_1,v}},
		\end{align*}where we used definition \eqref{141}, \eqref{142}, \eqref{143} and \eqref{144} and write $f^\ve$ to illustrate the dependence on $f^\ve$. Choosing $\ve_0$ in \eqref{156a} and $T_0>0$ sufficiently small, we have 
		\begin{align}\label{165}\notag
			&\quad\,\E_{{T_0},w}(f^\ve)+\D_{{T_0},w}(f^\ve)+\E_{{T_0}}(f^\ve)+\D_{T_0}(f^\ve)\\
			\notag&\quad+\ve\sum_{|\alpha|+|\beta|\le 3}\int_{\Z^2}\Big(\int^{T_0}_0\|\<v\>^{15-4|\beta|}({w\partial^{\alpha}_{\beta} \partial^{\alpha_1} f_\pm^{\ve}})^\wedge\|^2_{L^2_{x_1,v}}\,dt\Big)^{1/2}\,d\Sigma(\bar{k})
			\\&\le \sum_{|\alpha|\le 1}\big(\|\widehat{w\partial^{\alpha_1} f_0}\|^2_{L^1_{\bar{k}}L^2_{x_1,v}}+\|\widehat{\partial^{\alpha_1} E_0}\|^2_{L^1_{\bar{k}}L^2_{x_1}}\big)+T_0^{1/2}\sum_{|\alpha|\le 1}\|\widehat{\partial^{\alpha_1} h}\|_{L^1_{\bar{k}}L^\infty_{T_0}L^2_{x_1,v}}. 
		\end{align}
		Thus, $f^\ve$ is uniformly bounded with respect to norms:
		\begin{multline*}
			\sum_{|\alpha_1|\le 1}\big(\|\widehat{w\partial^{\alpha_1} f}\|^2_{L^1_{\bar{k}}L^\infty_TL^2_{x_1,v}}+\sqrt{qN}\|{\<v\>^{\frac{\vt}{2}}}{(1+t)^{-\frac{N+1}{2}}}w\widehat{\partial^{\alpha_1}{h}}\|_{L^1_kL^2_TL^2_v}+ \|\widehat{\partial^{\alpha_1} f}\|^2_{L^1_{\bar{k}}L^2_TL^2_{x_1}L^2_{D,w}}\\
			+\|\widehat{\partial^{\alpha_1} f}\|_{L^1_{\bar{k}}L^\infty_TL^2_{x_1,v}} 
			+\|\wh{\partial^{\alpha_1} E}\|_{L^1_{\bar{k}}L^\infty_TL^2_{x_1}}+\|\widehat{\partial^{\alpha_1} f}\|_{L^1_{\bar{k}}L^2_TL^2_{x_1}L^2_{D}}
			+\|\widehat{\partial^{\alpha_1} E}\|_{L^1_{\bar{k}}L^2_TL^2_{x_1}}\big).
		\end{multline*}
		Denote the weak limit of $\{f^\ve\}$ as $\ve\to0$ to be $f$. Taking limit $\ve\to 0$ in \eqref{157}, then $f$ solves 
		\begin{align*}
			&
			\int^{T_0}_0(\partial_t\widehat{f_\pm}+v\cdot \widehat{\nabla_xf_\pm} \pm \frac{1}{2}(\nabla_x\psi\cdot v{f_\pm})^\wedge \mp(\nabla_x\psi\cdot\nabla_v{f_\pm})^\wedge
			\pm\widehat{\nabla_x\phi}\cdot v\mu^{1/2},\widehat{g_\pm})_{L^2_{x_1,\bar{k},v}}\,dt\\&\quad - \int^{T_0}_0(\widehat{A_\pm f},\widehat{g_\pm})_{L^2_{x_1,\bar{k},v}}\,dt = \int^{T_0}_0(\widehat{\Gamma_{\pm}(h,f)},\widehat{g_\pm})_{L^2_{x_1,\bar{k},v}}+(\widehat{Kh},\widehat{g_\pm})_{L^2_{x_1,\bar{k},v}}\,dt, 
		\end{align*}with initial data $f(0)=f_0$, for any $g$ belonging to 
		\begin{align*}
			\big\{g=(g_+,g_-)\in L^\infty_{\bar{k}}L^\infty_{T_0} L^2_{x_1,v}:\,&\widehat{g}(t,-1,\bar{k},v_1,\bar{v})|_{v_1>0} = \widehat{g}(t,-1,\bar{k},-v_1,\bar{v}),\\
			&\widehat{g}(t,1,\bar{k},v_1,\bar{v})|_{v_1<0} = \widehat{g}(t,1,\bar{k},-v_1,\bar{v})\big\}.
		\end{align*}
		Taking limit $\ve\to 0$ in \eqref{165}, we obtain \eqref{158a}. This completes the proof of Lemma \ref{Lem82}. 
		%	\begin{align*}\notag
		%		&\quad\,\E_{{T_0},w}(f)+\D_{{T_0},w}(f)+\E_{{T_0}}(f)+\D_{T_0}(f)\\&\le \sum_{|\alpha|\le 1}\big(\|\widehat{w\partial^{\alpha_1} f_0}\|^2_{L^1_{\bar{k}}L^2_{x_1,v}}+\|\widehat{\partial^{\alpha_1} E_0}\|^2_{L^1_{\bar{k}}L^2_{x_1}}\big)+T_0^{1/2}\sum_{|\alpha|\le 1}\|\widehat{\partial^{\alpha_1} h}\|_{L^1_{\bar{k}}L^\infty_{T_0}L^2_{x_1,v}}. 
		%	\end{align*}
		\qe\end{proof}
	
	\begin{proof}[Proof of Theorem \ref{localsolution}]Write $(f_0^\ve,E_0^\ve)$ to be the mollification of $(f_0,E_0)$.
		We now construct the approximation solution sequence which is denoted by \begin{align*}
			\{(f^n(t,x,v),\phi^n(t,x))\}^\infty_{n=0}
		\end{align*} using the following iterative scheme:
		\begin{equation*}\left\{
			\begin{split}
				&\partial_tf^{n+1}_\pm + v\cdot \nabla_xf^{n+1}_\pm  \pm \frac{1}{2}\nabla_x\phi^n\cdot vf^{n+1}_\pm  \mp\nabla_x\phi^n\cdot\nabla_vf^{n+1}_\pm \\&\qquad\qquad\pm \nabla_x\phi^{n+1}\cdot v\mu^{1/2} - A_\pm f^{n+1} = \Gamma_{\pm}(f^n,f^{n+1})+Kf^n,\\% \text{ in }(0,\infty)\times[-1,1]\times\T^3\times\R^3.
				&-\Delta_x\phi^{n+1} = \int^{1}_{-1}\int_{\T^2}(f^{n+1}_+-f^{n+1}_-)\mu^{1/2}\,dx,\\ %\text{ in }(0,\infty)\times[-1,1]\times\T^3.
				&f^{n+1}(0,x,v) = f^{\frac{1}{n+1}}_0(x,v),\quad E^{n+1}(0,x)=E^{\frac{1}{n+1}}_0(x),\\ 
				&\widehat{f^{n+1}}(t,-1,\bar{k},v_1,\bar{v})|_{v_1>0} = \widehat{f^{n+1}}(t,-1,\bar{k},-v_1,\bar{v}),\\
				&\widehat{f^{n+1}}(t,1,\bar{k},v_1,\bar{v})|_{v_1<0} = \widehat{f^{n+1}}(t,1,\bar{k},-v_1,\bar{v}),\\
				&\partial_{x_1}\phi^{n+1} = 0,\ \text{ on } x_1=\pm 1, 
			\end{split}\right.
		\end{equation*}
		for $n=0,1,2,\cdots$, where we set $f^0(t,x,v)=f_0(x,v)$. 
		With Lemma \ref{Lem82}, it is a standard procedure to apply the induction argument to show that there exists $\ve_0>0$ and $T_0>0$ such that if 
		\begin{align*}
			\sum_{|\alpha|\le 1}\big(\|\widehat{w\partial^\alpha f_0}\|^2_{L^1_{\bar{k}}L^2_{x_1,v}}+\|\widehat{\partial^\alpha E_0}\|^2_{L^1_{\bar{k}}L^2_{x_1}}\big)\le \ve_0,
		\end{align*}
		then the approximate solution sequence $\{f^n\}$ is well-defined in the sense of the following norms are finite:
		\begin{multline}\label{168}
			\sum_{|\alpha|\le 1}\big(\|\widehat{w\partial^{\alpha} f}\|^2_{L^1_{\bar{k}}L^\infty_TL^2_{x_1,v}}+\sqrt{qN}\|{\<v\>^{\frac{\vt}{2}}}{(1+t)^{-\frac{N+1}{2}}}w\widehat{\partial^{\alpha}{h}}\|_{L^1_kL^2_TL^2_v}+ \|\widehat{\partial^{\alpha} f}\|^2_{L^1_{\bar{k}}L^2_TL^2_{x_1}L^2_{D,w}}\\
			+\|\widehat{\partial^{\alpha} f}\|_{L^1_{\bar{k}}L^\infty_TL^2_{x_1,v}} 
			+\|\wh{\partial^{\alpha} E}\|_{L^1_{\bar{k}}L^\infty_TL^2_{x_1}}+\|\widehat{\partial^{\alpha} f}\|_{L^1_{\bar{k}}L^2_TL^2_{x_1}L^2_{D}}
			+\|\widehat{\partial^{\alpha} E}\|_{L^1_{\bar{k}}L^2_TL^2_{x_1}}\big).
		\end{multline}
		Notice that $f^{n+1}-f^n$ solves 
		\begin{multline}\label{715}
			\partial_t(f^{n+1}_\pm-f^n_\pm) + v\cdot \nabla_x(f^{n+1}_\pm-f^n_\pm)  \pm \frac{1}{2}\nabla_x\phi^n\cdot v(f^{n+1}_\pm -f^n_\pm) \pm \frac{1}{2}(\nabla_x\phi^n-\nabla_x\phi^{n-1})\cdot vf^n_\pm\\
			 \mp\nabla_x\phi^n\cdot\nabla_v(f^{n+1}_\pm -f^n_\pm)
			\mp  (\nabla_x\phi^n-\nabla_x\phi^{n-1})\cdot\nabla_vf^n_\pm
			\pm (\nabla_x\phi^{n+1}-\nabla_x\phi^n)\cdot v\mu^{1/2}\\
			- A_\pm (f^{n+1}-f^n) = \Gamma_{\pm}(f^n,f^{n+1}-f^n)+\Gamma_{\pm}(f^n-f^{n-1},f^n)+K(f^n-f^{n-1}),
		\end{multline}
		for $n=1,2,3,\cdots$. Using the method for deriving \eqref{158a}, we know that $(f^{n+1}-f^n,E^{n+1}-E^n)$ is Cauchy sequence with respect to norms in \eqref{168}%{\color{red} is this ture? please check. it's hard potential, so I think it's correct. Do you think we need more explanations?}. 
		Then the limit function $f(t,x,v)$ is indeed a solution to \eqref{1}, \eqref{specular_finite} and \eqref{Neumann_finite} satisfying estimate \eqref{169}. 
		For the positivity, we can use the argument from \cite[Lemma 12, page 800]{Guo2012}; the details are omitted for brevity. 
		If $(g,\psi)$ is another solution to \eqref{1}, \eqref{specular_finite} and \eqref{Neumann_finite} satisfying \eqref{169}, then similar to \eqref{715}, $f-g$ satisfies 
		\begin{align*}
			&\partial_t(f_\pm-g_\pm) + v\cdot \nabla_x(f_\pm-g_\pm)  \pm \frac{1}{2}\nabla_x\psi\cdot v(f_\pm -g_\pm) \pm \frac{1}{2}(\nabla_x\phi-\nabla_x\psi)\cdot vg_\pm\\
			&\quad \mp\nabla_x\psi\cdot\nabla_v(f_\pm -g_\pm)
			\mp  (\nabla_x\phi-\nabla_x\psi)\cdot\nabla_vg_\pm
			\pm (\nabla_x\phi-\nabla_x\psi)\cdot v\mu^{1/2}\\
			&\quad- A_\pm (f-g) = \Gamma_{\pm}(g,f-g)+\Gamma_{\pm}(f-g,g)+K(f-g),
		\end{align*}
	with zero initial data. 
	Applying the similar calculations for deriving \eqref{158a} and noticing the zero initial data, we deduce that $f=g$, by choosing $T_0$ sufficiently small. 
		The proof of Theorem \ref{localsolution} is complete. 	
		\qe\end{proof}

	\medskip
	\noindent {\bf Acknowledgements.} 
	Dingqun Deng was supported by a Direct Grant from BIMSA and YMSC. 
	The research of Renjun Duan was partially supported by the NSFC/RGC Joint Research Scheme (N\_CUHK409/19) from RGC in Hong Kong and a Direct Grant from CUHK.

\end{document}